\documentclass[11pt]{article}
\usepackage{fullpage}
\usepackage{amsfonts}
\usepackage{bbding}
\usepackage{graphicx,float}
\usepackage{amsfonts,amssymb,amsmath,amsthm,amscd}
\usepackage{mathrsfs}
\usepackage{xcolor}
\usepackage{empheq}
\usepackage{adforn}
\usepackage{cancel}
\usepackage{xr-hyper}
\usepackage{xr}
\usepackage{mdframed}
\usepackage{mathdots}
\usepackage{enumerate}
\usepackage{lmodern}
\usepackage{mdframed}
\usepackage{stmaryrd}
\usepackage{blindtext}
\usepackage[all, knot]{xy}
\usepackage{leftidx}
\usepackage{accents}
\usepackage{appendix}
\usepackage{cases}
\usepackage{bbm}
\usepackage{subfigure}

\definecolor{slightblue}{rgb}{.8, .8, 1}
\definecolor{hair}{RGB}{100,225,190}
\definecolor{ruby}{RGB}{220,50,120}
\definecolor{grass}{RGB}{150,220,110}
\definecolor{ceruleanblue}{rgb}{0.16, 0.32, 0.75}
\definecolor{deepcarmine}{rgb}{0.66, 0.13, 0.24}
\definecolor{otterbrown}{rgb}{0.4, 0.26, 0.13}
\definecolor{sapphire}{rgb}{0.03, 0.15, 0.4}

\usepackage[colorlinks=true,  linkcolor=otterbrown, citecolor=sapphire,  urlcolor=hair!80!black]{hyperref}

\usepackage[T1]{fontenc}

\newtheorem{theorem}{Theorem}[section] \newtheorem{lemma}[theorem]{Lemma}
\newtheorem{proposition}[theorem]{Proposition} \newtheorem{corollary}[theorem]{Corollary}

\theoremstyle{definition} 
\newtheorem{definition}[theorem]{Definition}

\newtheorem{remark}[theorem]{Remark} \numberwithin{equation}{section}
\numberwithin{figure}{section}

\newcommand{\Cb}{\mathbb{C}}

\newcommand{\Eb}{\mathbb{E}}

\newcommand{\Hb}{\mathbb{H}}

\newcommand{\Nb}{\mathbb{N}}

\newcommand{\Pb}{\mathbb{P}}

\newcommand{\Rb}{\mathbb{R}}

\newcommand{\Tb}{\mathbb{T}}

\newcommand{\Zb}{\mathbb{Z}}

\newcommand{\Pf}{\mathbf{P}}

\newcommand{\Lfr}{\mathfrak{L}}

\newcommand{\Ac}{\mathcal{A}}
\newcommand{\Bc}{\mathcal{B}}
\newcommand{\Cc}{\mathcal{C}}
\newcommand{\Dc}{\mathcal{D}}
\newcommand{\Ec}{\mathcal{E}}
\newcommand{\Fc}{\mathcal{F}}
\newcommand{\Gc}{\mathcal{G}}

\newcommand{\Lc}{\mathcal{L}}

\newcommand{\Nc}{\mathcal{N}}

\newcommand{\Pc}{\mathcal{P}}

\newcommand{\Tc}{\mathcal{T}}
\newcommand{\Uc}{\mathcal{U}}
\newcommand{\Vc}{\mathcal{V}}

\newcommand{\Zc}{\mathcal{Z}}

\usepackage[section]{placeins}
\usepackage{wrapfig}
\usepackage{lpic}

\newcommand{\wt}{\widetilde}

\DeclareFontFamily{OMX}{yhex}{}
\DeclareFontShape{OMX}{yhex}{m}{n}{<->yhcmex10}{}
\DeclareSymbolFont{yhlargesymbols}{OMX}{yhex}{m}{n}
\DeclareMathAccent{\wideparen}{\mathord}{yhlargesymbols}{"F3}

\makeatletter
\newcommand*\rel@kern[1]{\kern#1\dimexpr\macc@kerna}
\newcommand*\wb[1]{%
	\begingroup
	\def\mathaccent##1##2{%
		\rel@kern{0.8}%
		\overline{\rel@kern{-0.8}\macc@nucleus\rel@kern{0.2}}%
		\rel@kern{-0.2}%
	}%
	\macc@depth\@ne
	\let\math@bgroup\@empty \let\math@egroup\macc@set@skewchar
	\mathsurround\z@ \frozen@everymath{\mathgroup\macc@group\relax}%
	\macc@set@skewchar\relax
	\let\mathaccentV\macc@nested@a
	\macc@nested@a\relax111{#1}%
	\endgroup
}
\makeatother

\usepackage{soul}

\def \eps {\varepsilon}

\newcommand{\ind}{\mathbf{1}}
\newcommand{\loc}{\mathrm{loc}}

\newcommand{\lemclustersize}{3.2} 
\newcommand{\lemcrossingloop}{3.4} 
\newcommand{\proplocality}{5.3} 
\newcommand{\proplocalityin}{5.10} 
\newcommand{\proplocalityb}{5.12} 
\newcommand{\lemquasimult}{6.1} 
\newcommand{\proparmexp}{6.2} 
\newcommand{\propquasin}{6.3} 
\newcommand{\proparmexpbdy}{6.4} 
\newcommand{\seclocality}{5} 
\newcommand{\secquasimult}{6} 
\newcommand{\lemloops}{5.7} 

\newcommand{\eqrec}{5.10} 
\newcommand{\figquasi}{6.1} 
\newcommand{\eqquasidec}{6.6} 
\newcommand{\eqgamma}{6.9} 

\title{Percolation of discrete GFF in dimension two\\ II. Connectivity properties of two-sided level sets}
\author{Yifan Gao\thanks{Westlake University} \and Pierre Nolin\thanks{City University of Hong Kong} \and Wei Qian\thanks{The University of Hong Kong}}
\date{}
\begin{document}

\maketitle
	
\begin{abstract}
We study percolation of two-sided level sets for the discrete Gaussian free field (DGFF) in dimension two. 
For a DGFF $\varphi$ defined in a box with side length $N$, we show that for $C$ large enough, there exist low crossings in the set of vertices $z$ where $|\varphi(z)|\le C \sqrt{\log \log N}$, with probability tending to $1$ as $N \to \infty$, while the average and the maximum of $\varphi$ are of order $\sqrt{\log N}$ and $\log N$, respectively.
As a consequence, we also obtain connectivity properties of the set of thick points of a random walk.

We rely on an isomorphism between the DGFF and the random walk loop soup with critical intensity $\alpha=1/2$, and further extend our study to the occupation field of the loop soup for all subcritical intensities $\alpha\in(0,1/2)$. For the loop soup in a box with side length $N$, 
 we show that for some constant $\lambda$ large enough, with probability tending to $1$ polynomially fast in $N$, there exists a low path connecting the left and right sides of the box, whose vertices all have a value smaller than $\lambda$, even though the average occupation time is of order $\log N$. In addition, no such path exists for $\lambda$ small. Our results thus uncover a \emph{non-trivial phase-transition} for this highly-dependent percolation model. 

For both the DGFF and the occupation field of the random walk loop soup (with intensity $\alpha\in(0,1/2)$), we further show that such low crossings can be found in the ``carpet'' of the loop soup -- the set of vertices which are not in the interior of any outermost cluster of the loop soup. We also study crossings in subdomains of the box, where boundary effects do not come into play, and we obtain an exponential decay property. Finally, we explain how to extend our results to the metric graph.

This work is the second part of a series of two papers. 
It relies heavily on tools and techniques developed for the random walk loop soup in the first part, especially surgery arguments on loops, which were made possible by a separation result for random walks in a loop soup. This allowed us, in that companion paper, to derive several useful properties such as quasi-multiplicativity, and obtain a precise upper bound for the probability that two large connected components of loops ``almost touch'', which is instrumental here.
\bigskip

\end{abstract}
	
\tableofcontents
	
\section{Introduction} \label{sec:intro}
Our original motivation for this work lies in connectivity properties of the discrete Gaussian free field (DGFF) in dimension $2$. We want to study \emph{two-sided level sets} (TSLS), where the DGFF remains below a certain level, in absolute value, on which very little is known. In our study, we start by invoking a connection between the DGFF and the occupation field of the random walk loop soup (RWLS) with critical intensity $\alpha=1/2$. It is then natural to extend our analysis to the occupation fields of all RWLS with (subcritical and critical) intensities $\alpha \in (0,1/2]$.

Section~\ref{subsec:main_gff} is devoted to the results on the discrete Gaussian free field, as well as consequences on the thick points of the random walk. Section~\ref{subsec:rwls} is devoted to the results on random walk loop soups. In Section~\ref{subsec:geometry}, we describe the additional geometric properties of the percolating paths. In Section~\ref{subsec:strategy}, we present the general strategy of our proof. In Section~\ref{sec:open_questions}, we mention some natural open questions and extensions. Finally in Section~\ref{subsec:org}, we explain the structure of the paper.

\subsection{Discrete Gaussian free field}\label{subsec:main_gff}

Level set percolation of the DGFF has received much attention in the past, as a prominent example of percolation in strongly correlated systems, see e.g.\ \cite{LS1986,BLM1987,MR2080601,marinov2007percolation}.
More recently, impressive progresses have been made in dimension $d\ge 3$, since the introduction of the random interlacement by Sznitman \cite{MR2680403}, which is another celebrated model where long-range dependence occurs. Due to its close relation with the DGFF \cite{MR3050507,MR2892408,MR2932978}, 
the methods developed for the random interlacement, see e.g.\ \cite{MR2744881,MR2891880}, has led to the proof of a non-trivial phase transition for the percolation of the \emph{one-sided level set}, i.e.\ the set $\{ z\in\Zb^d: \phi(z)\ge h \}$ for the DGFF $\phi$ defined on $\Zb^d$ \cite{MR3053773}. 
On the other hand, for the \emph{two-sided level set} $\{ z\in\Zb^d: |\phi(z)|\le \lambda \}$,  a non-trivial phase transition for percolation in the TSLS (resp.\ the complementary set of the TSLS) has also been proved in  \cite{DPR2018a} (resp.\ \cite{Ro2014}).
We are not able to be exhaustive due to the vast literature, but refer to e.g.\ \cite{MR3502602,DPR2018a, Sz2019,MR4112719,CN2020,DPR2022,MR4474499,DPR2018b, MR4568695, DPR2023,drewitz2023arm,CD2023,drewitz2024critical,cai2024one} for many other interesting and important progresses related to this subject. 

An essential feature in $d\ge 3$ is the fast-decaying correlations of the field. 
The DGFF in $d=2$ is more delicate, due to the $\log$-order correlations. In spite of this, exploiting the domain Markov property of the DGFF, \emph{one-sided level sets} in $d=2$ have been well understood. In particular, the crossing probability of a macroscopic annulus in the bulk is non-vanishing for any level $h\in \Rb$, suggesting the absence of a phase transition, see \cite{MR3800790,MR4112719} -- this is completely different from the behavior in higher dimensions. 
On the other hand, the one-sided level sets can be described by the level-lines of the DGFF, which are known to converge to SLE$_4$-type curves  \cite{MR3101840}, coupled with the continuum GFF. The well-developed theory of SLE$_4$ and GFF then implies connectivity properties of the one-sided level sets of the DGFF,  see e.g.\ \cite{MR4452651}.

Let us also mention a particular type of level set for the two-dimensional DGFF in a box of side length $2N$, above a level which depends on $N$. The leading-order asymptotic of the maximum of the GFF (in absolute value) has been shown in \cite{BDG2001} to be  $2\sqrt{g}\log N$ for $g:=2/\pi$, see \eqref{eq:max}.
The following one-sided level set is known as \emph{intermediate level sets}
\[
\{ z\in B_N: \varphi_N(z)\ge \left(2\sqrt{g}a\right)\log N \}, \quad \text{ with } a \in (0,1).
\] 
It was established in \cite{Da2006} that for any $a \in (0,1)$, the corresponding intermediate level set contains $N^{2(1-a^2)+o(1)}$ vertices, exhibiting an interesting fractal structure. Later, it was proved in \cite{BL2019} that if one encodes an intermediate level set as a point measure, then it converges, after a suitable rescaling, to a Liouville quantum gravity (LQG) measure.
\medbreak

In sharp contrast to the extensive literature on level-set percolation of the DGFF that we mentioned above, little is known about the percolative property of the \emph{two-sided level sets} (TSLS), for the DGFF $\phi$ in $d=2$.
The chemical distance within the TSLS was studied in \cite{MR4419197}, but the results there are based on an assumption about the percolative property of the TSLS, whose validity is still unknown (see Remark~\ref{rmk:yg}). 
Before the present paper, the best conclusion we can draw seems to be the following: the aforementioned results about intermediate level sets (combined with the fact that $\phi$ has the same law as $-\phi$) could imply that the TSLS $\{z: |\phi(z)|\le \lambda\}$ for any $\lambda>\sqrt{2g} \log N$ has macroscopic crossings.
Prior to this paper, percolation of the TSLS of the DGFF in $d=2$ had remained out of reach, in particular because the techniques that proved so powerful in the earlier settings are inapplicable here. In particular, it is highly unclear whether the theory of GFF and its level lines in the continuum is relevant for our purpose (see Remark~\ref{rmk:gff1}), and we develop a completely different approach, based on the geometric representation provided by the critical random walk loop soup.

Finally, we mention a recent work \cite{AGS2022} which studies percolation of the level set of (the norm of) the \emph{$n$-vector-valued} DGFF in $d=2$, where many interesting results have been obtained for $n\ge 2$ (see Remark~\ref{rmk:supercritical} for more details). However, as pointed out in \cite{AGS2022}, a geometric understanding of the TSLS for $n=1$ is still lacking. 
An important motivation in \cite{AGS2022} lies in consequences for the spin $O(n+1)$ model, in particular the spin $O(3)$ model (the classical Heisenberg model) for $n=2$. Hence, results for $n=1$ (which is the focus of the present paper) could potentially give information about the celebrated $XY$ model (which is the spin $O(2)$ model).

Interestingly, an explicit conjecture about the $n=1$ case was made in \cite{AGS2022}. The authors believe that the percolation threshold for the TSLS of the DGFF should be constant, and further hinted  (see Remark 1 and a simulation in Figure 3 there) that this threshold should be the same as  the one for two-valued sets of the continuum GFF \cite{ASW2019}. 
Even though we will not provide a rigorous proof, we have obtained some preliminary simulation results which suggest a different scenario than what is hinted in \cite[Figure 3]{AGS2022}, see Remark~\ref{rem:improve_critical} below.

\subsubsection{Main result on TSLS of DGFF}\label{sec:main_gff} 
Throughout the introduction, we will state the results on $\Zb^2$, but they can also be generalized to the metric graph (see Section~\ref{sec:extensions} for some extensions of our main results).

For $N \ge 1$, let $B_N := [-N,N]^2\cap\Zb^2$ be the discrete box centered on $0$ with side length $2N$.
Let  $\varphi_N$ be a DGFF in $B_N$ with Dirichlet boundary conditions. We are interested in the existence of large-scale paths, for instance crossing $B_N$ from left to right, in two-sided level sets (TSLS) of the field  $\varphi_N$. That is, nearest-neighbor paths along which $|\varphi_N|$ remains smaller than some given level $\lambda > 0$. 	
For any $0 < n \leq N$, and any subset $A\subseteq B_N$, consider the crossing event
\begin{equation}\label{eq:Cn}
	\Cc_{n}(A) :=  \left\{ 
	\text{there exists a path in $A\cap B_n$ crossing from left to right in $B_n$}  \right\},
\end{equation}
where the left side of $B_n$ is defined as the set of leftmost vertices in $B_n$, and the right side is defined similarly (as well as the top and bottom sides).
Our first main result is the following.

\begin{theorem}\label{thm:GFF}
	There exist $c_0, C_0>0$, such that for all $C\ge C_0$,
	\begin{align}\label{eq:gff-2}
		&\Pb( \Cc_N( \{ z\in B_N: |\varphi_N(z)|\le C \sqrt{\log \log N} \} ) )= 1- O((\log N)^{-c_0 C^2}) \quad \text{as} \quad N\to \infty.
	\end{align}
	Moreover, for all $C\ge C_0$,
	\begin{align}
		\label{eq:gff-1}
		&\liminf_{N \to \infty}\Pb(  \Cc_{N/2}( \{ z\in B_N: |\varphi_N(z)|\le C \sqrt{\log \log N} \} ) )>0.
	\end{align}
\end{theorem}

Note that $\varphi_N(z)$ follows a Gaussian distribution with mean $0$ and variance $G_N(z,z)$, where $G_{N}(\cdot,\cdot)$ is the Green's function in $B_N$ (its precise definition is given in \eqref{eq:Green}). 
By \cite[Lemma 1]{BDG2001}, there exists a universal constant $C'>0$ such that for all $N\ge 10$ and $z\in B_{N/2}$,
\begin{equation}\label{eq:green-estimate}
	\big|G_N(z,z)- g\log N\big| \le C'
\end{equation}
(recall that $g=2/\pi$).
Thus there exists a universal constant $c>0$ such that for all $N\ge 10$ and $z\in B_{N/2}$,
\begin{equation}\label{eq:density}
	\Pb( |\varphi_N(z)|\le C \sqrt{\log \log N} ) \le \frac{c\, C\, \sqrt{\log \log N}}{\sqrt{\log N}}.
\end{equation}
The maximum of the GFF (in absolute value) has been shown in \cite{BDG2001} to be 
\begin{align}\label{eq:max}
	\max_{z\in B_N} |\varphi_N(z)|=\left(2\sqrt{g}+o(1)\right)\log N,
\end{align}
where $o(1)\rightarrow 0$ in probability as $N\rightarrow\infty$.

\begin{remark}\label{rmk:gff_bernoulli}
	By \eqref{eq:density}, the density of the random set $\{ z\in B_{N/2}: |\varphi_N(z)|\le C \sqrt{\log \log N} \}$ vanishes as $N$ goes to infinity.
    We stress that this property makes our situation very different from the one-sided case, where $\{ z\in B_{N/2}: \varphi_N(z)\ge h\}$ occupies half of the vertices as $N\to \infty$ (by \eqref{eq:density-bound}).
	Our result is also contrary to Bernoulli percolation, where open sites should have a positive density (above a certain critical threshold) to percolate. 
	This situation, which we find remarkable, occurs because the DGFF takes advantage of the strong correlations. A similar phenomenonology, where positive association helps percolation, is also conjectured in three dimensions, as explained in the introduction of \cite{DPR2018a}.
\end{remark}

\begin{figure}[t]
	\centering
	\includegraphics[width=\textwidth]{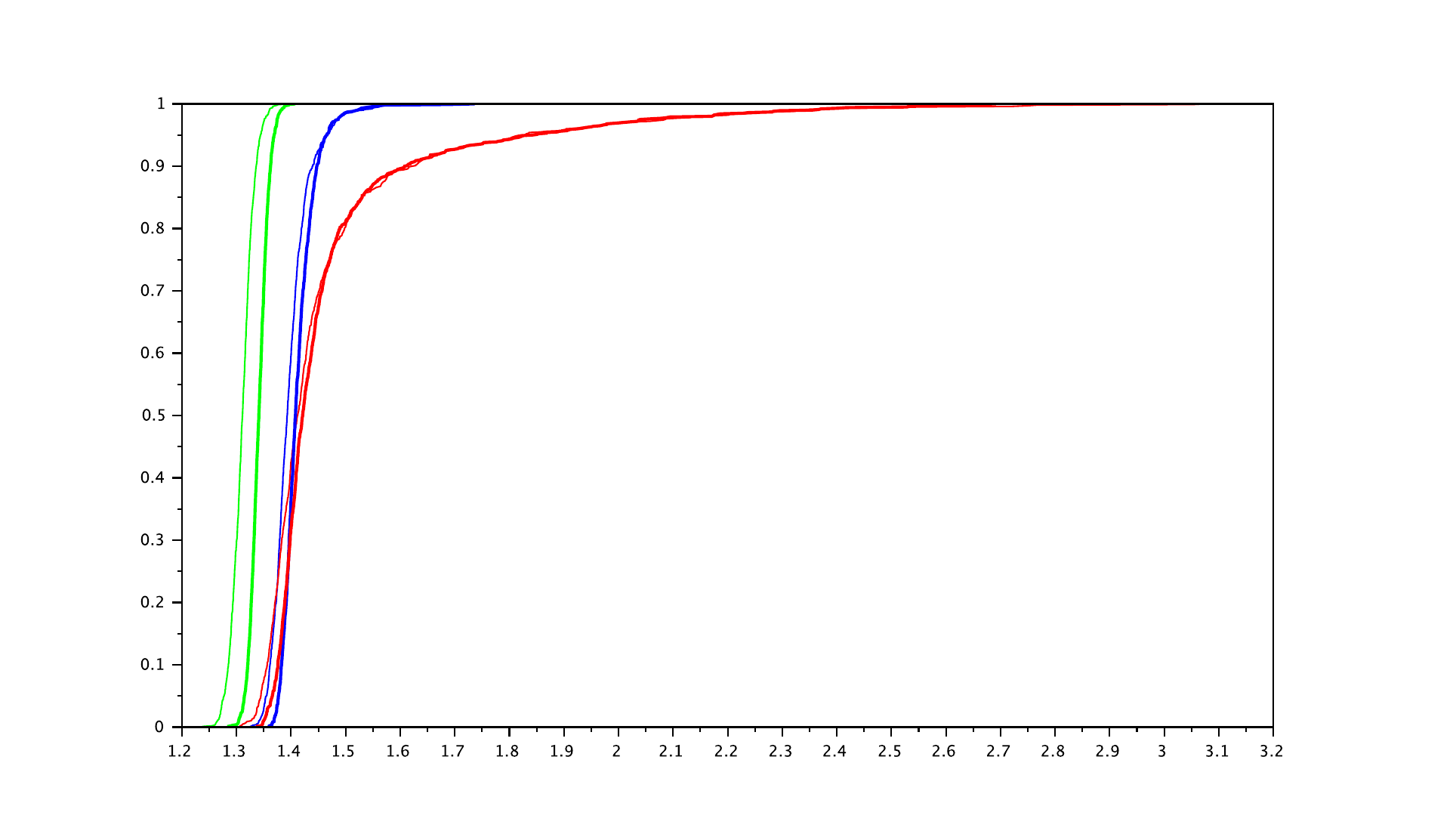}
	\caption{This figure shows estimates for the distribution of the threshold in the box $B_N$ with side length $2N$, for $N = 2048$ (in thinner line) and $N=4096$ (in thicker line), using three different settings. Each of them is obtained from $1000$ samples, in the following way. For the blue curve, we measure the smallest value of $\lambda$ such that $\Cc_N( \{ z\in B_N: |\varphi_N(z)|\le \lambda \} )$ occurs. In red, we find the smallest $\lambda$ such that $\Cc_{N/2}( \{ z\in B_N: |\varphi_N(z)|\le \lambda \} )$ occurs (bulk case). Finally, in green we consider a setting which is analogous to that in \cite[Figure 3]{AGS2022}: we determine the smallest $\lambda$ such that there is a crossing of the annulus $B_N \setminus B_{N/2}$ along which $|\varphi_N| \leq \lambda$.}
	\label{fig:simulation}
\end{figure}

\begin{remark} \label{rem:improve_critical}
As far as we know, it is still a widely open question to determine the percolation threshold for the TSLS of the DGFF, or at least its order of magnitude (as $N \to \infty$). It was conjectured by Aru, Garban and Sep\'ulveda  in \cite[Remark 1]{AGS2022} that this threshold should be constant in the limit. They further hinted, supported by a simulation in \cite[Figure 3]{AGS2022}, that this constant is possibly $\lambda_0$, where $2\lambda_0$ is the height gap of the continuum GFF. This conjecture is quite natural, because $\lambda_0$ also turns out to be the percolation threshold for the two-valued sets of the GFF, which is conceptually the continuum analogue of the TSLS of the DGFF.

At present, we have no argument that would support or refute this conjecture. However, we have made some preliminary simulations which give strong hints that the discrete threshold is \emph{not} $\lambda_0$, see Figure~\ref{fig:simulation}. In our normalization, the height gap is $2\lambda_0 = \sqrt{2 \pi}$, and the thresholds that we observe are clearly above $\lambda_0 = 1.253314\ldots$ (even though in the annulus case, they are rather close for small boxes).

To conclude this discussion, we wish to point out that there is, on the mathematical side, also an essential gap between the discrete picture and the continuum one (see Remark~\ref{rmk:gff1} for a more detailed discussion). 
\end{remark}

\subsubsection{Consequences on thick points of a random walk}

Thanks to the generalized second Ray-Knight theorem (see Theorem~\ref{thm:gsrk}), \label{page}
we can further use the result for the DGFF to study \emph{thick points} of a random walk in $B_N$. 

For a continuous-time random walk $(W_s)_{s\ge 0}$ in $B_N$ with wired boundary conditions, its \emph{cover time} of $B_N$ is the first time that every vertex of $B_N$ has been visited. At the cover time, the local time at a typical vertex is asymptotic to $(\log N)^2/\pi$. The thick points are roughly speaking the points where the random walk has spent more time than average.
More precisely, we follow the definition given in \cite[Eq.~(2.4)]{AB2022}.
For any $t > 0$, let $(\Lfr^{t}_{B_N}(z))_{z\in B_N}$ be the family of local times induced by $(W_s)_{s\ge 0}$, up to the time when the local time at the boundary accumulates up to $t$ (see Section~\ref{subsec:GFF} for rigorous definitions). For $a\in(0,1]$ and $\theta>0$, we define the following set of thick points with thickness $a$
\begin{equation}\label{eq:thick}
	\Tc_N^+(\theta,a):= \Big\{  z\in B_N: \Lfr^{t_N}_{B_N}(z)\ge \frac{(\sqrt{\theta}+a)^2}{\pi}(\log N)^2 \Big\},
\end{equation}
where
	\begin{equation}\label{eq:tN}
		t_N:=\frac{\theta}{\pi}(\log N)^2
	\end{equation}
	(which is asymptotic to the typical time that the random walk spends at any given vertex, at the $\theta$-multiple of the cover time).
It is proved in \cite[Theorem 1.2]{Ab2015} (see also \cite[Eq.~(2.18)]{AB2022}) that
\begin{align}\label{eq:thick_estim}
	|\Tc_N^+(\theta,a)|=N^{2(1-a^2)+o(1)},
\end{align}
where $o(1)\rightarrow 0$ in probability. Note that this asymptotics depends only on the thickness $a$, but not on $\theta$.

We are able to deduce the following connectivity result for thick points, with a thickness which vanishes as $N\to \infty$. We first state our results in this case as the following Theorem~\ref{thm:thick}, and refer the reader to Theorem~\ref{thm:lt} for a more general form.
\begin{theorem}\label{thm:thick}
Let $c_0,C_0$ be as in Theorem~\ref{thm:GFF}.
	Let $a_{C, N}:=C\, \sqrt{\frac{\pi}{2}}\, \frac{\sqrt{\log\log N}}{\log N}$. Then for all $\theta>0$ and $C\ge C_0$, we have 
	\begin{align}\label{eq:thick-2}
		&\Pb\left( \Cc_{N}\left( \Tc_N^+ \left( \theta,a_{C,N}\right) \right) \right)=O((\log N)^{-c_0 C^2}) \quad \text{as} \quad N\to \infty.
	\end{align}
	Moreover, for  all $\theta, \eps>0$,
	\begin{align}
		\label{eq:thick-1}
		&\limsup_{N \to \infty} \Pb\left( \Cc_{N/2}\left( \Tc_N^+ \left( \theta, a_{C,N} \right) \right) \right)<1.
	\end{align}
\end{theorem}
In fact, Theorem~\ref{thm:thick} holds true even if so-called $*$-crossings are considered, i.e., if one allows paths in $\Tc_N^+ ( \theta,a_{C,N})$ to use the diagonals of the faces of $\Zb^2$ (in other words, at each step, one can jump not only to nearest neighbors, but also to next-nearest neighbors), see Remark~\ref{rem:*thick}.

Thick points of a random walk were first studied by Erd\H{o}s and Taylor in \cite{MR0121870}, and
have been revisited many times (see, e.g., \cite{Ro2005,DPRZ2006,Ab2015,Ok2016,Je2020b, AB2022}) since the breakthrough \cite{DPRZ2001}, which described its multifractal structure.
More recently, thick points have been studied and exploited in a more detailed way, producing many interesting objects related to conformal geometry (see \cite{BBK,AHS, AB2022,ABL2019,Je2020a,Je2023,ABJL2023}). 
Note that for $a>1/\sqrt{2}$, \eqref{eq:thick_estim} immediately implies that there cannot be macroscopic crossings (length of order $N$) in $\Tc_N^+(\theta,a)$. However, to the best of our knowledge, no connectivity results for $a$-thick points were known for $a\le 1/\sqrt{2}$, partly because the strong correlations place it out of the class of Bernoulli percolation. We believe that our results and our techniques open the door to the study of such strongly correlated fields.

\subsection{Random walk loop soup}\label{subsec:rwls}
Le Jan \cite{LeJa2010, MR2815763} states that the occupation field $X_N$ of a random walk loop soup (RWLS) with critical intensity $\alpha=1/2$ in $B_N$ has the same law as $\varphi_N^2/2$, where $\varphi_N$ is a DGFF  in $B_N$.
This is a particular case of the \emph{isomorphism theorems} which have a rich history, see e.g.\
\cite{Symanzik66Scalar,Symanzik1969QFT,
	BFS82Loop,Dynkin1984Isomorphism,
	Dynkin1984IsomorphismPresentation,
	Wolpert2,
	EKMRS2000,
	MarcusRosen2006MarkovGaussianLocTime,
	MR2932978}.
This immediately allows us to translate Theorem~\ref{thm:GFF} to a result about the critical ($\alpha=1/2$)  RWLS. 

In fact, we will study the occupation fields of all RWLS with (subcritical and critical) intensities $\alpha \in (0,1/2]$. 
For the subcritical case $\alpha\in (0,1/2)$, we obtain much stronger results than in the critical $\alpha=1/2$ case, namely percolation for the occupation field exhibits a non-trivial phase transition at a constant level. 

\subsubsection{Background on Random walk loop soup}
The RWLS, introduced in  \cite{LTF2007}, is the discrete analogue of the Brownian loop soup (abbreviated as BLS in the discussion), introduced by Lawler and Werner \cite{LW2004}. The BLS emerged as a key object in the study of conformally invariant random systems in dimension two. 
In particular, the outer boundaries of outermost clusters in the BLS with intensity $\alpha \in (0,1/2]$ are shown by Werner and Sheffield \cite{SW2012} to be distributed as a conformal loop ensemble (CLE) with parameter $\kappa \in (8/3,4]$, where $\kappa$ and $\alpha$ are related explicitly by 
\begin{equation} \label{eq:alpha_kappa}
\alpha(\kappa) := \frac{(3\kappa-8)(6-\kappa)}{4\kappa}.
\end{equation}
The CLE, was first constructed in \cite{MR2494457} for $\kappa\in(8/3,8)$ using the Schramm-Loewner evolutions (SLE) \cite{Sc2000}, which are simple curves for $\kappa\in(0,4]$ \cite{MR2153402}. The range $\kappa\in(8/3,4]$ corresponds to the regime where the CLE is a countable collection of disjoint simple loops.
The CLE and SLE are both known, or conjectured, to be related to the scaling limits of many discrete models on lattices at criticality, in particular Bernoulli percolation \cite{Sm2001, CN2006}, the Ising model \cite{CS2012, BH2019}, and the random-cluster model with cluster weight $q=2$ \cite{Sm2010, KS2019, KS2016}.

Both the RWLS and the BLS are constructed from an infinite measure on loops, respectively $\mu^{\textrm{BL}}$ on Brownian loops in the continuous case, and $\mu^{\textrm{RWL}}$ on random walk loops in the discrete case. More precisely, for any given $\alpha$, we consider a Poisson point process of loops with an intensity which is either $\alpha \mu^{\textrm{BL}}$ on $\Cb$, or $\alpha \mu^{\textrm{RWL}}$ on $\Zb^2$. From Donsker's invariance principle, it is very natural to expect a convergence result for the rescaled RWLS to the BLS, and such a connection indeed holds true \cite{LTF2007}. 
It was shown in \cite{SW2012} that the BLS in a simply connected domain $D\subset \Cb$ displays a phase transition at the critical intensity $\alpha=1/2$, in the following sense: for all $\alpha \leq 1/2$, it contains infinitely many connected components of loops, while for all $\alpha > 1/2$, all loops are connected (it contains only one cluster). 
It was shown in \cite{BCL2016, Lu2019} that the outer boundaries of the outermost clusters of a RWLS with intensity $\alpha\in (0,1/2]$ (on the lattice $(\delta\Zb)^2 \cap D$) converge as $\delta\to 0$ in law to a CLE$_\kappa$ in $D$.
In the present paper, we thus restrict to intensities $\alpha \leq 1/2$, which are exactly the intensities where the RWLS  and the BLS are connected to CLE.

We consider the random walk loop soup $\Lc_N=\Lc_{B_N}$ with intensity $\alpha >0$. 
Let $X_N=(X_N(z))_{z\in B_N}$ be the occupation time field of $\Lc_N$, that is, how much time the loops spend, in total, at the various vertices (see Section~\ref{subsec:RWLS} for precise definitions). 
Note that an occupation field is by definition a positive field. It was shown in \cite{MR2815763, MR3298468} that the occupation field of the loop soup is a permanental field. The latter was introduced by Vere-Jones \cite{MR1450811} before the BLS (also see \cite{MR3077522}). 
Note that all moments of that field can be computed explicitly (see \cite[Proposition 16]{MR2815763}), and we have in particular for its covariance structure:
\begin{equation}
	\mathrm{Cov}(X_N(z),X_N(w))=\alpha\, G_N(z,w)^2.
\end{equation}

\subsubsection{Main results on random walk loop soup}\label{subsubsec:rwls}
For any $\lambda>0$, each vertex $z\in B_N$ is called $\lambda$-open if $X_N(z)\le\lambda$, and $\lambda$-closed otherwise, and a $\lambda$-open path for $X_N$ is a nearest-neighbor path in $B_N$ along which all vertices are $\lambda$-open. Recall the event $\Cc_n(\cdot)$ defined in \eqref{eq:Cn}.
Define the (horizontal) boundary crossing event by
\begin{equation}\label{eq:E-lbd}
	\Ec_{\alpha,N}(\lambda) :=  \Cc_N( \{ z\in B_N: \text{$z$ is $\lambda$-open} \} ).
\end{equation}
Since this event is clearly increasing in $\lambda$, we can associate with it the critical parameter
\begin{align} \label{eq:crit_val1}
	\lambda_*(\alpha) := \inf \left\{ \lambda\ge 0: \liminf_{N \to \infty} \Pb( \Ec_{\alpha,N}(\lambda) )>0 \right\} \in [0, \infty],
\end{align}
with the usual convention $\inf \emptyset=\infty$. 
We are also interested in the ``bulk'' case: we define the internal crossing event
\begin{equation}\label{eq:tE-lbd}
	\Ec'_{\alpha,N}(\lambda) := \Cc_{N/2}( \{ z\in B_N: \text{$z$ is $\lambda$-open} \} ),
\end{equation}
and we associate with it the critical value
\begin{align} \label{eq:crit_val2}
	\lambda'_*(\alpha) := \inf \left\{ \lambda\ge 0: \liminf_{N \to \infty} \Pb( \Ec'_{\alpha,N}(\lambda) )>0 \right\} \in [0,\infty].
\end{align}

Before stating our main theorem for the random walk loop soup, we also need to define the following ``boundary two-arm'' exponent (see \eqref{eq:2_arm_exponent_bdy}). For $\alpha\in (0,1/2)$,
\begin{align}\label{eq:four_arm}
	\xi^{2,+}(\alpha) := \frac{24}{13-2\alpha-\sqrt{4\alpha^2-52\alpha+25}}-1 \in (1,2).
\end{align}

\begin{theorem}\label{thm:main}
	For all $\alpha\in (0,1/2)$, we have $0<\lambda_*(\alpha)\le\lambda'_*(\alpha)<\infty$. 
	Moreover, for all $\eps>0$, there exists $\bar\lambda_0(\alpha,\eps)\in(0,\infty)$ such that for all $\lambda\ge \bar\lambda_0(\alpha,\eps)$,
	\begin{align}\label{eq:main-cr}
		\Pb( \Ec_{\alpha,N}(\lambda) )= 1- O(N^{1-\xi^{2,+}(\alpha)+\eps}) \quad \text{as} \quad N\to \infty.
	\end{align}
\end{theorem}
We point out that $\Eb(X_N(z))$ is of order $\log N$ for all $z\in B_{N/2}$. 	
Indeed, it follows from \cite[Corollary 1]{MR2815763} (see the paragraph below this result) that $X_N(z)$ has a Gamma distribution with parameters $(\alpha,G_{N}(z,z))$, i.e.\ $X_N(z)$ has density
\begin{equation}\label{eq:density_XN}
	\frac{1}{\Gamma(\alpha) G_{N}(z,z)^{\alpha}} x^{\alpha-1} e^{-x/G_{N}(z,z)} \mathbbm{1}_{x > 0}.
\end{equation}
In particular, $\Eb(X_N(z))=\alpha\, G_N(z,z)$, which has order $\alpha g\log N$  for all $z\in B_{N/2}$ by~\eqref{eq:green-estimate}. 
Moreover, by combining \eqref{eq:density_XN} and \eqref{eq:green-estimate}, we can get immediately that for all $\alpha >0$ and $\lambda>0$, there exists $C(\alpha,\lambda)>0$ such that for all $z\in B_{N/2}$,
\begin{equation}\label{eq:density-bound}
	\Pb( X_N(z)\le \lambda ) \le \frac{\lambda^{\alpha}}{\alpha\Gamma(\alpha)G_{N}(z,z)^{\alpha}}\le C(\alpha,\lambda) (\log N)^{-\alpha}.
\end{equation}

\begin{remark}\label{rmk:rwls}
	The right-hand side of \eqref{eq:density-bound} can be viewed as an upper bound for the density of $\{ z\in B_{N/2}: \text{$z$ is $\lambda$-open} \}$, which is in particular vanishing as $N \to \infty$ for any given $\lambda$. This is again in stark contrast to Bernoulli percolation (also commented in Remark~\ref{rmk:gff_bernoulli}).
\end{remark}

Finally, we can analyze the one-arm event, defined by: for any $\lambda > 0$,
	\begin{equation} \label{eq:arm_def}
		\mathrm{Arm}_{n,N}(\lambda):=\{ \text{$0$ is connected to $\partial B_n$ by a $\lambda$-open path in $B_N$} \}, \quad 0 < n \leq N.
	\end{equation}
More precisely, we prove an exponential decay property for the probability of that event, in the whole subcritical regime (see Section~\ref{subsec:annulus}).
	\begin{theorem} \label{thm:exp_decay}
		For all $\alpha\in (0,1/2)$, and $\lambda < \lambda'_*(\alpha)$, there exists $c(\alpha,\lambda)>0$ such that the following holds. For all $0 < n \le N$,
		\begin{equation}\label{eq:arm}
			\Pb( \mathrm{Arm}_{n,N}(\lambda) ) \le e^{-cn}.
		\end{equation}
\end{theorem}

\begin{remark}\label{rmk:supercritical}
	In the aforementioned paper \cite{AGS2022}, which uses very interesting and completely different techniques and ideas than here (e.g.\ a Mermin-Wagner-type theorem for the GFF),  the authors focus on the $n$-vector valued DGFF $\phi: \Zb^2\to \Rb^n$. For all $n\ge 2$, the authors prove that for $\lambda =o((\log n)^{1/6})$, the level set for $\|\phi\| \le \lambda$ is degenerate, namely it does not percolate and an exponential decay property holds for low paths in a box. By the isomorphism theory, $\|\phi\|^2$ is equal to the occupation field of a RWLS with intensity $\alpha=n/2$. This immediately implies the same triviality of level sets (under the level $o((\log n)^{1/3})$) for the occupation field of a \emph{supercritical} RWLS, for any intensity $\alpha \geq 1$.
	This naturally leads to the question of what happens for the intensities $\alpha\in(1/2,1)$, see (Q\ref{q2}).
\end{remark}

\subsection{Geometry of the percolating paths}\label{subsec:geometry}
Our proof relies on a fine analysis of the geometric properties of the clusters (i.e., maximal connected components) in the RWLS, which enables us to leverage the aforementioned connection between the  boundaries of these clusters and the conformal loop ensemble (CLE) process with parameter $\kappa \in (8/3,4]$, through certain arm exponents (see Section~\ref{sec:prop_four_arm}). We refer to Section~\ref{subsec:strategy} for a detailed description of our strategy.

In particular, this analysis allows us to show that in the percolating regime, one can find a ``low'' crossing path which stays in the \emph{carpet} of the corresponding RWLS (see Definition~\ref{def:carpet} for a rigorous definition), namely the path never intersects the ``interior'' of any of the outermost clusters.

\subsubsection{Main result for the DGFF}
As mentioned earlier, Le Jan's isomorphism allows us to couple a DGFF $\varphi_N$ in $B_N$ with a critical ($\alpha=1/2$) RWLS $\Lc_N$ in $B_N$, so that the occupation field $X_N$ of $\Lc_N$ is equal to $\varphi_N^2/2$.
\begin{definition} \label{eq:def_DN}
	For $\lambda>0$, let $\Dc_N(\lambda)$ (resp.\ $\Dc'_N(\lambda)$) be the event that there exists a nearest-neighbor path $\gamma$ in $\{ z\in B_N: |\varphi_N(z)|\le \lambda\}$ crossing from left to right in $B_N$ (resp.\ $B_{N/2}$), such that $\gamma$ also stays in the carpet of $\Lc_N$. 
\end{definition}
The following theorem strengthens Theorem~\ref{thm:GFF}.

\begin{theorem}\label{thm:GFF_carpet}
	Let $c_0,C_0$ be as in Theorem~\ref{thm:GFF}. Then for all $C\ge C_0$,
	\begin{align}
		\label{eq:critical_carpet1}
		&\Pb( \Dc_{N}(C \sqrt{\log\log N}) )=1-O( (\log N)^{-c_0 C^2}) \quad \text{as} \quad N\to \infty.
	\end{align}
	Moreover, for all $C\ge C_0$,
	\begin{align}
		\label{eq:critical_carpet2}
		&\liminf_{N \to \infty}\Pb( \Dc'_{N}( C\sqrt{\log\log N}) )>0.
	\end{align}
\end{theorem} 
\begin{proof}[Proof of Theorem~\ref{thm:GFF}]
	Note that by definition, $\Dc_{N}(\lambda) \subseteq \Cc_N( \{ z\in B_N: |\varphi_N(z)|\le \lambda \})$ and $\Dc'_{N}(\lambda) \subseteq \Cc_{N/2}( \{ z\in B_N: |\varphi_N(z)|\le \lambda \})$. Hence Theorem~\ref{thm:GFF_carpet} implies Theorem~\ref{thm:GFF}.
\end{proof}

\begin{remark}\label{rmk:carpet}
	The carpet of CLE$_\kappa$  has Hausdorff dimension $d(\kappa)=(\kappa+8)(3\kappa+8)/(32\kappa)$ \cite{MR2491617,MR2802511}. For $\alpha=1/2$, this dimension is $15/8$ (note that $\alpha$ and $\kappa$ are related by \eqref{eq:alpha_kappa}).
	We believe that if a suitable rate of convergence from the RWLS clusters to the BLS clusters is established, then Theorem~\ref{thm:GFF_carpet} would imply that the chemical distance dimension in the TSLS at level $C\sqrt{\log\log N}$, for all $C$ sufficiently large, is bounded above by $15/8$. We plan to address this question in the future.
\end{remark}

Note that Theorem~\ref{thm:GFF_carpet} concerns not only the DGFF, but also the RWLS coupled with it via isomorphism. 
In the following Remark~\ref{rmk:gff1}, we recall some analogous couplings in the continuum. In particular, we evoke some striking similarities w.r.t.\ the GFF/CLE$_4$ coupling by Miller and Sheffield, but also differences.  Despite the similarities, there seems to be a huge gap between the geometric properties of the corresponding objects in the continuum and the percolative properties of the TSLS in the discrete. We emphasize that our proof does not use the continuum GFF nor its level lines (readers unfamiliar with the level lines of the continuum GFF can also skip Remark~\ref{rmk:gff1}). 
This is again in contrast with the one-sided case, where 
many connectivity properties of the discrete level set follow closely from properties of the level-lines (which are SLE$_4$-type curves \cite{MR3101840}) in the continuum, see e.g.\ \cite{MR4452651}.

\begin{remark}\label{rmk:gff1}
	It is natural to compare Theorem~\ref{thm:GFF_carpet} with some results on the continuum GFF.
	Let us now carry out a heuristic discussion, to point out some similarities and differences.
	
	\underline{Analogy with the couplings in the continuum.}
	In the continuum, the critical BLS ($\alpha=1/2$), GFF and CLE$_4$ are coupled together via three different relations: (1) Le Jan's isomorphism also applies to the occupation field of the critical BLS and the square of the GFF \cite{LeJa2010, MR2815763}, but a renormalization is needed, since the BLS cumulates an infinite amount of time in any open set and the GFF is a distribution which does not have pointwise value. (2) The outer boundaries of the outermost clusters of the critical BLS have the law of CLE$_4$, by Sheffield and Werner \cite{SW2012}. (3) Miller and Sheffield state that CLE$_4$ can be coupled with the GFF as its level loops \cite{MS,MR3708206,ASW2019}. It is proved in \cite{QW2019} that we can couple the critical BLS, CLE$_4$ and GFF in the same space so that the three relations hold simultaneously. 
	
	Level loops are a special case of the level lines (for the continuum GFF) defined by Schramm and Sheffield \cite{MR3101840}, using the fundamental notion of \emph{local set} which was also introduced there. The continuum $a$-level lines are shown in \cite{SS2009} to be the scaling limit of $a$-level lines of the DGFF (discrete $a$-level lines are interfaces between connected components of vertices with value $<a$ and $>a$). In the continuum, the CLE$_4$ carpet is the set bounded between outermost $\lambda_0$-level loops and $-\lambda_0$-level loops for some explicit $\lambda_0>0$. The result in \cite{SS2009} was shown in the chordal setting, but we believe that a version of this result should still be true for the level-loops. In this sense, the CLE$_4$ carpet can (loosely) be seen as the ``continuum analogue'' of the TSLS in $[-\lambda_0, \lambda_0]$ for the DGFF. 
	
	This bears a conceptual similarity with Theorem~\ref{thm:GFF_carpet} which states that there is a percolating path (in the TSLS of the DGFF) that stays in the carpet of the RWLS. However, we will point out some major differences in the following.
	
	\underline{Differences compared to the continuum setting.}
	\begin{itemize}
		\item 
		The analogy between the CLE$_4$ carpet and the TSLS is based on the convergence of the discrete level lines/loops. Any two loops in CLE$_4$ do not touch each other, but the discrete level loops (which should converge to the CLE$_4$ loops) may very well touch each other.
		This is because there are infinitely many loops in CLE$_4$, and its carpet is a fractal set with $0$ Lebesgue measure. Therefore, although the carpet of CLE$_4$ is connected and connected to the boundary, the convergence of the level lines/loops does not imply that the TSLS in $[-\lambda_0, \lambda_0]$ is also connected. 
		In fact, a more general two-valued set was studied in \cite{ASW2019}, and can be seen (loosely) as the continuum analogue of the TSLS in $[-a, b]$ for $a,b>0$. 
		However, similarly, 
		properties of the two-valued sets do not readily imply connectivity properties of the TSLS in $[-a, b]$, even for arbitrarily large $a,b>0$.

		There is a subtle difference here compared to the one-sided case. The continuum analogues of one-sided level sets (above a constant level) can also be described using level lines  \cite{ALS2020a, ALS2020b}, and they are also fractal sets with $0$ Lebesgue measure. However, there is a natural Markovian way of exploring the one-sided level sets from the boundary of the domain (as in \cite{SS2009}), so that the boundary-intersecting behavior of the level lines (which are SLE$_4$-type curves in the scaling limit) does imply the connectivity property of certain one-sided level sets (one usually needs to fix a piecewise constant boundary condition, and look at the connected component of the one-sided set containing a given piece of the boundary). This is similar to the fact that SLE$_6$ (which is the scaling limit of percolation interfaces \cite{Sm2001}) can describe many connectivity properties of the percolation clusters (see e.g.\ \cite{LSW2002b,SW2001}).
		However, there seems to be no natural Markovian way to explore the TSLS from the boundary, so that the interface is given by a level line. In fact, it seems plausible that the interface of a component of the TSLS (say in $[-a,a]$) should not stay at a constant level (either $a$ or $-a$), but that it should rather alternate between $-a$ and $a$, probably infinitely many times as the meshsize goes to $0$.

		\item The CLE$_4$ carpet was compared to the TSLS in $[-\lambda_0, \lambda_0]$, but we obtain results for TSLS at level $C\sqrt{\log\log N}$.
		This might inspire one to make use of the continuum level lines in the following way: If a sequence of discrete level lines $\ell_N$ is known to converge to a continuum level line $\ell$ which is itself known to make a left-right crossing, then can we say that the discrete vertices which are adjacent to $\ell_N$ have absolute values less than $C\sqrt{\log\log N}$? This would create a low crossing in the discrete by the vertices along $\ell_N$. This question is legitimate, because it is shown in \cite{MR3101840} that level lines in the continuum are local sets of the GFF, in the sense that if one conditions on a level line $\ell$, then the GFF in the complement of $\ell$ is just a GFF with constant boundary values along $\ell$ (there is in fact a $2\lambda_0$ height gap at the two sides of $\ell$).
		Since the DGFF converges to the GFF in distribution,  the DGFF restricted to the domain cut by $\ell_N$ should also be very close to a DGFF with constant boundary values along $\ell_N$.
		
		However, not to mention the difficulty to make the approximations quantitative, the following heuristic indicates that this approach can hardly work. Take a DGFF $\varphi_N$ in $B_N$ with $0$ boundary conditions (i.e.\ it is $0$ on $\Zb^2\setminus B_N$), and look at the straight line $\eta_N$ from $(-N/2, -N)$ to $(N/2, -N)$ with length $N$. For each vertex $z$ on $\eta_N$, $\varphi_N(z)$ is a Gaussian variable with mean $0$ and variance $G_N(z,z)$, where $G_N(z,z)$ tends to a constant $\sigma>0$ as $N\to\infty$. One can roughly consider the random variables $\varphi_N(z)$ as i.i.d.\ for $z$ on $\eta_N$, because the correlation between $\varphi_N(z_1)$ and $\varphi_N(z_2)$ decays rapidly as $z_1$ gets away from $z_2$.
		The maximum of $N$ i.i.d.\ Gaussian variables is known to have expectation $\sigma \sqrt{\log N}$, which is above $C\sqrt{\log\log N}$. Hence $\eta_N$ is not a low path even though it is adjacent to the $0$ boundary.

		\item As mentioned earlier, we can couple the critical BLS, CLE$_4$ and GFF in the same space so that the three relations listed at the beginning of this remark hold simultaneously \cite{QW2019}. 
		However, the simultaneous coupling of the critical RWLS, its carpet and the DGFF does not hold exactly in the discrete. Even though the carpet of a critical RWLS is connected and connected to the boundary, it is \emph{not} the same as the TSLS of the DGFF. In fact, there are vertices in the carpet of a RWLS with arbitrarily high occupation times. 
		Nevertheless, we observe that a typical vertex in the carpet has low occupation time with high probability, and make use of this fact in our proof. Even though the carpet of the RWLS is connected, it is a thin fractal set which is very close to being disconnected. Our proof relies on a quantitative analysis which combines the spatial distributions and the occupation times of the vertices in the carpet (see Section~\ref{subsec:strategy} for more details on the proof strategy).
	\end{itemize}
\end{remark}

\subsubsection{Main result for the subcritical RWLS}

Similarly to the $\alpha=1/2$ case, we also have the following result which provides additional geometric information about the low crossings for $\alpha\in(0,1/2)$. 
\begin{definition}\label{def:carpet_cross}
	Let $\Lc_N$ be a RWLS in $B_N$ with intensity $\alpha\in(0,1/2]$.
	Let $\Fc_{\alpha,N}(\lambda)$ (resp.\ $\Fc'_{\alpha,N}(\lambda)$) be the event that there is a $\lambda$-open crossing from left to right in $B_N$ (resp.\ $B_{N/2}$) (for the occupation field of $\Lc_N$) which stays in the carpet of $\Lc_N$. 
\end{definition}
By definition, we have $\Fc_{\alpha,N}(\lambda) \subseteq \Ec_{\alpha,N}(\lambda)$ and $\Fc'_{\alpha,N}(\lambda) \subseteq \Ec'_{\alpha,N}(\lambda)$, so the following theorem strengthens certain points of Theorem~\ref{thm:main}.

\begin{theorem}\label{thm:carpet1}
	For all $\alpha\in(0,1/2)$ and $\eps>0$, there exists $\bar\lambda_0(\alpha,\eps)\in(0,\infty)$ such that the following holds. For all $\lambda\ge \bar\lambda_0(\alpha,\eps)$,
	\begin{align}\label{eq:F_alpha1}
		&\Pb( \Fc_{\alpha,N}(\lambda) )= 1- O(N^{1-\xi^{2,+}(\alpha)+\eps}) \quad \text{as} \quad N\to \infty,\\[2mm]
		\label{eq:F_alpha2}
		&\liminf_{N \to \infty}\Pb( \Fc'_{\alpha,N}(\lambda) )>0.
	\end{align}
\end{theorem}

Similarly to Remark~\ref{rmk:carpet}, Theorem~\ref{thm:carpet1} could potentially lead to a non-trivial upper bound on the chemical distance dimension of $\lambda$-level sets for the occupation field of RWLS. The following remark, on the other hand, provides the parallel with Remark~\ref{rmk:gff1}.

\begin{remark}\label{rmk:rwls}
	The occupation field $X_N$ of a loop soup $\Lc_N$ is by definition positive. However, it is also possible to translate Theorems~\ref{thm:main} and~\ref{thm:carpet1} to the TSLS of the following field $\varphi_{\alpha, N}$ which is centered: For $\alpha\in(0,1/2]$, let  $|\varphi_{\alpha,N}(z)|:= X_N(z)^{1-\alpha}$ for each $z\in B_N$, and then give i.i.d.\ signs $\pm 1$ with probability $1/2$ to each cluster of $\Lc_N$. For $\alpha=1/2$, it was shown in \cite{MR3502602} that $\varphi_{1/2, N}$ is a DGFF. For $\alpha\in (0,1/2)$, a recent work \cite{JLQ2023b} has given the conjectural scaling limit of $\varphi_{\alpha, N}$, which is a  conformally invariant field $\varphi_\alpha$ in the continuum.
	It was further shown in  \cite{JLQ2023b} that the coupling between the critical BLS ($\alpha=1/2$), CLE$_4$ and GFF (mentioned in Remark~\ref{rmk:gff1}) can be extended to all subcritical intensities $\alpha\in(0,1/2)$, where $\varphi_\alpha$ plays the role of the GFF. 
	For example, $\varphi_\alpha$ admits CLE$_\kappa$ (where $\kappa$ and $\alpha$ are related by \eqref{eq:alpha_kappa}) as level lines, in a sense that conditionally on a CLE$_\kappa$ loop $\gamma$, $\varphi_\alpha$ restricted to the domain encircled by $\gamma$ has constant boundary conditions. 
	Therefore, we believe that Remark~\ref{rmk:gff1} should also apply to the $\alpha\in(0,1/2)$ case, modulo some differences (which we do not discuss for the sake of brevity).
\end{remark}

\subsection{Strategy of the proof}\label{subsec:strategy}

We now try to convey the general road map that we follow in this paper, and mention a few of the most important technical issues that we have to address. 
Even though we make use of the connection with CLE in the continuum, we mostly focus on the discrete analysis. Along the way, we had to develop a robust toolbox for the four-arm events that we use, and a large part of the companion paper \cite{GNQ2024a} is devoted to this endeavor. We believe that the tools and ideas developed there, in particular separation lemmas, quasi-multiplicativity, and locality properties for such events, will be useful to tackle other related questions.

\subsubsection{Chains of clusters} \label{sec:intro_chains}
First of all, even though the levels of various vertices are far from being independent of each other in our situation, the loops creating strong correlations, we use insight provided by the classical Bernoulli percolation model where vertices are independent. More precisely, we focus on its site version on the infinite two-dimensional lattice $\Zb^2$, which can be described as follows. For some given $p \in [0,1]$, known as the percolation parameter, each vertex of $\Zb^2$ is declared occupied with probability $p$, and vacant with probability $1-p$, independently of all other vertices. This process displays a phase transition at some critical value $p_c = p_c^{\text{site}}(\Zb^2) \in (0,1)$ of the parameter $p$, called the percolation threshold, where its connectivity changes drastically. In particular, for all $p < p_c$, there exists a.s.\ no infinite cluster of occupied vertices, while for $p > p_c$, there is a.s.\ at least one.

At a high level, our proof can be described as a Peierls'-type argument. In Bernoulli percolation, the existence of an infinite occupied cluster can be proved by showing the non-existence of a blocking circuit made of vacant sites, which can be done for values of the parameter close enough to $1$. In the same way, in order to prove the existence of a path with low occupation time crossing horizontally (from left to right) $B_N$, we consider the complementary event, namely the existence of a blocking path that crosses $R$ vertically (from top to bottom), all of whose vertices are high (see Figure~\ref{fig:blocking_path}, left). We show that this event has vanishingly small probability as $N$ tends to infinity, provided that the level $\lambda$ has been chosen large enough. For this purpose, we use the union bound over the set of all possible paths on the close-packed graph $(\Zb^2)^*$ of $\Zb^2$, which is obtained by adding the two diagonal edges to each face of $\Zb^2$. More specifically, we only keep track of the clusters visited by the path, together with the edges (of $(\Zb^2)^*$) between them.

The occupation field of a RWLS on $B_N$ has the following nice property: If one conditions on a cluster $\Cc$ (i.e.\ on all the vertices and edges visited by this cluster), then the occupation field on $\Cc$ depends only on the shape of $\Cc$, not on the position of $\Cc$ in $B_N$, and is also independent from the loop configuration in $B_N\setminus \Cc$ (as long as the other loops are disjoint from $\Cc$).
If all the clusters of loops were single-site, we would observe exactly Bernoulli percolation, with a parameter $p(\lambda) \in (0,1)$ which tends to $1$ as $\lambda\to \infty$. This would readily lead to the existence of a critical parameter $\lambda_*$. The actual situation is of course very different (see Remark~\ref{rmk:rwls}), since the open sites have a vanishing density as $N\to \infty$.

Heuristically, this picture should remain valid, at least to some extent, if all clusters were microscopic. However, we know from the description of these clusters in the continuum (their connection to CLE in the scaling limit, with a parameter $\kappa \leq 4$) that macroscopic components do exist.  
It can be tempting for the blocking path to use big clusters as ``highways'', where it is less costly (at least, deep in their interior) to have high occupation time. However, using such big clusters also has a cost: they cannot come too close to each other, as can be seen from the scaling limit. On the discrete level, we will use that for each intensity $\alpha\in (0,1/2)$, there is a corresponding exponent $\xi = \xi(\alpha)$ which is $>2$.

Of course, this is a very crude explanation. In reality, there is a whole range of sizes of clusters which are at our disposal, microscopic, mesoscopic and macroscopic, leading to all types of paths. At first sight, taking into account the ``entropy'' contribution in the union bound, which comes from the wide variety of all possible chains of clusters of all sizes, may look hopeless.

The clusters visited by a blocking path are used through the edges connecting any two successive ones, that we call ``passage edges'' from now on. Around each such edge, one can observe a ``four-arm'' configuration created by the two clusters that meet, as shown on Figure~\ref{fig:blocking_path}, right.

\begin{figure}
	\centering
	\subfigure{\includegraphics[width=.47\textwidth]{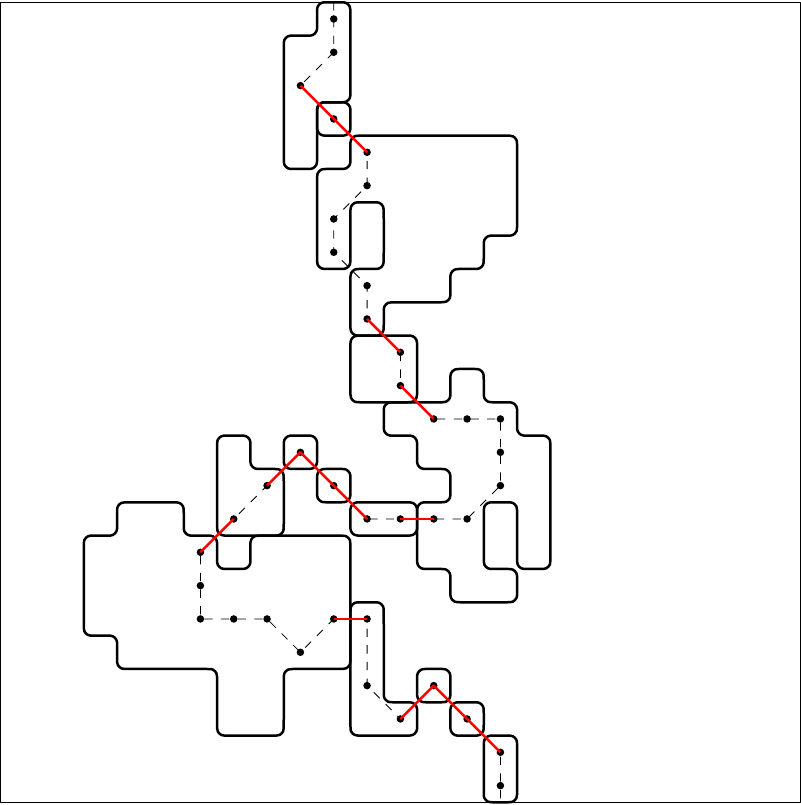}}
	\hspace{0.5cm}
	\subfigure{\includegraphics[width=.47\textwidth]{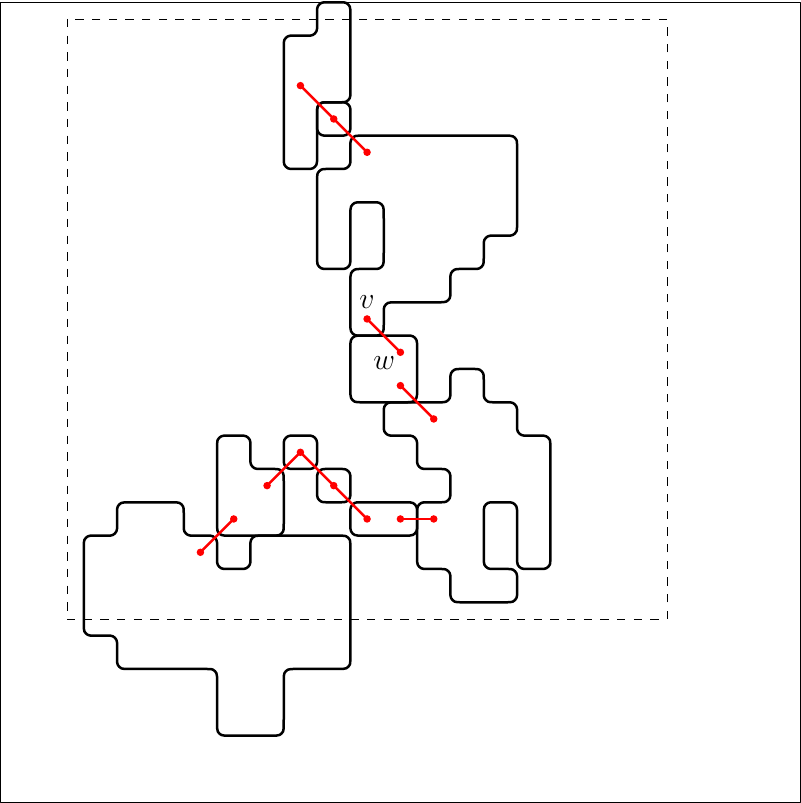}}
	\caption{\emph{Left:} We use the sequence of clusters of loops used by a potential blocking path, a vertical crossing of high vertices preventing the existence of a horizontal crossing of low vertices. \emph{Right:} If $(v,w)$ is a passage edge between two sucessive clusters, then we observe a ``four-arm'' configuration around $v$.}
	\label{fig:blocking_path}
\end{figure}

\subsubsection{Connection with forest fires} \label{sec:intro_FF}

It has been explained in the paper \cite{KMS2015} how such summations on all possible paths can be carried out, still in a percolation context but for a completely different problem, namely understanding the behavior of the Drossel-Schwabl forest fire model near criticality, that is, as the density of trees approaches the critical threshold $p_c$. There, Kiss, Manolescu and Sidoravicius answer a question of van den Berg and Brouwer which had remained open for a decade. To our knowledge, this is the first instance where such detailed summations have been successfully performed, combining microscopic, mesoscopic and macroscopic scales, and with an unbounded number of clusters.

More specifically, in order to analyze near-critical forest fires, one needs to understand the effect of ``macroscopic'' fires, in the following sense. We consider Bernoulli percolation with parameter $p_c$ in $B_N$, and we assume that an occupied horizontal crossing exists in this square. First, we turn vacant (``burn'') all the vertices which are connected (by an occupied path) to such a horizontal crossing, which has obviously the effect of disconnecting the bottom and the top sides (which may have been connected or not in the first place). We then consider a small ``recovery'' in the process, of size $\delta > 0$: all vacant vertices independently become occupied with probability $\delta$. It was shown in \cite{KMS2015} (or rather, a variant of this) that for some universal $\delta_0$ small enough, the probability of observing an occupied vertical crossing in the final configuration, i.e.\ after recoveries, converges to $0$ as $N \to \infty$. Informally speaking, it takes a positive amount of time for the forest to recover from large-scale fires.

In \cite{KMS2015}, the authors use chains containing distinguished vertices, that they call ``passage sites'', with the property that (roughly speaking) new vertical connections are formed when all passage sites are recovered. A key observation is that for such a chain which is minimal, a configuration as depicted schematically on Figure~\ref{fig:six_arms_FF} arises, with a six-arm configuration of the type $ovvovv$, where $o$ and $v$ stand for $p_c$-occupied and $p_c$-vacant arms. The corresponding arm exponent $\xi_6$ is known to be $> 2$ for critical Bernoulli percolation on a wide variety of two-dimensional graphs, and in particular on the square lattice $\Zb^2$ (where, similarly to above, the $p_c$-vacant arms lie on $(\Zb^2)^*$). Indeed, this only requires minimal symmetry assumptions on the graph, essentially so that a form of Russo-Seymour-Welsh estimate holds true. Note that the precise value of this exponent $\xi_6 = \frac{35}{12}$ is known in the case of site percolation on the triangular lattice $\Tb$ \cite{SW2001}. For universality reasons, the same value is expected to be observed on $\Zb^2$, however, this is not at all needed for the summations and the key is the inequality $\xi_6 > 2$.

The summations in \cite{KMS2015} use the exact spatial independence possessed by Bernoulli percolation. For the probability of new vertical connections, the proof produces the upper bound $N^{2 - \xi_6 +o(1)} = N^2 \cdot N^{-\xi_6 +o(1)}$, which can be thought as a stability result for six-arm events: not exactly the usual arm event (with types $ovvovv$) mentioned above, but rather a complicated and peculiar arm event which should incorporate the complicated ``network'' structure of passage sites. This led van den Berg and the second author to introduce such an ``exotic'' six-arm event whose convoluted definition -- this is a price to pay -- makes it amenable to an induction argument, in the spirit of Lemma~8.4 in \cite{GPS2018}. Such an induction was already applied by the same authors for a process of percolation with impurities \cite{BN2021} that they introduced as a stochastic lower bound for the Drossel-Schwabl process. A key idea in our proofs is to exploit the four-arm structure shown on Figure~\ref{fig:blocking_path} by introducing so-called $\lambda$-arm events, that incorporate the successive clusters potentially visited by a path, together with the passage edges between them. This then allows us to take advantage of the inequality $\xi(\alpha) > 2$ (see \eqref{eq:arm_exponent}), for any given subcritical intensity $\alpha \in (0,1/2)$.

\begin{figure}[t]
	\centering
	\includegraphics[width=.5\textwidth]{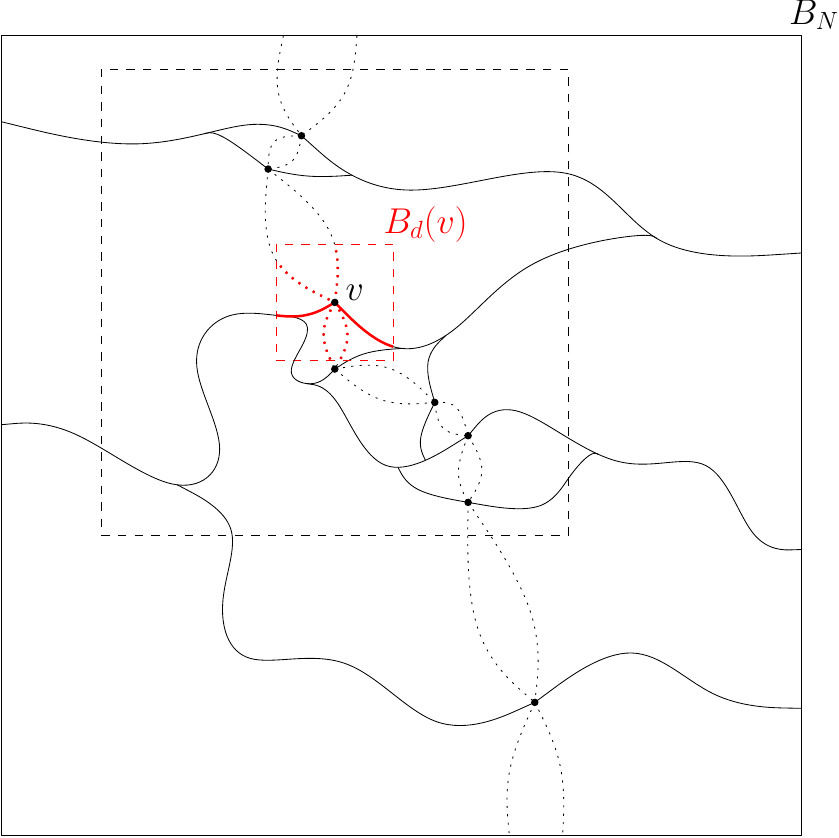}
	\caption{In order to understand the effect of recoveries in the forest-fire process, chains of so-called passage sites are considered in \cite{KMS2015}: for an occupied vertical crossing of $B_N$ using this chain to appear, all of the passage sites need to be rendered occupied by the $\delta$-recovery. This figure shows the geometric picture along such a chain, if it is chosen to be minimal: the $p_c$-occupied paths are drawn in solid line, while the $p_c$-vacant ones are in dotted line. Around each passage site $v$, we observe locally a ``six-arm'' configuration made of four vacant ($v$) and two occupied ($o$) arms, in the clockwise order $ovvovv$ (colored in red).}
	\label{fig:six_arms_FF}
\end{figure}

The critical case $\alpha=1/2$ turns out to be more subtle, since the corresponding exponent $\xi(1/2)$ is exactly equal to $2$, making the summations quite delicate. We study this case in Section~\ref{sec:critical-metric}, after having analyzed the subcritical case $\alpha < 1/2$ in detail in Section~\ref{sec:pf-thm}. For the critical intensity $\alpha=1/2$, we manage to derive a result with a level depending on $N$, $\lambda_N = C \log\log N$ for $C$ large enough, see Theorem~\ref{thm:carpet2} (which yields Theorem~\ref{thm:GFF_carpet}). This property is weaker than for $\alpha < 1/2$, but it is still highly non-trivial since the average occupation time is of order $\log N$. Moreover, our proof is contingent on getting rid of an error term $o(1)$ in the exponent (i.e. of potential logarithmic corrections in the power-law estimate for four arms), which we have done recently (together with Yijie Bi) in \cite{bi2025arm} -- this is not an issue in the subcritical case, since the inequality $\xi > 2$ leaves some ``breathing space''.

\subsubsection{Toolbox: arm events for RWLS}\label{subsubsec:arm}

In order to carry out the summation arguments sketched in Sections~\ref{sec:intro_chains} and \ref{sec:intro_FF}, we need to devise a set of tools to work with arm events in a random walk loop soup (RWLS). The companion paper \cite{GNQ2024a} is primarily concerned with establishing such results. In our proofs, we use the four-arm event, that two distinct clusters cross a given annulus, as well as the two-arm event (for only one cluster). As a matter of fact, the intuition provided by near-critical Bernoulli percolation turned out to be useful, even though the details of the proofs are not at all the same, involving different objects and making use of distinct sets of techniques.

We think that the properties of arm events that we developed, with our specific purpose in mind, are also interesting in their own right, and we hope that some of them (or some variations of them) will turn out to be useful in other contexts. As an illustration, we mention some of the main results that we obtained in this direction, in a brief and informal manner, before providing more precise statements in Section~\ref{sec:prop_four_arm}.
\begin{enumerate}
	\item First, we proved a \emph{separation} result, for packets of random walks inside a RWLS. This property is instrumental in that it then allows us to repeatedly perform surgery on well-chosen loops, by resampling specific parts of them.
	
	\item When working with four-arm events, one needs to address the potential effect of loops coming from far away. Thanks to the separation property, we were able to recombine big loops, which enabled us to show a property that we called \emph{locality}. We proved that the probability of a four-arm event in a given annulus $A$ can be upper bounded, up to a multiplicative constant, by a ``localized'' version of it, where only loops with a diameter comparable to that of $A$, or smaller than it, are taken into account. Note that the reverse inequality is not true.
	
	\item The locality property, combined with further surgery arguments, was then used to establish a \emph{quasi-multiplicativity} upper bound for four-arm probabilities. Together with the arm exponents provided by CLE, this bound allowed us to derive power law upper bounds on four-arm probabilities, which play a central role in the present paper.
\end{enumerate}

\subsection{Open questions and extensions}\label{sec:open_questions}

We conclude this introductory discussion by stating explicitly several open questions about possible extensions of our work.

\begin{enumerate}[(Q1)]
	\item\label{q1} A first obvious direction, already hinted at in Remark~\ref{rem:improve_critical}, is to try to improve our understanding of the critical intensity $\alpha=1/2$. In particular, can the thresholds appearing in Theorems~\ref{thm:GFF} and~\ref{thm:GFF_carpet} be improved?

	\item\label{q2} The results in \cite{AGS2022} provide the degeneracy of two-sided level sets in the RWLS for any intensity $\alpha \geq 1$ (see Remark~\ref{rmk:supercritical}): what about $\alpha \in (1/2,1)$? Similarly to (Q\ref{q1}), we can also ask what the actual threshold is for each $\alpha>1/2$. In particular, for $\alpha \geq 1$, can the corresponding bound $o((\log n)^{1/3})$ (from \cite{AGS2022}) for the threshold be improved? Moreover, can an independent proof be spelled out using techniques for random walk loops, similarly to the present paper?

	\item\label{q3} Finally, it is natural to ask whether some form of sharpness can be established for the percolation transition, as the threshold $\lambda$ increases. For $\alpha\in(0,1/2)$, we can also define the following critical values
	\begin{align*}
		\lambda_{**}(\alpha) := \inf \big\{ \lambda\ge 0: \liminf_{N \to \infty} \Pb( \Ec_{\alpha,N}(\lambda) )=1 \big\}, \, \lambda'_{**}(\alpha) := \inf \big\{ \lambda\ge 0: \liminf_{N \to \infty} \Pb( \Ec'_{\alpha,N}(\lambda) )=1 \big\}.
	\end{align*}
	Theorem~\ref{thm:main} implies that $\lambda_{**}(\alpha)<\infty$: do we have $\lambda_*(\alpha)=\lambda_{**}(\alpha)$? On the other hand, is it true that $\lambda'_{**}(\alpha)<\infty$?
	
	Recently, sharpness of the phase transition for one-sided level sets of the DGFF on $\Zb^d$, for all $d\ge 3$, has been established in \cite{MR4568695}. Furthermore, an improved description has also been obtained for a wide class of Gaussian percolation models on $\Zb^d$ or $\Rb^d$, for $d\ge 2$, under the assumption that correlations decay algebraically fast, see for example \cite{MS2022,Mu2022,MRVKS2023} and the references therein.
	
\end{enumerate}

\subsection{Organization of the paper}\label{subsec:org}

The paper is structured in the following way. We start by giving the setup and choosing notations in Section~\ref{sec:setup}. We then recall, in Section~\ref{sec:tools_RWLS}, the tools developed in the first part of this work \cite{GNQ2024a} for the random walk loop soup, primarily regarding the four-arm event. In that section, we also collect several elementary results for future use. In Section~\ref{sec:pf-thm}, we prove Theorems~\ref{thm:main} and~\ref{thm:carpet1} by introducing suitable modified arm events, arising from chains of clusters visited by a high path. We establish a stability property for their probabilities, based on an induction procedure. In Section~\ref{sec:critical-metric}, we discuss the critical case $\alpha=1/2$ and its relation to the discrete Gaussian free field. In particular, we obtain results for the two-sided level sets of the GFF, as well as for the thick points of a simple random walk. Finally, we discuss extensions of our results in Section~\ref{sec:extensions}.

\section{Setup and notations} \label{sec:setup}

\subsection{Notations}

Let $d(z,z') = |z-z'|$, for $z, z' \in \Rb^2$, be the standard Euclidean distance in the plane. Two vertices $z$ and $z'$ in $\Zb^2$ are said to be neighbors if $d(z,z') = 1$, which we denote by $z \sim z'$. As usual, the $l_{\infty}$ distance between two points $z, z' \in \Rb^2$, with $z=(x,y)$ and $z'=(x',y')$, is defined as $d_{\infty}(z,z') = |z-z'|_{\infty} := |x-x'| \vee |y-y'|$. In particular, $z, z' \in \Zb^2$ are \emph{$*$-neighbors} if $|z-z'|_{\infty} = 1$, and we indicate it by $z \sim^* z'$. We also write $d_{\infty}(A,B):=\inf_{z \in A,z'\in B} |z-z'|_{\infty}$ for any two subsets $A$ and $B$ of $\Zb^2$. We say that $A$ and $B$ are \emph{$*$-connected} if $d_{\infty}(A,B)\le 1$, i.e., if they intersect or there exist $z \in A$ and $z' \in B$ with $z \sim^* z'$. The complement of a subset $A\subseteq\Zb^2$ is defined as $A^c := \Zb^2 \setminus A$, its inner vertex boundary is $\partial A= \partial^{\mathrm{in}} A:= \{ z\in A : |z-w|_{\infty}=1 \text{ for some } w \in A^c\}$, and we let $\mathring{A} := A \setminus \partial A$. 
We define $\mathrm{Fill}(A)$ to be the complement of the unbounded $*$-connected component (i.e.\ for the $*$-connection) of $A^c$, and we let $\partial_{\mathrm{ext}}A:=  \partial^{\mathrm{in}}\, \mathrm{Fill}(A)$ be the ``external'' inner vertex boundary of $A$.
Finally, we use the notation $A^+ := A \cap \Hb$, where $\Hb:=\{ (x,y) \in \Rb^2: y>0 \}$ is the upper half-plane.

We consider \emph{paths} (resp. \emph{$*$-paths}) of vertices, which are simply finite sequences $\eta = (z_1,\cdots,z_j)$, where each $z_i \in \Zb^2$, and such that for all $i=1,\ldots,j-1$, $z_i \sim z_{i+1}$ (resp. $z_i \sim^* z_{i+1}$). We denote by $|\eta| = j-1$ the length of such a path, and $z_1$ and $z_j$ are respectively its starting and ending vertices. Moreover, a path consisting only of vertices in a subset $A\subseteq\Zb^2$ is called a \emph{path in $A$}. An \emph{excursion in $A$} is a path in $A$ whose starting and ending vertices both belong to $\partial A$.

A path which has the same starting and ending vertices is called a \emph{rooted loop}. Any such loop can be \emph{unrooted}, and we denote the corresponding map by $\Uc$. This map $\Uc$ assigns, to each rooted loop $\eta$, its equivalence class under rerooting, where two loops $\eta$ and $\eta'$ are considered as equivalent if they have the same length, and if, denoting $j = |\eta| + 1 = |\eta'| + 1$, we have that for some $k \in \{0, \ldots, j-1\}$, $\eta'_i = \eta_{i+k}$ for all $i = 1, \ldots, j$ (where of course, indices are meant modulo $j$). Later, we always use the terminology \emph{loop} for unrooted loops, unless this is specified explicitly. A single vertex can also be thought of as a loop with length $0$, and we call it a trivial loop. The set of all such loops is denoted by $\Zc^2$, and we use the notation $\Gamma$ for the set of non-trivial loops, i.e.\ with length at least $1$.

The box with radius $n \geq 0$ (for $d_{\infty}$) centered on a given vertex $z \in \Zb^2$, which has side length $2n$, is denoted by $B_n(z) := z + [-n,n]^2 \cap \Zb^2$ ($= \{z' \in \Zb^2 : d_{\infty}(z,z') \le n\}$). For $0<n<N$, the annulus around $z$ with radii $n$ and $N$ is defined as $A_{n,N}(z) := B_N(z) \setminus B_{n-1}(z)$ ($= \{ z' \in \Zb^2 : n \le d_{\infty}(z,z') \le N \}$). We remove $z$ from these notations if $z=0$, i.e., we simply write $B_n=B_n(0)$ and $A_{n,N}=A_{n,N}(0)$.

Finally, let us adopt some more notation with functions and constants. If $f$ and $g$ are two positive functions, $f \lesssim g$ means that $f \leq C g$ for some constant $C > 0$. If we need to allow the constant $C$ to depend on another parameter, say $\alpha$, we indicate it explicitly by writing $f \lesssim_{\alpha} g$. We use the letters $c$ and $C$ (or $c', c'', \ldots$) to denote arbitrary constant which may change from one line to the next. On the other hand, we use subscripts to denote particular constants which remain fixed, such as $c_0, c_1, \ldots$, and we always specify it when a constant depends on other parameter, writing for example $c_2(\alpha)$.

\subsection{Random walk loop soup} \label{subsec:RWLS}

\subsubsection*{Definition}

We now define precisely the random walk loop soup. For this purpose, we use the same notations as in the companion paper \cite{GNQ2024a}, that we recall briefly. We remind the reader that the sets of trivial and non-trivial loops are denoted by $\Zc^2$ and $\Gamma$, respectively. In the present paper, a loop configuration is by definition of the form $L=L' \cup \Zc^2$, where $L'\in (\Nb_0)^{\Gamma}$ is a multiset of $\Gamma$, and we denote by $n_L(\gamma)$ the number of occurrences of $\gamma \in \Gamma \cup \Zc^2$ in $L$. The union of two loop configurations $L_1$ and $L_2$ is denoted by $L:=L_1 \uplus L_2$, with $n_L(\gamma)=1$ for all $\gamma \in \Zc^2$, and $n_L(\gamma)=n_{L_1}(\gamma)+n_{L_2}(\gamma)$ for $\gamma \in \Gamma$. Sometimes, we also need to consider the difference between two configurations, and we let $L:=L_1\setminus L_2$ be the loop configuration so that, again, $n_L(\gamma)=1$ for $\gamma \in \Zc^2$, while $n_L(\gamma)=(n_{L_1}(\gamma)-n_{L_2}(\gamma))\vee 0$ if $\gamma \in \Gamma$. For any subset $D \subseteq \Zb^2$, we denote by $L_D$ the set of all loops in $L$ which are contained in $D$. For a finite $A \subseteq \Zb^2$, we write $L(A)$ for the set of non-trivial loops in $L$ that visit \emph{all} the vertices in $A$. For a vertex $z \in \Zb^2$, we abbreviate $L(\{z\})$ (the set of loops visiting $z$) simply as $L^z$.

Next, we consider the measure $\mu$ on paths, defined by $\mu(\eta) = 4^{-|\eta|}$ if $\eta$ has length at least $1$, and $\mu(\eta) = 0$ if $\eta$ is trivial (i.e.\ it has length $0$). For any two vertices $z,z'\in \Zb^2$, we denote by $\mu^{z,z'}$ the restriction of $\mu$ to the paths whose starting and ending vertices are $z$ and $z'$, respectively. We also define the measure $\mu_0$ on unrooted loops, such that if the loop $\gamma$ is non-trivial, $\mu_0(\gamma) = 4^{-|\gamma|}$ (and $\mu_0(\gamma) = 0$ otherwise). The \emph{unrooted loop measure} is then defined as 
\begin{equation}\label{eq:nu}
	\nu(\gamma):=\frac{\mu_0(\gamma)}{J(\gamma)},
\end{equation}
where $J(\gamma)$ is the \emph{multiplicity} of the (non-trivial, unrooted) loop $\gamma$ (i.e.\ $J(\gamma)$ is the largest integer $j$ such that $\gamma$ can be decomposed as $j$ successive copies of the same loop). This measure plays a central role, since the definition of the \emph{random walk loop soup} (abbreviated as RWLS most of the time) relies on it, as we recall next.

For any $\alpha > 0$, we define the RWLS in $\Zb^2$ of intensity $\alpha$ to be the random loop configuration $\Lc$ obtained as a Poisson point process with intensity $\alpha \nu$, together with all trivial loops. 
In some works (e.g.\ in \cite{LTF2007}), the RWLS is defined without trivial loops. The trivial loops have no influence on the connected components formed by the loops, but it is the version with trivial loops that generates an occupation field which is a permanental field \cite{MR2815763, MR3298468}, and corresponds to the Gaussian free field  at $\alpha=1/2$ (see Theorem~\ref{thm:Le-Jan}).
For future use, note that for any $k$ distinct non-trivial loops $\gamma_1,\ldots,\gamma_k$, we have by definition the following formula:
\begin{equation}\label{eq:rwls}
	\Pb( n_{\Lc}(\gamma_1)=n_1,\ldots, n_{\Lc}(\gamma_k)=n_k )=\prod_{i=1}^{k} e^{-\alpha \nu(\gamma_i)} \frac{(\alpha \nu(\gamma_i))^{n_i}}{n_i !}.
\end{equation}
We can transfer, in a natural way, all the earlier notations for a given loop configuration $L$ to the RWLS, just replacing $L$ by $\Lc$. In particular, for a subset $D \subseteq \Zb^2$, $\Lc_D$ contains all the loops in $\Lc$ which remain in $D$, and we call it the \emph{random walk loop soup in $D$}.

\subsubsection*{Occupation time field}

We now describe the occupation time field associated with a loop configuration. For each non-trivial loop in $\Gamma$, we assign an $\mathrm{Exp}(1)$ distributed waiting time to each of its jumps. For each trivial loop in $\Zc^2$, we assign a waiting time with law $\Gamma(\alpha,1)$ to the vertex where it stays. All these waiting times are assumed to be independent from each other and from the loop configuration. Then, for any loop configuration $L$, the \emph{occupation time field} (or simply occupation field) $X_L$ for $L$ is defined, at each $z\in \Zb^2$, as the total waiting time at $z$ produced by all the loops in $L$ (i.e., non-trivial and trivial ones). For any bounded subset $D\subseteq \Zb^2$, we consider the occupation field $X_D: D\mapsto [0,\infty)$ for the RWLS $\Lc_D$ of intensity $\alpha$, which is defined by $X_D(z):=X_{\Lc_D}(z)$ for all $z\in D$.

From standard properties, the occupation field $X_D$ can be sampled by the following procedure.
\begin{enumerate}[(1)]
	\item We first sample the RWLS $\Lc_D$, which produces the total number of visits to each vertex $z\in D$ by all the non-trivial loops in $\Lc_D$:
	\begin{equation}\label{eq:disc-occup}
		n(z):=\sum_{\gamma\in\Lc_D\cap\Gamma} n_{\gamma}(z), \quad z \in D,
	\end{equation}
	where $n_{\gamma}(z)$ is the number of visits to $z$ by the loop $\gamma$.
	\item Conditioned on $( n(z) )_{z\in D}$, we sample $t(z)$ for each $z \in D$ independently, with law $\Gamma(n(z)+\alpha,1)$. 
	\item In this way, we have defined the joint law of $((n(z),t(z)))_{z\in D}$: the marginal distribution of $t$ is then the same as that of $X_D$.
\end{enumerate}
It was shown in \cite{MR2815763, MR3298468} that $X_D$ is a certain permanental field, and its Laplace transform has been computed in \cite[Corollary 1]{MR2815763}.  Some related results are listed below Theorem~\ref{thm:main}.

\subsubsection*{Connected components}

Finally, connectivity for loops in a loop configuration $L$ can be defined in a natural way. For two loops $\gamma$ and $\gamma'$, we write $\gamma \sim \gamma'$ if $\gamma$ and $\gamma'$ intersect each other, i.e., if there is at least one vertex which is visited by both loops. For any $j \geq 1$, a sequence of loops $\gamma_1, \cdots, \gamma_j$ such that $\gamma_i \sim \gamma_{i+1}$ for each $i = 1, \ldots, j-1$ is called a \emph{chain} of loops (with length $j$), and we say that two loops $\gamma$ and $\gamma'$ are \emph{connected} in $L$ if one can find such a chain of loops, with some length $j \geq 1$, such that $\gamma_1 = \gamma$ and $\gamma_j = \gamma'$. This produces obviously an equivalence relation, and we refer to the associated classes as \emph{connected components} or \emph{clusters} (of $L$). 
By a slight abuse of notation, we can regard a cluster either as a collection of loops, or as a subset of $\Zb^2$.
Note that clusters may be trivial, i.e.\ contain only a single loop in $\Zc^2$. For any $D\subseteq \Zb^2$, the clusters of $L_D$ induce a partition of $D$.  A cluster $\Cc$ of $L$ is said to be an \emph{outermost} cluster if there is no other cluster $\Cc'$ of $L$ such that $\Cc$ is included in $\mathrm{Fill}(\Cc')$. For each cluster $\Cc$, let $\mathrm{int}(\Cc):= \Cc  \setminus \partial_{\mathrm{ext}}\Cc$ be the ``interior'' of $\Cc$. Note that $\mathrm{int}(\Cc)$ is empty if $\Cc$ is trivial.

\begin{definition}[Carpet]\label{def:carpet}
	Let $L$ be a loop configuration and let $D\subseteq \Zb^2$. The \emph{carpet} of $L$ in $D$ is equal to $D\setminus \bigcup \mathrm{int} (\Cc)$, where the union is over clusters $\Cc$ of $L_D$. In other words, it is obtained as the union of $\partial_{\mathrm{ext}}\Cc$ over all \emph{outermost} clusters $\Cc$.
\end{definition}

\subsection{Gaussian free field}\label{subsec:GFF}

From now on, we will only consider the \emph{discrete} Gaussian free field. Hence, we use the abbreviation GFF (instead of DGFF) in the remainder of this paper, for simplicity.

\subsubsection*{Definition}

Consider a bounded subset $D$ of $\Zb^2$. Let $S$ be a simple random walk started from some vertex in $D$, and let $\zeta_{D}:=\min\{ i\ge 0: S_i \notin D \}$. The \emph{Green's function} in $D$ is defined as 
\begin{equation}\label{eq:Green}
	G_D(u,v):= \Eb^u\bigg( \sum_{i=0}^{\zeta_D} \ind_{\{S_i=v\}}  \bigg)  \quad \text{ for all } u,v\in D, 
\end{equation}
where $\Eb^u$ denotes the expectation for the random walk $S$ starting from $u$. The \emph{Gaussian free field} (GFF) in $D$ with $0$ boundary condition (along $\partial^{\mathrm{out}} D := \partial^{\mathrm{in}}(D^c)$) is a centered Gaussian process on $D$, which is denoted by $\varphi_D$, with covariance given by $G_D(\cdot,\cdot)$ and such that $\varphi_D(z) = 0$ for all $z \in D^c$.

\subsubsection*{Isomorphism theorems}
Le Jan's isomorphism theorem relates the occupation field for the RWLS at the critical intensity $\alpha=1/2$ to the square of the GFF in the following way. It will allow us to translate readily our percolation result for the occupation field into an analogous result for the (squared) GFF (see Section~\ref{subsec:toGFF}).

\begin{theorem}[Le Jan's isomorphism, \cite{MR2815763}]\label{thm:Le-Jan}
	Consider the occupation field $X_D = (X_D(z))_{z\in D}$ of a RWLS in $D$ with intensity $\alpha=1/2$, and let $\varphi_D$ be a GFF in $D$ (with $0$ boundary condition). Then $X_D$ has the same law as $\frac12 \varphi_D^2=\big( \frac12 \varphi_D(z)^2 \big)_{z\in D}$.
\end{theorem}

As we mentioned in the introduction, this result belongs to a more general class of isomorphism theorems, relating random walks and Brownian motions to the GFF. In particular, our percolation result for the squared GFF can be used to derive similar results for other processes, via another isomorphism theorem, that we mention now. For this purpose, we consider the modified graph obtained by contracting all vertices of $D^c$ into a single point, denoted by $\varrho$. Let $E$ be the set of edges in $D$, together with all the edges from $D$ to $D^c$ (each of these edges is now an edge from $\partial D$ to $\varrho$). Let $\Gc=(D\cup\{\varrho\},E)$ be the resulting connected graph (see e.g.\  \cite[Figure~1]{AB2022} for an illustration). Let $\wt S$ be the continuous-time simple random walk on $\Gc$ with jump rate (conductance) $\frac14$ along every edge in $E$. We use $\Pb^{\varrho}$ to denote the law of $\wt S$ when it starts from $\varrho$. For any $t > 0$, the \emph{local time up to time $t$} is given by, at each $z\in D\cup\{\varrho\}$,
\begin{equation}
	\ell^t_{D}(z):=\int_0^t \ind_{\{ \wt S_s=z \}} \, ds.
\end{equation}
Let $\varsigma(t):=\inf\{ s\ge0: \ell^s_{D}(\varrho)>t \}$ be the inverse local time at $\varrho$, and write 
\begin{equation}\label{eq:local-time}
	\Lfr^t_D(z):=\ell^{\varsigma(t)}_{D}(z), \quad z\in D \cup \{\varrho\}.
\end{equation}
The generalized second Ray-Knight theorem can be stated as follows. The precise form below is proved in  \cite{EKMRS2000}, but we also refer the reader to the references given in \cite{EKMRS2000, AB2022}.

\begin{theorem}[Generalized second Ray-Knight theorem]\label{thm:gsrk}
	Let $\varphi_D$ be a GFF in $D$, and let $\varphi'_D$ be a copy of it. For each $t>0$, if $\varphi'_D$ is independent of $\Lfr^t_D$, then 
	\[
	\Big(\Lfr^t_D(z)+\frac12 \varphi'_D(z)^2\Big)_{z\in D}\stackrel{d}{=}\Big(\frac12 ( \varphi_D(z) + \sqrt{2t} )^2\Big)_{z\in D}.
	\]
\end{theorem}
As a corollary of the above theorem, we get 
\begin{corollary}\label{cor:gsrk}
	Let $\varphi_D$ and $\Lfr^{\boldsymbol{\cdot}}_D$ be as above. For all $t>0$,
	\begin{equation}\label{eq:dom}
		\big(\Lfr^t_D(z)\,\big)_{z\in D} \preceq \Big(\frac12 ( \varphi_D(z) + \sqrt{2t} )^2 \Big)_{z\in D},
	\end{equation}
	where $\preceq$ denotes stochastic domination.
\end{corollary}

\subsubsection*{Estimate: Green's function}

Finally, we state an estimate for the Green's function in a ball of $\Zb^2$, which has proved to be useful in the companion paper \cite{GNQ2024a}. This result can be found for instance in \cite{La1991} (see Theorem~1.6.6 and Proposition~1.6.7 there).

\begin{lemma}\label{lem:green's}
	There exists a universal constant $c>0$ such that the following holds. For all $n \geq 1$,
	$$G_{B_n}(0,0)=\frac2\pi\log n+c+O(n^{-1}).$$
	Moreover, for all $x\in B_n\setminus\{0\}$, 
	$$G_{B_n}(x,0)=\frac2\pi(\log n-\log |x|)+O \big( |x|^{-1}+n^{-1} \big).$$
\end{lemma}

\section{Preliminary results about RWLS} \label{sec:tools_RWLS}

In this section, we collect results for the RWLS which are needed in future proofs.

\subsection{Classical tools}

We first state two standard results for random walk loop soups, that we already used repeatedly in the companion paper \cite{GNQ2024a}. We start with Palm's formula, which is also often called Mecke's equation (we refer the reader to \cite[Proposition 15]{MR2815763} or \cite[Theorem 4.1]{LP2017}, for example). It originates from the Poisson structure of the random ensemble.

\begin{lemma}[Palm's formula] \label{lem:Palm formula}
Consider the random walk loop soup $\Lc_{\alpha}$ in $\Zb^2$ with intensity $\alpha$, for any given $\alpha > 0$. If $\Phi$ is a bounded functional on loop configurations, and $F$ is an integrable function on loops, then we have
$$\mathbb{E}\left(\sum_{\gamma \in \mathcal{L}_\alpha} F(\gamma) \Phi\left(\mathcal{L}_\alpha\right)\right)=\int \mathbb{E}\left(\Phi\left(\mathcal{L}_\alpha \uplus\{\gamma\}\right)\right) \alpha F(\gamma) \nu(d \gamma).$$
Furthermore, the same result holds true for the loop soup restricted to any finite domain $D \subseteq \Zb^2$.
\end{lemma}

Next, we recall the FKG inequality, stating it in the specific setting of the RWLS (see e.g. \cite[Theorem 20.4]{LP2017}). This inequality is very common in models from statistical mechanics, and holds in a much wider generality (e.g. for general Poisson point processes, and Bernoulli percolation). We remind the reader that by definition, a function $f$ on the space of loop configurations is \emph{increasing} (resp. decreasing) if: for all $L$ and $L'$, $L \subseteq L'$ implies $f(L) \le f(L')$ (resp. $f(L) \ge f(L')$).

\begin{lemma}[FKG inequality] \label{lem:FKG-RWLS}
For any $\alpha > 0$, consider the random walk loop soup in $\Zb^2$ with intensity $\alpha$, and let $f,g\in L^2(\Pb)$ be two increasing functions on loop configurations. Then,
$$\Eb(fg)\ge \Eb(f)\cdot \Eb(g).$$
\end{lemma}

\subsection{Properties of arm events} \label{sec:prop_four_arm}

We now summarize the results developed in our work on arm events in the random walk loop soup \cite{GNQ2024a}, which play a key role in our later proofs.

The values of the arm exponents, given below in \eqref{eq:arm_exponent}, \eqref{eq:arm_exponent_bdy} and \eqref{eq:2_arm_exponent_bdy}, 
are equal to the well-known SLE$_\kappa$ arm exponents, which were first computed in \cite{SW2001} in 2001 (for the $\kappa=6$ case, due to the particular interest in percolation). The proof in \cite{SW2001} was in fact based on a series of breakthroughs \cite{LSW2001a,LSW2001b,MR1899232}, which computed many other closely related exponents for SLE$_\kappa$ for all $\kappa\in(0,8)$. A derivation of the SLE$_\kappa$ arm exponents through the KPZ relations can be found in \cite{MR2112128,MR2581884} (and see the references therein). In 2018, a proof was also written down in \cite{MR3846840}.

There is some non-trivial work to relate the SLE$_\kappa$ arm exponents to the arm exponents for the random walk loop soup, as we explained in Section~\ref{subsubsec:arm}. This is done in our companion paper \cite{GNQ2024a}, which also uses the input from \cite{GNQ2024c}.

Throughout this section, we let $\alpha\in (0,1/2]$, and we use $\Lc_D$ to denote a RWLS in the domain $D$ with intensity $\alpha$. We say that a connected subset $A \subseteq \Zb^2$ crosses a given annulus $A_{l,d}(z)$, or the semi-annulus $A_{l,d}^+(z)$ (recall that it is $A_{l,d}(z) \cap \Hb$, by definition), if $A$ intersects both $\partial B_l(z)$ and $\partial B_d(z)$.

\subsubsection*{Interior four-arm events}

The interior four-arm event is defined as follows.

\begin{definition}[Interior four-arm event]
Let $\Lc$ be a random loop configuration. The \emph{interior four-arm event} in the annulus $A_{l,d}(z)$, for any $0<l<d$ and $z \in \Zb^2$, is defined as $\Ac_{\Lc}(z;l,d) := \{$there exist (at least) two outermost clusters in $\Lc$ which cross $A_{l,d}(z)\}$. Furthermore, for any subset $D \subseteq \Zb^2$, we denote $\Ac_D(z;l,d) := \Ac_{\Lc_D}(z;l,d)$, and in particular $\Ac_{\loc}(z;l,d) := \Ac_{B_{2d}(z)}(z;l,d)$. As before, we do not specify $z$ in these notations when $z=0$.
\end{definition}

This last event can be viewed as a ``localized'' version of the four-arm event, where large loops, coming from ``outside the box $B_{2d}(z)$'', are simply discarded.

For $\kappa\in (8/3,4]$, we define the \emph{interior four-arm exponent}
\begin{equation} \label{eq:arm_exponent}
\eta(\kappa) := \frac{(12-\kappa)(\kappa+4)}{8\kappa}.
\end{equation}
Note that $\eta(\kappa)>2$ for $\kappa\in(8/3,4)$, and $\eta(4)=2$. We can then define $\xi(\alpha) := \eta(\kappa)$ for all $\alpha \in (0,1/2]$, where $\alpha(\kappa)$ is given by \eqref{eq:alpha_kappa}, which  is an increasing bijection from $[8/3,4]$ to $[0,1/2]$.
Therefore, $\xi(\alpha)>2$ for $\alpha\in (0,1/2)$, and $\xi(1/2) = 2$.

We then have the following upper bound, which was established in \cite[Proposition~\proparmexp{}]{GNQ2024a}.

\begin{proposition} \label{prop:four_arm_proba}
For any $\alpha\in(0,1/2]$, and $\eps>0$, there exists $c_2(\alpha,\eps)>0$ such that: for all $0<d_1<d_2$, $z \in \Zb^2$, and $D\supseteq B_{2d_2}(z)$,
\begin{equation} \label{eq:arm_proba}
\Pb( \Ac_{D}(z; d_1,d_2) ) \le c_2\, (d_2/d_1)^{-\xi(\alpha)+\eps}.
\end{equation}
\end{proposition}

In the particular case of the critical intensity $\alpha=1/2$, the above four-arm estimate was recently strengthened in \cite[Corollary 1.3]{bi2025arm} by using different techniques, relying directly on the GFF. We recall it below.
\begin{proposition}\label{prop:critical four arm}
	For $\alpha=1/2$, there exists $c_1^*>0$ such that for all $z\in \Zb^2$, $0<d_1<d_2$, and $D \supseteq B_{2d_2}(z)$,
	$$\Pb( \Ac_D(z;d_1,d_2) ) \le c_1^*\, (d_2/d_1)^{-2}.$$
\end{proposition}

\subsubsection*{Boundary four-arm events}

In our proofs, we also have to take care of four-arm events near the boundary of a domain. For this purpose, we introduce the boundary four-arm event as follows.

\begin{definition}[Boundary four-arm event]\label{def:bfa}
Let $\Lc$ be a random loop configuration. The \emph{boundary four-arm event} in the semi-annulus $A_{l,d}^+(z)$, for any $0<l<d$ and $z = (x,0)$, is defined as $\Ac^+_{\Lc}(z;l,d) := \Ac_{\Lc_{\Hb}}(z;l,d) = \{$there exist (at least) two outermost clusters in $\Lc_\Hb$ which cross $A_{l,d}^+(z)\}$. Furthermore, for any subset $D \subseteq \Zb^2$, we write $\Ac^+_D(z;l,d) := \Ac^+_{\Lc_D}(z;l,d) = \Ac_{\Lc_{D^+}}(z;l,d)$, and in particular $\Ac^+_{\loc}(z;l,d) := \Ac^+_{B_{2d}(z)}(z;l,d)$. As always, we omit $z$ from these notations in the case $z=0$.
\end{definition}

We now define the \emph{boundary four-arm exponent}. For $\kappa \in (8/3,4]$, let 
\begin{equation} \label{eq:arm_exponent_bdy}
\eta^+(\kappa) := \frac{2 (12-\kappa)}{\kappa}.
\end{equation}
Note that $\eta^+(\kappa) \ge 4$ for $\kappa \in (8/3,4]$. We then define $\xi^+(\alpha) := \eta^+(\kappa)$, where, as in the interior case, $\alpha$ and $\kappa$ are related through \eqref{eq:alpha_kappa}. In particular, $\xi^+(\alpha) \ge 4$ for $\alpha\in (0,1/2]$. 

In the boundary four-arm case, the upper bound below holds true (see  \cite[Theorem~\proparmexpbdy{}]{GNQ2024a}).

\begin{proposition}\label{prop:four_arm_proba_bdy}
For any $\alpha\in (0,1/2]$, and $\eps>0$, there exists $c_2^+(\alpha,\eps)>0$ such that: for all $0<d_1<d_2$, $z = (x,0)$, and $D\supseteq B_{2d_2}(z)$,
\begin{equation} \label{eq:arm_proba_bdy}
\Pb( \Ac^+_{D}(z; d_1,d_2) ) \le c^+_2 \, (d_2/d_1)^{-\xi^+(\alpha)+\eps}.
\end{equation}
\end{proposition}

\subsubsection*{Boundary two-arm events}

Later, when we consider boundary crossing probabilities, we make also use of the boundary two-arm event. We collect the results that we need below.

\begin{definition}[Boundary two-arm event]\label{def:bta}
	For any $0<l<d$ and $D\supseteq B_{2d}$, let $\Bc_D(l,d)$ be the \emph{boundary two-arm event} that there is a cluster in $\Lc_{D^+}$ crossing $A_{l,d}^+$. 
\end{definition}

We now define the \emph{boundary two-arm exponent}. For $\kappa \in (8/3,4]$, let 
\begin{equation} \label{eq:2_arm_exponent_bdy}
	\eta^{2,+}(\kappa) := \frac{8}{\kappa}-1.
\end{equation}
We define $\xi^{2,+}(\alpha) := \eta^{2,+}(\kappa)$ with $\alpha=\alpha(\kappa)$ via \eqref{eq:alpha_kappa}.
By Theorem~\proparmexpbdy{} in \cite{GNQ2024a}, we have 
\begin{proposition}\label{prop:2_arm_proba_bdy}
	For any $\alpha\in (0,1/2]$, and $\eps>0$, there exists $\bar c_2(\alpha,\eps)>0$ such that for all $0<d_1<d_2$ and $D\supseteq B_{2d}$,
	\begin{equation} \label{eq:2_arm_proba_bdy}
		\Pb( \Bc_{D}(d_1,d_2) ) \le \bar c_2 \, (d_2/d_1)^{-\xi^{2,+}(\alpha)+\eps}.
	\end{equation}
\end{proposition}

The above estimate was also strengthened in \cite[Corollary 1.3]{bi2025arm} for the critical intensity, as follows.
\begin{proposition}\label{prop:critical 2_arm_proba_bdy}
	For $\alpha=1/2$, there exists $c_2^*>0$ such that for all $0<d_1<d_2$ and $D\supseteq B_{2d}$,
	\begin{equation} 
		\Pb( \Bc_{D}(d_1,d_2) ) \le c_2^* \, (d_2/d_1)^{-1}.
	\end{equation}
\end{proposition}

\subsubsection*{Other useful tools}

We conclude this section by mentioning a few important intermediate results which were established in \cite{GNQ2024a}, and led in particular to the upper bounds \eqref{eq:arm_proba} and \eqref{eq:arm_proba_bdy} above. We state them here because for our proofs in the present paper, they turn out to be also useful in themselves (as well as the arguments to show them).

First, we define the \emph{truncated} four-arm event $\overrightarrow\Ac_{\Lc}(z;l,d) := \Ac_{\Lc}(z;l,d) \cap \{ \Lambda(B_l(z), \Lc) \subseteq B_{2d}(z) \}$, where $\Lambda(B_l(z), \Lc)$ denotes the union of $B_l(z)$ and all clusters in $\Lc$ intersecting that ball. In a similar way as earlier, we write $\overrightarrow\Ac_D(z;l,d) := \overrightarrow\Ac_{\Lc_D}(z;l,d)$. The following result was established in \cite[Proposition~\proplocality{}]{GNQ2024a}.

\begin{proposition}[Locality] \label{prop:locality}
There is a universal constant $C>0$ such that for all $z\in\Zb^2$, $1\le l\le d/2$, $D \supseteq B_{2d}(z)$, and any intensity $\alpha\in (0,1/2]$,
$$\Pb(\Ac_D(z;l,d))\le C\, \Pb( \overrightarrow\Ac_D(z;l,d) ),$$
and thus,
$$\Pb(\Ac_D(z;l,d)) \le C\, \Pb(\Ac_{\loc}(z;l,d)).$$
\end{proposition}

We also established an analogous result in the reversed direction, for the ``inward'' arm event defined as $\overleftarrow \Ac_{\Lc}(z; l,d) := \Ac_{\Lc}(z; l,d) \cap \{ \Lambda(\partial B_d(z), \Lc) \subseteq B_{l/2}(z)^c \}$, for any $z\in \Zb^2$, $1<l<d$, and any random loop configuration $\Lc$ (with a similar notation for $\Lambda(., .)$ as for the truncated four-arm event). As before, we write $\overleftarrow\Ac_D(z;l,d):=\overleftarrow\Ac_{\Lc_D}(z;l,d)$. The result below was established in \cite{GNQ2024a} (Proposition~\proplocalityin{} there).

\begin{proposition}[Inward locality] \label{prop:in-locality}
There is a universal constant $C>0$ such that for all $z\in\Zb^2$, $2\le l\le d/2$, $D \supseteq B_{d}(z)$, and any intensity $\alpha\in (0,1/2]$,
$$\Pb(\Ac_D(z; l,d))\le C\, \Pb(\overleftarrow\Ac_D(z; l,d)).$$
\end{proposition}

Next, we proved a quasi-multiplicativity upper bound (Proposition~\lemquasimult{} in \cite{GNQ2024a}).

\begin{lemma}[Quasi-multiplicativity] \label{lem:quasi}
For any $\alpha\in (0,1/2]$, there exists a constant $c_1(\alpha)>0$ such that for all $z\in \Zb^2$, $1\le d_1\le d_2/2\le d_3/16$, and $D \supseteq B_{2d_3}(z)$,
$$\Pb( \Ac_{D}(z; d_1,d_3) ) \le c_1\, \Pb( \Ac_{\loc}(z; d_1,d_2) )\, \Pb( \Ac_{D}(z; 4d_2,d_3) ).$$
\end{lemma}

Finally, all the results that we just stated for the interior four-arm events are also valid for the boundary two-arm and four-arm events, thanks to some minor modifications in the proofs; see Proposition~\proplocalityb{} and Proposition~\propquasin{} in \cite{GNQ2024a}.

\subsection{Further results}

For future use, we finally derive some further estimates on loops visiting given vertices, and the occupation field of the RWLS.

Recall that for any two vertices, $z \sim^* z'$ means that $|z-z'|_{\infty} = 1$, and that $\Lc_D(A)$ denotes the set of \emph{non-trivial} loops in $\Lc_D$ that visit all the points in $A$.
For any $d>0$, we denote $\Lc_d:=\Lc_{B_d}$.
We show the following result regarding the occurrence of loops that visit some predetermined vertices.

\begin{lemma}\label{lem:0w}
	There exists a universal constant $C>1$ such that for all $\alpha>0$, $w \in \Zb^2$ with $w \sim^* 0$, and $d\ge 2$,
	\[
	\Pb( \Lc_d({0,w})=\emptyset ) \le C^{\alpha}\, \Pb( \Lc_d(0)=\emptyset ).
	\] 
\end{lemma}

\begin{proof}
	By  \cite[Proposition~18]{MR2815763}, we have 
	\[
	\Pb( \Lc_d(0)=\emptyset )=G_{B_d}(0,0)^{-\alpha},
	\]
	and by \cite[Eq.~(4.9)]{MR2815763}, we have
	\[
	\Pb( \Lc_d({0,w})=\emptyset ) = \left(1-\frac{G_{B_d}(0,w)^2}{G_{B_d}(0,0)G_{B_d}(w,w)}\right)^{\alpha}.
	\]
	Let $u,v \in \Zb^2$. For a simple random walk in $\Zb^2$ started from $u$, let $p^d_{u,v}$ be the probability that it visits $v$ before hitting $B_d^c$.
	By Lemma~4.6.1 of \cite{MR2677157}, we have
	\[
	G_{B_d}(0,w)=p^d_{0,w} \, G_{B_d}(w,w)=p^d_{w,0} \, G_{B_d}(0,0).
	\]
	Furthermore, by \cite[Proposition 1.6.7]{La1991},
	\[
	p^d_{w,0}=1+O((\log d)^{-1}) \quad \text{and} \quad p^d_{0,w}=1+O((\log d)^{-1}).
	\]
	Therefore, for some universal constants $c,C>0$,
	\[
		\Pb( \Lc_d({0,w})=\emptyset ) = (1-p^d_{0,w}\, p^d_{w,0})^{\alpha} \leq c^{\alpha}\, (\log d)^{-\alpha}\le C^{\alpha}\, G_{B_d}(0,0)^{-\alpha},
	\]
	where we used Lemma~\ref{lem:green's} in the last inequality. 
	This completes the proof.
\end{proof}

Next, we give a result on the occupation field. Recall that $X_D=X_{\Lc_D}$ is the occupation field of the random walk loop soup $\Lc_D$. One can easily see that $X_{B_d}(0)$ tends to infinity as $d\rightarrow\infty$. However, its distribution has an exponential upper bound if we consider the occupation at $0$ produced only by the loops avoiding some vertex near $0$. More precisely, for $w\sim^* 0$ and $\{0,w\}\subseteq D$, we let $X_D^w(0):=X_{\Lc_D\setminus \Lc_D^w}(0)$ be the occupation at $0$ created by all the loops in $\Lc_D$ that avoid $w$. Then, we have the following result.

\begin{lemma}[Exponential bound]
	\label{lem:Gauss}
	There exist two universal constants $c,C>0$ such that for all $\alpha\in (0,1/2]$, $w \in \Zb^2$ with $w \sim^* 0$, and $D \subseteq \Zb^2$ with $D \supseteq B_1$, we have
	\begin{equation}\label{eq:exp bound}
	\Eb \Big( e^{c \, X_D^w(0)} \Big) \le C.
	\end{equation}
\end{lemma}

\begin{proof}
	By monotonicity, we can assume that $\alpha=1/2$.
	For any loop $\gamma$, we denote by $n_0(\gamma)$ the number of visits to $0$ made by $\gamma$. Denote the total number of visits to $0$ by non-trivial loops in $\Lc_D$ that avoid $w$ by
	\[
	W_D^w(0):=\sum_{\gamma\in \Lc_D\cap\Gamma} n_0(\gamma) \ind_{\{w\notin \gamma\}}.
	\]
	Then by definition, 
	\begin{equation}\label{eq:conditional law}
		\text{conditioned on $W_D^w(0)$, the conditional law of $X_D^w(0)$ is given by $\Gamma(W_D^w(0)+1/2,1)$}.
	\end{equation}
	We first prove an exponential bound for $W_D^w(0)$, and we then conclude the proof by using \eqref{eq:conditional law}. 
	Using a standard property of Poisson point processes (see e.g.\ \cite[Exercise 3.4]{LP2017}), we have 
	\begin{equation}\label{eq:W-nu}
	\Eb \Big( e^{c \, W_D^w (0)} \Big) =\exp \Big( \int (e^{c\,n_0(\gamma)}-1)\,\ind_{\{w\notin \gamma,\gamma\subseteq D\}}\,\nu(d\gamma) \Big),
	\end{equation}
	where $\nu$ is the unrooted loop measure defined in \eqref{eq:nu}. Now, let $\nu_0$ be the measure on rooted loops which assigns $0$ to a loop if its root is not $0$, and assigns to each loop $l$ rooted at $0$, the weight $\mu_0(l)/n_0(l)$, where $\mu_0(l)=4^{-|l|}$ is extended to rooted loops in the natural way. Since $\nu(\gamma)=\sum_{\Uc(l)=\gamma}\nu_0(l)$,
	it is immediate that $\nu$ is the pushforward of $\nu_0$ under the unrooting map $\Uc$. Therefore, \eqref{eq:W-nu} remains true if we replace $\nu$ by $\nu_0$, where the integral is then over a different space of loops. In the remainder of the proof, we use $l$ to denote a loop rooted at $0$. We have
	\[
	\int (e^{c\,n_0(l)}-1)\,\ind_{\{w\notin l, l \subseteq D\}}\,\nu_0(dl)
	=\sum_{k=1}^{\infty}(e^{ck}-1)\, k^{-1}\,\mu_0(\{ l: n_0(l)=k, w\notin l, l \subseteq D \}).
	\]
	Note that $\mu_0(\{ l: n_0(l)=k, w\notin l, l \subseteq D \})=q^k$, with
	\[
	q:=\mu_0(\{ l: n_0(l)=1, w\notin l, l \subseteq D \})
	=\Pb(\text{$S$ returns to $0$ before visiting $w$ or exiting $D$}),
	\]
	where $S$ is a simple random walk starting from $0$.
	Since $w\sim^* 0$ and $B_1\subseteq D$, we have $q \le 7/8$. Combining the above observations, we get that for all $c<\log (8/7)$,
	\begin{equation}\label{eq:WD}
	\Eb \Big( e^{c \, W_D^w (0)} \Big) \le \exp \bigg( \sum_{k=1}^{\infty}(e^{ck}-1)\, k^{-1}\, (7/8)^k \bigg)\le c',
	\end{equation}
for some constant $c' = c'(c) < \infty$. Thus, we have proved \eqref{eq:exp bound}, but with $W_D^w(0)$ instead of $X_D^w(0)$.
	Finally, by \eqref{eq:conditional law} and the explicit formula for the moment generating function of the Gamma distribution, we obtain that
	\begin{align*}
	\Eb \Big( e^{c \, X_D^w(0)} \Big)= \Eb \Big(\Eb \Big( e^{c \, X_D^w(0)} \mid W_D^w(0) \Big) \Big)=\Eb \Big( (1-c)^{-W_D^w(0)-1/2} \Big).
	\end{align*}
    This, combined with \eqref{eq:WD}, concludes the proof of the lemma (for all $c > 0$ sufficiently small so that $1/(1-c) < 8/7$).
\end{proof}

\section{Absence of $*$-crossings with high occupation}\label{sec:pf-thm}

In this section, we establish the first of our main results, Theorem~\ref{thm:main}, by showing the non-existence, in a subcritical loop soup, of a $*$-crossing with high occupation time. 
In reality, we show the non-existence of a $*$-crossing with high occupation time at the passage edges between clusters, which will imply the stronger result Theorem~\ref{thm:carpet1}.
For this purpose, we start by defining some well-suited modified ``four-arm'' events (in the interior case) in Section~\ref{subsec:moe}. Such arm events would arise along a high $*$-crossing, as we explained informally in the introduction (see Figure~\ref{fig:blocking_path}). We then derive an upper bound for their probability in Section~\ref{subsec:moe_upper}, based on an induction argument inspired by the proof of \cite[Theorem~5.3]{BN2022} (this result showed the stability of modified six-arm events in the case of forest fires, arising from configurations such as that on Figure~\ref{fig:six_arms_FF}). We also need to consider boundary analogs of these events, which we do in Section~\ref{subsec:boundary}. Finally, in Section~\ref{subsec:main} we combine the estimates obtained on arm probabilities, in both the interior and the boundary cases, to prove Theorem~\ref{thm:main}.

\subsection*{Notations for arm probabilities}

We first set some notations, which are used in later computations. As we observed in Section~\ref{sec:prop_four_arm}, $\xi(\alpha)>2$ for all $\alpha \in (0,1/2)$. We consider any such $\alpha$ in the remainder of Section~\ref{sec:pf-thm}, and we let
\begin{equation} \label{eq:def_beta}
	\beta(\alpha) := \bigg( \frac{\xi(\alpha) - 2}{2}\bigg) \wedge \frac{1}{2} \in (0,1),
\end{equation}
so that $2 < 2+\beta(\alpha) < \xi(\alpha)$. Then by Proposition~\ref{prop:four_arm_proba}, for all $\alpha\in (0,1/2)$, there exists a positive constant $c_2(\alpha)$ such that for all $0<d_1<d_2$, $z\in\Zb^2$, and $D\supseteq B_{2d_2}(z)$,
\begin{equation}\label{eq:beta1}
	\Pb( \Ac_{D}(z;d_1,d_2) ) \le c_2 (d_2/d_1)^{-2-\beta}.
\end{equation}
Hence, we denote 
\begin{equation} \label{eq:beta2}
	\pi(d_1,d_2) = \pi_{\alpha}(d_1,d_2) := (d_2/d_1)^{-2-\beta}.
\end{equation}

In the boundary case, since $\xi^+(\alpha)\ge 4$ for all $\alpha\in (0,1/2]$, by choosing $\beta(\alpha) < 1$ in \eqref{eq:def_beta}, we have the following. By Proposition~\ref{prop:four_arm_proba_bdy}, for all $\alpha\in (0,1/2]$, there exists $c^+_2(\alpha)$ such that for all $z = (x,0)$, $0<d_1<d_2$, and $D \supseteq B_{2d_2}(z)$,
\begin{equation}\label{eq:beta3}
	\Pb( \Ac^+_{D}(z;d_1,d_2) ) \le c^+_2  (d_2/d_1)^{-3} \le c^+_2 \pi(d_1,d_2).
\end{equation}

\subsection{Modified arm events}\label{subsec:moe}

We now introduce a suitable version of arm events, but before that, we need some definitions, illustrated on Figure~\ref{fig:chain}. We then conclude this section by stating some estimates which will be useful later.

\begin{definition} [$*$- and $\lambda$-chains]
	Consider any loop configuration $L$.	
	\begin{enumerate}[(i)]
		\item We say that two distinct clusters $\Cc$ and $\Cc'$ are \emph{$*$-connected} if there exist $u\in\Cc$ and $v\in\Cc'$ such that $u \sim^* v$. In this case, the edge $(u,v)$ is called a \emph{$*$-passage edge} of $(\Cc,\Cc')$.
		A \emph{$*$-chain} in $L$, with some length $k \geq 1$, is a finite sequence $(\Cc_1, \ldots, \Cc_k)$ of distinct outermost clusters such that any two successive ones are $*$-connected. Note that such a $*$-chain might contain single-site clusters made by trivial loops.
		We say that the $*$-chain $(\Cc_1, \ldots, \Cc_k)$ \emph{crosses properly} the annulus $A_{d_1,d_2}(z)$ if $\Cc_1\cap \partial B_{d_1}(z)\neq\emptyset, \Cc_k \cap \partial B_{d_2}(z)\neq\emptyset$, and, when $k \geq 3$, $\Cc_i\subseteq \mathring{A}_{d_1,d_2}(z)$ for $2\le i\le k-1$.
		
		\item Let $X_{L}$ be the occupation field for $L$ (recall it from Section~\ref{subsec:RWLS}). For any $\lambda>0$, we say that two clusters $\Cc$ and $\Cc'$ are \emph{$\lambda$-connected} if there exists a $*$-passage edge $(u,v)$ of $(\Cc,\Cc')$ such that both $X_{L}(u)>\lambda$ and $X_{L}(v)>\lambda$. In this case, the edge $(u,v)$ is called a \emph{$\lambda$-passage edge} of $(\Cc,\Cc')$. A \emph{$\lambda$-connected chain}, or simply $\lambda$-chain for brevity, is a chain of distinct $\lambda$-connected outermost clusters in $L$. Moreover, we say that the sequence $(\Cc_1, \ldots, \Cc_k)$ is a \emph{$*\lambda*$-chain} if it is a $*$-chain, and, when $k\ge 4$, $(\Cc_2, \ldots, \Cc_{k-1})$ is a $\lambda$-chain. In the same way, we can define $*\lambda$-chains and $\lambda*$-chains.
	\end{enumerate}
\end{definition}

\begin{figure}
	\centering
	\includegraphics[width=.8\textwidth]{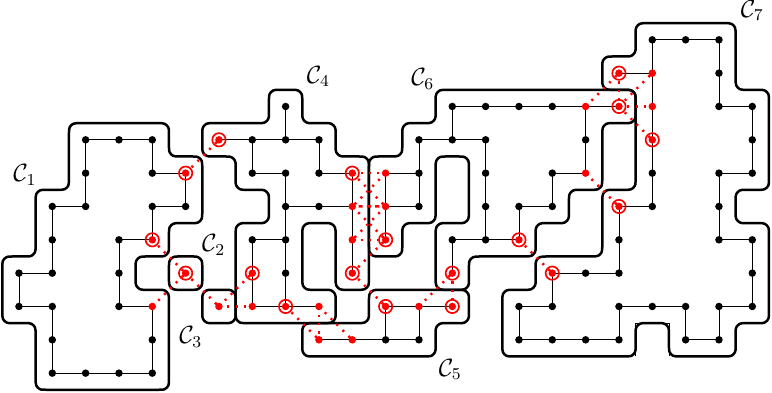}
	\caption{This figure depicts a $*$-chain $\Cc_1, \ldots, \Cc_7$: for all $i=1,\ldots,6$, $\Cc_i \sim^* \Cc_{i+1}$. The $*$-passage edges, $*$-connecting any two clusters, are shown in red. For any given $\lambda > 0$, each $*$-passage edge is also a potential $\lambda$-passage edge: for this to be the case, both of its endpoints need to have occupation time $> \lambda$ (vertices with that property are surrounded by a red circle on the figure). For instance, $\Cc_1, \Cc_4, \Cc_6, \Cc_7$ forms a $\lambda$-chain, since there exists at least one $\lambda$-passage edge between any two successive clusters in this sequence. Similarly, $\Cc_1, \Cc_4, \Cc_5, \Cc_6, \Cc_7$ is a $\lambda$-chain, but not $\Cc_1, \Cc_2, \Cc_3, \Cc_4, \Cc_6, \Cc_7$.}
	\label{fig:chain}
\end{figure}

We now give the definition of some modified arm events, namely, $*$- and $\lambda$-arm events (see Figure~\ref{fig:mae} for an illustration). Recall that for any set $A\subseteq \Zb^2$ and any loop configuration $L$, we denote by $L(A)$ the set of loops in $L$ which visit all of the vertices in $A$. Recall that we write $L^z:=L(\{z\})$ for $z \in \Zb^2$.

\begin{definition}[$*$- and $\lambda$-arm events]\label{def:mfa} Let $0< d_1<d_2$, $z\in \Zb^2$, and consider any loop configuration $L$.
	\begin{itemize}
		\item	The {\it $*$-arm event} $\Ac^*_{L}(z; d_1,d_2)$ is the event that there exist two disjoint $*$-chains of $L$ crossing properly $A_{d_1,d_2}(z)$.	
		\item For any $\lambda>0$, the {\it $\lambda$-arm event} $\Ac^{\lambda}_{L}(z; d_1,d_2)$ is the event that there exist two disjoint $*\lambda*$-chains of $L$ across $A_{d_1,d_2}(z)$.
		\item Let $\Ac^*_{L}(\dot{z}; d_1,d_2):=\Ac^*_{L\setminus L^z}(z; d_1,d_2)$ be the $*$-arm event \emph{avoiding the vertex $z$} (with a dot in the notation to emphasize that $z$ is avoided).
		Similarly, we let $\Ac^{\lambda}_L(\dot{z}; d_1,d_2):=\Ac^{\lambda}_{L\setminus L^z}(z; d_1,d_2)$ be the $\lambda$-arm event avoiding $z$.
	\end{itemize}
\end{definition}	

\begin{figure}[t]
	\centering
	\includegraphics[width=.6\textwidth]{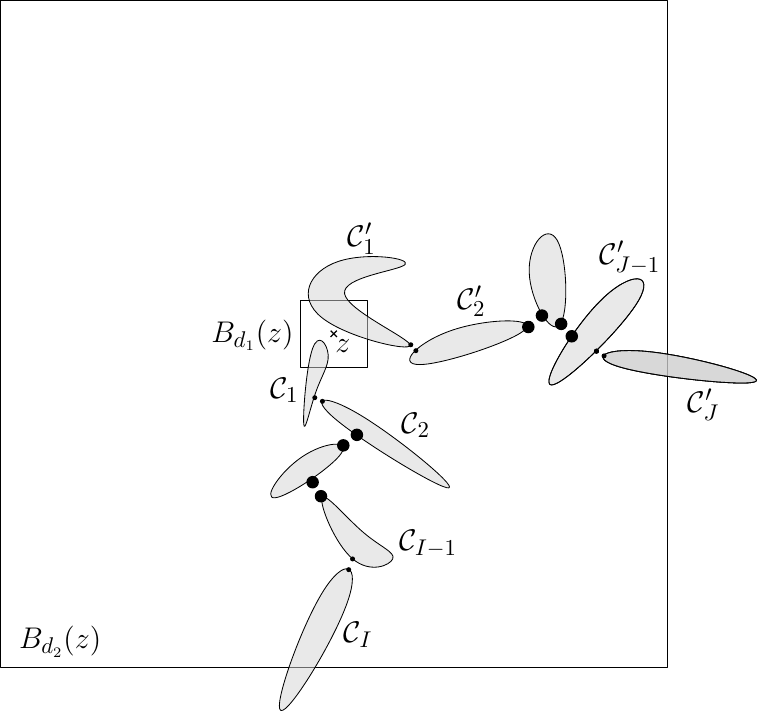}
	\caption{This figure illustrates the $\lambda$-arm event $\Ac^{\lambda}_L(z,d_1,d_2)$. 
		The two $*\lambda*$-chains $\Cc_1, \ldots, \Cc_I$ and $\Cc'_1, \ldots, \Cc'_J$ are composed of successive clusters shown in grey. The endpoints of passage edges are marked by dots, where the bigger ones are required to have occupation greater than $\lambda$.}
	\label{fig:mae}
\end{figure}

As before, we use the subscript $D$ to indicate that $L=\Lc_D$ for brevity, in all the notations that we just introduced. For example, we write $\Ac^*_{D}(z;d_1,d_2)=\Ac^*_{\Lc_D}(z;d_1,d_2)$. Further, as usual we omit $z$ if $z=0$, and we use $\loc$ to indicate the special case $D=B_{2d_2}(z)$, i.e.\ $\Ac_{\loc}^{\lambda}(z;d_1,d_2)=\Ac^{\lambda}_{B_{2d_2}(z)}(z;d_1,d_2)$.
We let $\Nc=\Nc'-2$ be the {\it connection number} of the associated $*$-arm event, where $\Nc'$ is the smallest possible value for the total number of clusters among all pairs of admissible $*$-chains (i.e.\ non-intersecting, and crossing properly the annulus). We define the connection number of the $\lambda$-arm event to be that of the corresponding $*$-arm event. 
When $\Nc=0$, the $*$-arm event reduces to the arm event introduced in earlier sections, that is, $\Ac^*_{D}(z;d_1,d_2)\cap\{ \Nc=0 \}=\Ac_{D}(z;d_1,d_2)$.

For the $*$-arm event, we denote by $\Pc^*$ the set of all $*$-passage edges for all admissible $*$-chains associated with it, and we call it the \emph{$*$-passage set}.
Similarly, for the $\lambda$-arm event,
we denote by $\Pc^{\lambda}$ the set of $\lambda$-passage edges for all admissible $*\lambda*$ chains associated with it, and we call it the \emph{$\lambda$-passage set}.

\begin{remark}\label{rem:modified arm event}
	We make the following observations about the definitions above, which will be useful later.
	\begin{enumerate}[(a)]
		\item For all $\lambda>0$, $\Ac^{\lambda}_D(z;d_1,d_2)\subseteq \Ac^*_D(z;d_1,d_2)$. Indeed, the $*$-arm event $\Ac^*_D(z;d_1,d_2)$ involves only the geometric structure of the loops, and it does not use any information about the occupation field.
		\label{rem:mae(a)}
		\item 	\label{rem:mae(b)}
		The $*$-arm event is monotone in the following sense:
		\[
		\Ac^{*}_D(z;d_1,d_2) \subseteq \Ac^{*}_D(z;d_3,d_4) \quad \text{ if } d_1\le d_3\le d_4\le d_2.
		\]
		This is also true for the $\lambda$-arm event.
		\item \label{rem:relax}
		In the definition of $\Ac^{\lambda}_D(z;d_1,d_2)$, we relax the condition on the first two and the last two clusters (i.e.\ they are required to be $*$-connected, rather than $\lambda$-connected).
		This relaxation allows us to modify easily the associated crossing loops (as in Sections~\seclocality{} and \secquasimult{} of \cite{GNQ2024a}) and ``pivotal'' loops (see the proof of Lemma~\ref{lem:mod-z}) to get locality (Lemma~\ref{lem:lambda-locality}), quasi-multiplicativity (Lemma~\ref{lem:quasi-lbd}) and decoupling (Lemma~\ref{lem:mod-z}) for the $\lambda$-arm event.
		\item \label{rem:modified arm event (iii)}
		By definition, the $\lambda$-arm event avoiding $z$ is independent of the loops that visit $z$. This observation, combined with the modification mentioned in (\ref{rem:relax}) just above, gives us a way to decouple the $\lambda$-arm event and the occupation field at $z$ (see Lemma~\ref{lem:mod-z} for details). This is the reason why we introduced the event $\Ac^{\lambda}_D(\dot{z};d_1,d_2)$.
		\item If the connection number of $\Ac^{\lambda}_D(z;d_1,d_2)$ is greater than $4$, i.e., $\Nc>4$, then the $\lambda$-passage set $\Pc^{\lambda}$ is not empty.
		\label{rem:mae(last)}
	\end{enumerate}
\end{remark}

In this section, we focus on $*$-arm events, and we postpone the analysis of $\lambda$-arm events to later sections. 
The following quasi-multiplicativity result for $*$-arm events can be proved in the same way as Lemma~\ref{lem:quasi} above, by suitably modifying crossing loops. Thus, we state the result without a proof.

\begin{lemma}[Quasi-multiplicativity for $*$-arm event]\label{lem:quasi*}
	For any $\alpha\in (0,1/2]$, there exists $c_3(\alpha)>0$ such that for all $z\in \Zb^2$, $1\le d_1\le d_2/2\le d_3/16$, and $D \supseteq B_{2d_3}(z)$,
	\begin{equation}\label{eq:quasi-*}
		\Pb( \Ac^{*}_D(z;d_1,d_3) ) \le c_3 \,\Pb( \Ac^*_{\loc}(z;d_1,d_2) )\, \Pb( \Ac^*_{D}(z;4d_2,d_3) ).
	\end{equation}
\end{lemma}

\begin{remark}
	In fact, a locality result of the same flavor as Proposition~\ref{prop:locality} holds for $*$-arm events, and one can even show that conditionally on $ \Ac^*_D(z;d_1,d_2)$, with positive probability all the $*$-chains intersecting $\partial B_{d_1}(z)$ are contained in $B_{2d_2}$. But this property will not be needed later in the paper.
\end{remark}

Thanks to the above quasi-multiplicativity property for $*$-arm events, we can get a simple a priori bound for the probability of $*$-arm events with fixed connection number. One can view such $*$-arm events as analogous to arm events with defects in Bernoulli percolation, the connection number being just the number of defects in the latter. Using a similar proof strategy as for arm events with defects, which was based on isolating the defects and applying quasi-multiplicativity (see Proposition $18$ in \cite{No2008}), we obtain the same upper bound with a logarithmic correction.

\begin{lemma}\label{lem:m4a}
	There exist $c>0$ such that for all $\alpha\in (0,1/2)$, $m\ge 1$, $z\in\Zb^2$, $1\le d_1\le d_2/2$ and $D \supseteq B_{2d_2}(z)$,
	\[
	\Pb( \Ac^*_D(z;d_1,d_2) \cap \{\Nc=m\} ) \le \Big(c\, c_2\, c_3\,\log \Big( \frac{d_2}{d_1} \Big) \Big)^m \pi(d_1,d_2),
	\]
	where $\pi$ and $c_2$ are provided by \eqref{eq:beta2} and \eqref{eq:beta1}, respectively. 
\end{lemma}

\begin{proof}[Proof of Lemma~\ref{lem:m4a}]
	It suffices to apply Lemma~\ref{lem:quasi*} repeatedly, at most $m$ times, around each annulus that contains a $*$-passage edge (along the two admissible $*$-chains that realize the minimum connection number $m$). In this way, we can decompose the $*$-arm event into successive arm events. Using \eqref{eq:beta1} and multiplying the corresponding probabilities of arm events then yields the result. Note that the multiplicative factor $( c'\, \log ( \frac{d_2}{d_1} ))^m$, for some universal $c'>0$, is an upper bound on the number of choices for the annuli where two clusters meet.
\end{proof}

Although this rough bound with logarithmic correction is in fact sufficient for our proofs in the present paper, we can get rid of the power of $\log(d_2/d_1)$ in the subcritical case, which we do now for potential use in the future. 
This improvement can be achieved by using more delicate arguments, taking into account the cost around each $*$-passage edge. To this end, we need the next result. Under the presence of a $*$-passage edge, it decomposes a given $*$-arm event into three independent such events in disjoint annuli  (see Figure~\ref{fig:decom} for a schematic representation of this decomposition). Since this quasi-multiplicativity result can also be proved in the same way as Lemma~\ref{lem:quasi}, we omit the proof again.

\begin{figure}[t]
	\centering
	\includegraphics[width=.75\textwidth]{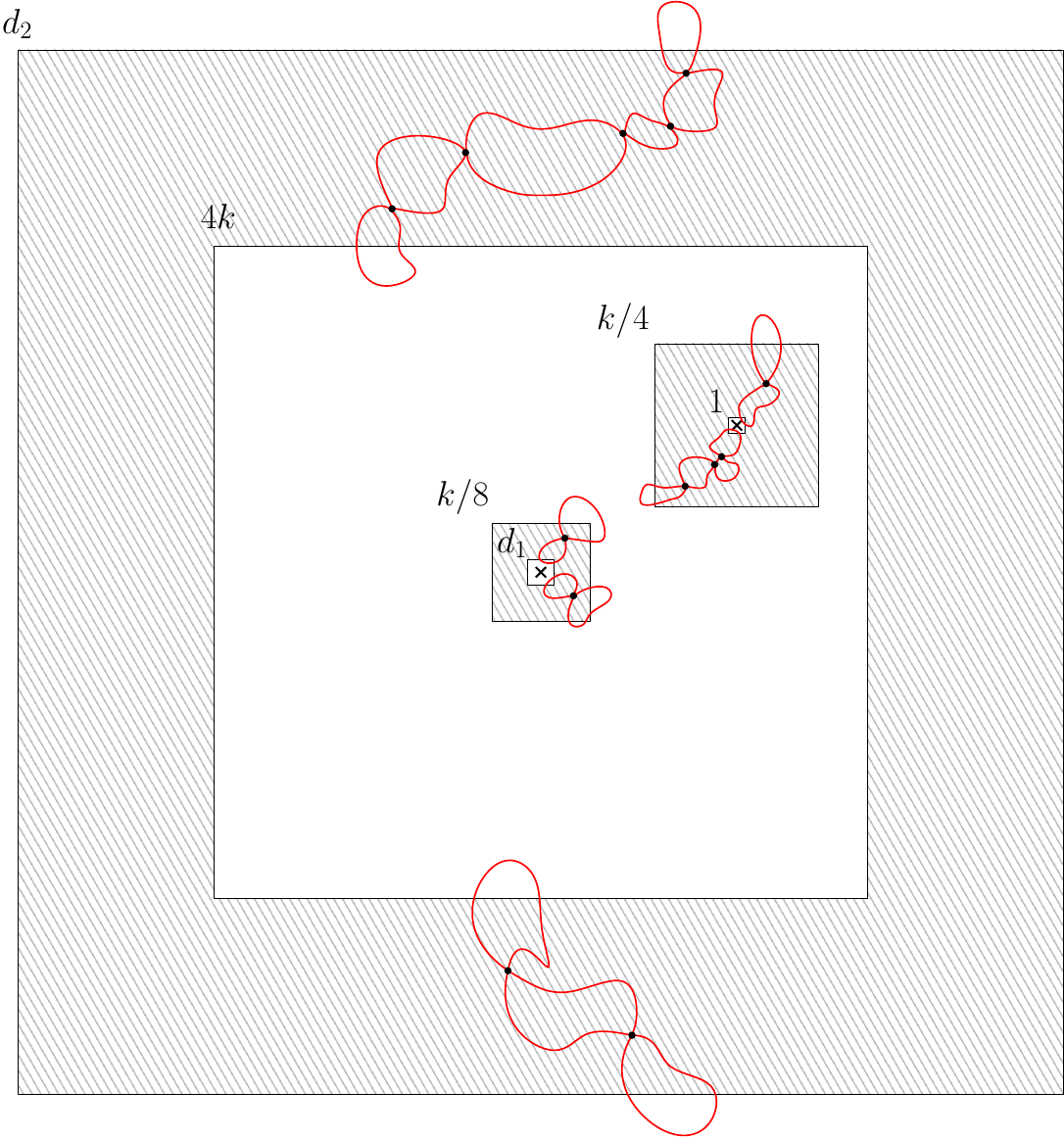}
	\caption{This figure shows the three arm events appearing in the decomposition provided by Lemma~\ref{lem:decom*}.}
	\label{fig:decom}
\end{figure}

\begin{lemma}\label{lem:decom*}
	For all $\alpha\in (0,1/2]$, there exists $c_4(\alpha)>0$ such that the following holds. Let $u,v\in \Zb^2$ with $|u|=k$ and $u\sim^* v$, and consider $0 < d_1 < d_2$ such that $16d_1\le k \le d_2/8$, and $D \supseteq B_{2d_2}$. Let $\Pc^*$ be the set of $*$-passage edges associated with $\Ac^{*}_D(d_1,d_2)$. Then, we have
	\begin{align}\notag
		\Pb( & \Ac^{*}_D(d_1,d_2) \cap \{ (u,v)\in\Pc^*  \} ) \\  \label{eq:decom*}
		& \quad \le c_4\,
		\Pb( \Ac^{*}_{\loc}(d_1,k/8) ) \, \Pb( \Ac^{*}_D(4k,d_2) ) \, \Pb( \Ac^{*}_{\loc}(u;1,k/4) ). 
	\end{align}
\end{lemma}

We are now in a position to improve the rough bound given in Lemma~\ref{lem:m4a}. The proof involves an induction argument akin to Theorem~5.3 of \cite{BN2022}. However, we will actually apply the induction on the connection number, while the induction is on the two scales $(d_1,d_2)$ in the proof of the referred theorem.
This two-scale induction in \cite[Theorem~5.3]{BN2022} is used later to prove Proposition~\ref{prop:mae} for $\lambda$-arm events. Note that the r.h.s. still depends on $m$. We will need to consider $\lambda$-arm events, for $\lambda$ sufficiently large, in order to get an upper bound which is in fact uniform in $m$.

\begin{proposition}\label{prop:m4a*}
	For all $\alpha\in (0,1/2)$, there exists $c_5(\alpha)>0$ such that 
	for all $m\ge 0$, $z\in\Zb^2$, $0< d_1< d_2$, and $D \supseteq B_{2d_2}(z)$,
	\begin{equation}\label{eq:mae-N}
		\Pb( \Ac^*_D(z;d_1,d_2) \cap \{\Nc\le m\} ) \le e^{c_5^{m+1}}\, \pi(d_1,d_2).
	\end{equation}
\end{proposition}

\begin{proof}[Proof of Proposition~\ref{prop:m4a*}]
	Note that by monotonicity (see (\ref{rem:mae(b)}) in Remark~\ref{rem:modified arm event} above), it suffices to show that \eqref{eq:mae-N} holds for $d_1=2^i$ and $d_2=2^j$, for any two integers $0\le i<j$. We prove it by induction on $m$. When $m=0$, the event in the l.h.s. of \eqref{eq:mae-N} reduces to the arm event $\Ac_D(z;d_1,d_2)$, so that by \eqref{eq:beta1},
	\begin{equation} \label{eq:*arm_m=0}
		\Pb( \Ac^*_D(z;d_1,d_2) \cap \{\Nc=0\} ) =\Pb( \Ac_D(z;d_1,d_2) ) \le c'_2\, \pi(d_1,d_2).
	\end{equation}
	Hence, \eqref{eq:mae-N} is verified for $m=0$ if $c_5$ is chosen so that $e^{c_5} \geq c'_2$.
	
	We now prove that the property holds for some $m\ge 1$, assuming that it is already verified for $m-1$. Let $E$ be the event that in the $*$-passage set $\Pc^*$ of $\Ac^*_D(z; 2^i, 2^j)$, there is no edge that intersects $A_{2^{i+4},2^{j-3}}(z)$. Then, by the union bound,
	\begin{align}\notag
		\Pb( \Ac^*_D(z; 2^i, 2^j) \cap \{\Nc\le m\} ) & \le 
		\Pb( \Ac^*_D(z; 2^i, 2^j) \cap E) \\ \notag
		&\hspace{-2cm} +\sum_{k=i+4}^{j-4} \sum_{u \in A_{2^k,2^{k+1}}(z)} \sum_{v: v \sim^* u}  \Pb( \Ac^*_D(z; 2^i, 2^j) \cap \{\Nc\le m\} \cap \{ (u,v)\in \Pc^* \} )\\
		\label{eq:I+II}
		&=: (I) + (II).
	\end{align}
	Note that 
	\begin{equation}\label{eq:(I)}
		(I)\le \Pb( \Ac_D(z;2^{i+4},2^{j-3}) )\le  c'_2 \, \pi( 2^{i+4},2^{j-3} ) = c'_2 2^{7(2+\beta)} \, \pi( 2^i,2^j ),
	\end{equation}
	where the second inequality uses \eqref{eq:beta1}. Next, we deal with $(II)$. Using Lemma~\ref{lem:decom*} to decompose the event under consideration into three $*$-arm events in disjoint annuli, and applying the induction hypothesis three times, to each of these three events, we obtain that 
	\begin{align}\notag
		(II)& \le \sum_{k=i+4}^{j-4} \sum_{u \in A_{2^k,2^{k+1}}(z)} \sum_{v: v \sim^* u} c_4\,
		\Pb(  \Ac^*_{\loc}(z;2^i,2^{k-3})\cap \{\Nc_1\le m-1\} ) \\\notag
		&\hspace{3cm} \cdot \Pb( \Ac^*_D(z;2^{k+2},2^j) \cap \{\Nc_2\le m-1\} ) \,
		\Pb( \Ac^*_{\loc}(u;1,2^{k-2})\cap \{\Nc_3\le m-1\} )\\ \label{eq:induction}
		& \le \sum_{k=i+4}^{j-4} \sum_{u \in A_{2^k,2^{k+1}}(z)} \sum_{v: v \sim^* u} c_4\,
		e^{c_5^{m}}\, \pi(2^i,2^{k-3}) \cdot e^{c_5^{m}}\, \pi(2^{k+2},2^j) \cdot e^{c_5^{m}}\, \pi(1,2^{k-2}),
	\end{align}
	where $\Nc_1,\Nc_2,\Nc_3$ are the connection numbers of the corresponding $*$-arm events, i.e.\ $\Ac^*_{\loc}(z;2^i,2^{k-3})$, $\Ac^*_D(z;2^{k+2},2^j)$ and $\Ac^*_{\loc}(u;1,2^{k-2})$. We remark that we have used the fact that our modification in the proof of Lemma~\ref{lem:decom*} does not increase the number of $*$-passage edges, which implies $\Nc_1,\Nc_2,\Nc_3\le m-1$ above.
	Finally, this yields
	\begin{equation}
		(II) \le 2^{5(2+\beta)} \, C_2 \, c_4\, e^{3c_5^{m}}\, \pi(2^i,2^{j}),
	\end{equation}
	using first that 
	$$\pi(2^i,2^{k-3})\, \pi(2^{k+2},2^j) = 2^{5(2+\beta)} \, \pi(2^i,2^{j}),$$
	and then that we have
	\begin{align}\label{eq:summation}
		\sum_{k=i+4}^{j-4} \sum_{u \in A_{2^k,2^{k+1}}(z)} \sum_{v: v \sim^* u}  \pi(1,2^{k-2}) 
		\le
		\sum_{k=i+4}^{j-4} C_1\, 2^{2k}\,2^{-(2+\beta)k}\le C_2,
	\end{align}
	for some universal constants $0 < C_1,C_2 < \infty$. 
	Thus, in order to make the induction work, we only need to choose a large enough constant $c_5$ such that 
	\[
	c'_2 \le e^{c_5} \:\: (\text{for $m=0$}), \quad \text{and for all $m \ge 1$,} \:\: c'_2 2^{7(2+\beta)} + 2^{5(2+\beta)} \, C_2\,c_4\, e^{3c_5^{m}}\le e^{c_5^{m+1}}.
	\]
	This finishes the proof of the lemma.
\end{proof}

We want to stress that in the proof of Proposition~\ref{prop:m4a*}, the fact that the associated exponent is strictly larger than $2$ is used in an essential way, cf. \eqref{eq:summation}. Thus, the arguments fail in the critical case $\alpha=1/2$, which provides us with the exponent $\xi(\alpha)=2$.
However, Lemma~\ref{lem:m4a} still holds true in that case, since its proof uses only the multiplicativity of the probability upper bound.

\subsection{Upper bound on  $\lambda$-arm events} \label{subsec:moe_upper}

In this key section, we start by showing some basic properties for $\lambda$-arm events, which are analogous to those for $*$-arm events. Then, in Proposition~\ref{prop:mae}, we prove that for all values of $\lambda$ large enough, the probability of the $\lambda$-arm event $\Ac^\lambda_D(z;d_1,d_2)$ ($z \in \Zb^2$, $0 < d_1 < d_2$) can be bounded by some constant multiple of $\pi(d_1,d_2)$. This is achieved by using an induction argument similar to the proof of Proposition~\ref{prop:m4a*}, while incorporating the cost given by excessive occupation at the ends of $\lambda$-passage edges.

\subsubsection{Preliminary results}

First, we note that it is straightforward to generalize the locality results for arm events to $\lambda$-arm events, in both directions, and also for avoiding versions. More specifically, for any $z$, $d_1$, $d_2$ as above, and any loop configuration $L$, we define $\overrightarrow\Ac_L^{\lambda}(z;d_1,d_2)$ (resp.\ $\overleftarrow\Ac_L^{\lambda}(z;d_1,d_2)$) as the event that $\Ac_L^{\lambda}(z;d_1,d_2)$ holds, and all the $*\lambda*$-chains of $L$ that intersect $\partial B_{d_1}(z)$ (resp.\ $\partial B_{d_2}(z)$) do not intersect $\partial B_{2d_2}(z)$ (resp.\ $\partial B_{d_1/2}(z)$). Moreover, we write $\overrightarrow\Ac_L^{\lambda}(\dot{z};d_1,d_2):=\overrightarrow\Ac_{L\setminus L^z}^{\lambda}(z;d_1,d_2)$, and $\overleftarrow\Ac_L^{\lambda}(\dot{z};d_1,d_2):=\overleftarrow\Ac_{L\setminus L^z}^{\lambda}(z;d_1,d_2)$. Finally, as usual, we use the subscript $D$ to indicate that the configuration of loops is $L=\Lc_D$.

With these definitions at hand, the locality results for $\Ac^\lambda_D(z;d_1,d_2)$ and $\Ac^\lambda_D(\dot{z};d_1,d_2)$ can then be obtained by using the same strategy of proof as that employed for Proposition~\ref{prop:locality} and Proposition~\ref{prop:in-locality}. Indeed, thanks to our relaxation on the first and the last clusters of the chains in the definition of $\lambda$-arm events, the modifications therein for crossing loops are still applicable here. Thus, we omit the details of the proofs.

\begin{lemma}[Locality for $\lambda$-arm events]
	\label{lem:lambda-locality}
	Uniformly in $\alpha\in (0,1/2]$, $\lambda>0$, $z\in\Zb^2$, $1\le d_1\le d_2/2$, and $D \supseteq B_{2d_2}(z)$, we have
	\begin{equation}
		\Pb(\Ac^\lambda_D(z;d_1,d_2)) \lesssim \Pb( \overrightarrow\Ac^\lambda_D(z;d_1,d_2)) \wedge \Pb( \overleftarrow\Ac^\lambda_D(z;d_1,d_2)),
	\end{equation}
	and 
	\begin{equation}
		\Pb(\Ac^\lambda_D(\dot{z};d_1,d_2)) \lesssim \Pb( \overrightarrow\Ac^\lambda_D(\dot{z};d_1,d_2)) \wedge \Pb( \overleftarrow\Ac^\lambda_D(\dot{z};d_1,d_2)).
	\end{equation}
\end{lemma} 

Using the same arguments as for Lemma~\ref{lem:quasi}, we can also obtain quasi-multiplicativity for $\lambda$-arm events.

\begin{lemma}[Quasi-multiplicativity for $\lambda$-arm events]\label{lem:quasi-lbd}
	For any $\alpha\in (0,1/2]$, there exists $c_6(\alpha)>0$ such that for all $\lambda>0$, $z\in \Zb^2$, $1\le d_1\le d_2/2\le d_3/16$, and $D \supseteq B_{2d_3}(z)$,
	\begin{equation}\label{eq:quasi-lbd}
		\Pb( \Ac^{\lambda}_D(z;d_1,d_3) ) \le c_6 \,\Pb( \Ac^\lambda_{\loc}(z;d_1,d_2) )\, \Pb( \Ac^\lambda_{D}(z;4d_2,d_3) ).
	\end{equation}
\end{lemma}

Using Lemma~\ref{lem:lambda-locality}, we can compare the probability of a $\lambda$-arm event to that of the corresponding avoiding one (recall that $\Lc^u_D$ is the set of loops in $\Lc_D$ visiting a given vertex $u$, and that we denote, for any $d > 0$, $\Lc_d = \Lc_{B_d}$).

\begin{lemma}\label{lem:delete-dot}
	For all $\alpha\in (0,1/2]$, $\lambda>0$, $u\in \Zb^2$, $8 \leq d_1 \le d_2/2$, and $D \supseteq B_{2d_2}(u)$,
	\[
	\Pb( \Ac^{\lambda}_{D}(\dot{u};d_1,d_2) )\, \Pb( \Lc^u_{D}\setminus \Lc_{d_1/2}^u=\emptyset ) \lesssim \Pb( \Ac^{\lambda}_{D}(u;d_1,d_2) ).
	\]
\end{lemma}

\begin{proof}[Proof of Lemma~\ref{lem:delete-dot}]
	By Lemma~\ref{lem:lambda-locality}, we have
	\begin{equation}
		\Pb( \Ac^{\lambda}_{D}(\dot{u};d_1,d_2) ) \lesssim 
		\Pb( \overleftarrow\Ac^{\lambda}_{D}(\dot{u};d_1,d_2) ).
	\end{equation}
	Moreover, we observe that 
	\[
	\overleftarrow\Ac^{\lambda}_{D}(\dot{u};d_1,d_2) \cap \{ \Lc^u_{D}\setminus \Lc_{d_1/2}^u=\emptyset \} \subseteq \Ac^{\lambda}_{D}(u;d_1,d_2),
	\]
	and the two events on the l.h.s. are independent, which implies the result.
\end{proof}

For $u\sim^* v$ and $D \supseteq B_{2d}(u)$, let $\Ac_{D}^{\lambda}(u,v;d)$ be the event that there exist two disjoint $\lambda*$-chains in $\Lc_D$ from $u$ to $\partial B_d(u)$ and from $v$ to $\partial B_d(u)$, respectively. 
Let $X^v_{D}(u):=X_{\Lc_D\setminus\Lc_D^v}(u)$ be the occupation at $u$ produced by loops in $\Lc_D$ that avoid $v$.
We show that $\Ac_{D}^{\lambda}(u,v;d)$ can be decoupled from the occupation field at $u$.

\begin{figure}[t]
	\centering
	\subfigure{\includegraphics[width=.48\textwidth]{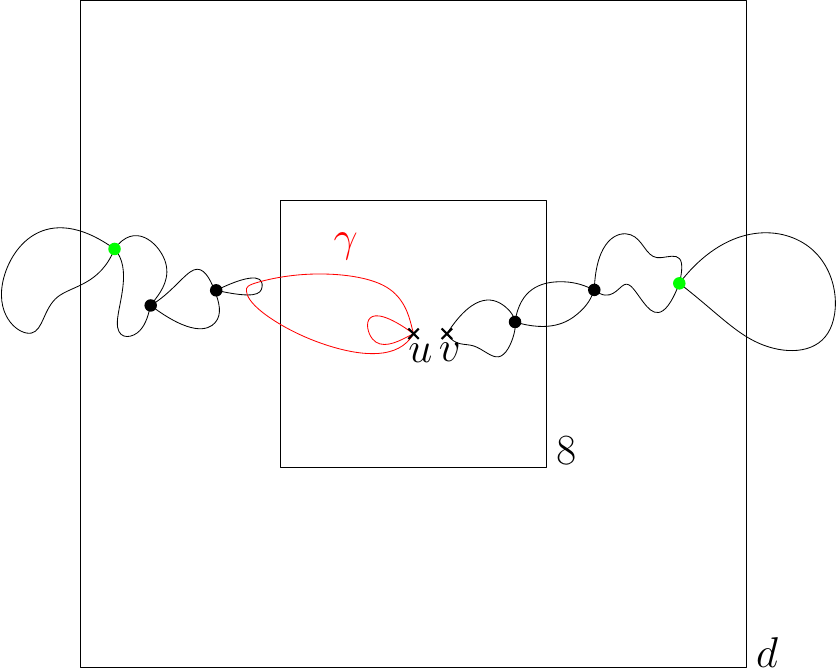}}
	\hspace{0.3cm}
	\subfigure{\includegraphics[width=.4\textwidth]{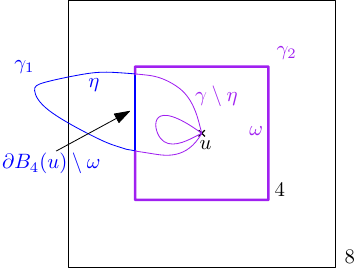}}
	\caption{\emph{Left:} Illustration of the event $\Ac^{\lambda}_D(u,v;d)\setminus \Ac^{\lambda}_D(\dot{u};8,d)$: $*$- and $\lambda$-passage edges are indicated by green and black dots, respectively. The red loop $\gamma$ is pivotal. \emph{Right:} We perform some local surgery, within $B_8(u)$, on the loop $\gamma$: $\eta$ is shown in blue, $\gamma\setminus\eta$ in purple, and $\omega$ (resp. $\partial B_4(u)\setminus\omega$) is the thicker purple (resp. blue) arc along $\partial B_4(u)$. Hence, $\gamma_1$ and $\gamma_2$ are obtained by concatenating the blue paths and the purple paths, respectively.}
	\label{fig:mod-z}
\end{figure}

\begin{lemma}\label{lem:mod-z}
	For all $\alpha>0$, there exists $c_7(\alpha)>0$ such that for all $\lambda>0$, $d> 8$, $u,v\in \Zb^2$ with $u\sim^* v$, $D \supseteq B_{2d}(u)$, and $a>0$,
	\begin{equation}\label{eq:modz}
		\Pb\big( \Ac_{D}^{\lambda}(u,v;d) \cap \{ X_{D}(u)>a \} \big) \le c_7\, \Pb\big( \Ac_{D}^{\lambda}(\dot{u};8,d) \big) \, \Pb( \Lc_D(\{u,v\})=\emptyset )\, \Pb( X^v_{D}(u)>a ).
	\end{equation}
\end{lemma}

\begin{proof}[Proof of Lemma~\ref{lem:mod-z}]
	Note that $X_{D}(u)=X^v_{D}(u)$ on the event $\Lc_D(\{u,v\})=\emptyset$, which is itself implied by the event $\Ac_{D}^{\lambda}(u,v;d)$.
	The three events on the r.h.s. of \eqref{eq:modz} are independent, since they use disjoint sets of loops, namely loops avoiding $u$, loops visiting both $u$ and $v$, and loops visiting $u$ but avoiding $v$, respectively. Thus, it suffices to show that 
	\[
	\Pb\big( (\Ac^{\lambda}_D(u,v;d)\setminus \Ac^{\lambda}_D(\dot{u};8,d)) \cap \{ X_{D}(u)>a \} \big) \lesssim_{\alpha} \Pb\big( \Ac_{D}^{\lambda}(\dot{u};8,d)\cap\{ \Lc_D(\{u,v\})=\emptyset \} \cap \{ X_{D}(u)>a \} \big).
	\]
	Write $\Lc_D^{(u)}:=\Lc_D\setminus \Lc_D^u$.
	On the event $\Ac^{\lambda}_D(u,v;d)\setminus \Ac^{\lambda}_D(\dot{u};8,d)$, there exists a $*\lambda$-chain in $\Lc_D^{(u)}$ from $\partial B_d(u)$ to $v$, and another $*\lambda$-chain in $\Lc_D^{(u)}$ that intersects $\partial B_{d}(u)$, stays in $B_8(u)^c$ and is connected (or $*$-connected) to $u$ by a loop $\gamma\in \Lc_D^u$, such that they together form two admissible $*\lambda*$-chains. Such a loop $\gamma$ that determines the occurrence of $\Ac^{\lambda}_D(u,v;d)$ is called a \emph{pivotal loop} from now on (note that in principle, there could be several of them). Our proof strategy is to split that pivotal loop into two loops, through local surgery around $u$, so that the event $\Ac^{\lambda}_D(\dot{u};8,d) \cap  \{ X_{D}^v(u)>a \} \cap\{ \Lc_D(\{u,v\})=\emptyset \}$ is satisfied after splitting. This will be enough to conclude. 
	
	More precisely, let $\gamma$ be any pivotal loop that visits $u$ as above. First, we find an excursion $\eta\subseteq \gamma$ in $\partial B_4(u)\cup B_4(u)^c$, which is $*$-connected to some $*\lambda$-chain in $\Lc_D^{(u)}$ that intersects $\partial B_{d}(u)$ and stays in $B_8(u)^c$.
	We refer the reader to Figure~\ref{fig:mod-z} for an illustration.
	Among the two arcs of $\partial B_4(u)$ joining the endpoints of $\eta$, let $\omega$ be the one such that the concatenation of $\gamma$ and $\omega$ surrounds $0$. Then, we attach $\partial B_4(u)\setminus\omega$ and $\omega$ to $\eta$ and $\gamma\setminus\eta$, respectively, to produce two loops $\gamma_1$ and $\gamma_2$. 
	In this way, $\gamma_1$ does not visit $u$, and the $\lambda$-arm event $\Ac^{\lambda}_D(\dot{u};8,d)$ occurs for $\Lc_D^{(u)}\uplus\{\gamma_1\}$, while $\gamma_2$ is a loop that visits $u$ but avoids $v$ and retains the same occupation at $u$.
	In other words, the event $\Ac^{\lambda}_D(\dot{u};8,d) \cap  \{ X_{D}^v(u)>a \} \cap\{ \Lc_D(\{u,v\})=\emptyset \}$ occurs after this modification (replacing $\gamma$ by $\gamma_1$ and $\gamma_2$). 
	Moreover, it is clear that $\gamma_1$ does not appear in the original configuration (otherwise $\Ac^{\lambda}_D(\dot{u};8,d)$ would occur). Since $\gamma_2$ intersects the other $*\lambda*$-chain (containing $v$), it does not appear in the original configuration as well.
	In addition, both $\gamma_1$ and $\gamma_2$ have multiplicity $1$.
	Therefore, by \eqref{eq:rwls}, the ratio between the probability of the resulting configuration (i.e.\ after modification) and that of the original configuration is bounded from below by 
	\begin{equation} \label{eq:cost_gamma12}
		\frac{\alpha \mu_0(\gamma_1) \,\alpha \mu_0(\gamma_2) }{\alpha \mu_0(\gamma)/(n(\gamma)J(\gamma))} \gtrsim \alpha,
	\end{equation}
	where $n(\gamma)\ge 1$ is the number of occurrences of $\gamma$ in the original configuration, and $J(\gamma)\ge 1$ is the multiplicity of $\gamma$. 
	Hence, \eqref{eq:cost_gamma12} implies that for any eligible configuration, our modification only costs a constant multiple of $\alpha^{-1}$ (which does not depend on the configuration). This completes the proof of Lemma~\ref{lem:mod-z}.
\end{proof}

\subsubsection{Decoupling using $\lambda$-passage edges}

Next, we show a decoupling result for $\lambda$-arm events under the presence of a $\lambda$-passage edge, which has a similar flavor to Lemma~\ref{lem:decom*}. However, in this case we make appear furthermore an extra factor in the multiplicative constant which decays exponentially in $\lambda$, due to the excessive occupation observed on the extremities of the $\lambda$-passage edge. This additional factor plays a key role in later summations.

\begin{lemma}\label{lem:decom}
	For all $\alpha\in (0,1/2]$, there exists $c_8(\alpha)>0$ such that the following holds. Let $u,v\in \Zb^2$ with $|u|=k$ and $u\sim^* v$, consider $0 < d_1 < d_2$ such that $16d_1\le k \le d_2/8$, and let $D \supseteq B_{2d_2}$. Let $\Pc^{\lambda}$ be the set of $\lambda$-passage edges associated with $\Ac^{\lambda}_D(d_1,d_2)$. For some universal constant $\rho\in (0,1)$, we have
	\begin{align}\notag
		\Pb( & \Ac^{\lambda}_D(d_1,d_2) \cap \{ (u,v)\in\Pc^{\lambda}  \} ) \\  \label{eq:decom}
		& \quad \le c_8\, \rho^{\lambda}\, 
		\Pb( \Ac^{\lambda}_{\loc}(d_1,k/8) ) \, \Pb( \Ac^{\lambda}_D(4k,d_2) ) \, \Pb( \Ac^{\lambda}_{\loc}(u;8,k/4)). 
	\end{align}
\end{lemma}

\begin{proof}[Proof of Lemma~\ref{lem:decom}]
	\textbf{Step 1.} For any loop configuration $L$ in $D$, let $E(L)$ be the event that both of the following two properties hold:
	\begin{enumerate}[(i)]
		\item there is only one loop in $L$ visiting $u$, and there is no loop in $L$ visiting both $u$ and $v$;
		\item there exists a $*\lambda*$-chain across $A_{d_1,d_2}$, a $*\lambda$-chain from $\partial B_{d_1}$ to $u$, and a $\lambda*$-chain from $v$ to $\partial B_{d_2}$, and these three chains are disjoint.
	\end{enumerate}
	Further, we subdivide this event as $E(L) = E_1(L) \cup E_2(L)$, with $E_1(L) := E(L) \cap \{X_L(u)\le \lambda/2\}$ and $E_2(L) := E(L) \cap \{X_L(u) > \lambda/2\}$.

	We can then decompose the event on the left-hand side of \eqref{eq:decom} according to the number of loops that visit $u$. For $m \geq 1$, let $\Vc_m(L) := \Ac^{\lambda}_L(d_1,d_2) \cap \{ (u,v)\in\Pc^{\lambda}  \} \cap \{$there are exactly $m$ loops in $L$ that visit $u\}$, so that we have immediately
	\begin{equation}\label{eq:dec1}
		\Pb( \Ac^{\lambda}_D(d_1,d_2) \cap \{ (u,v)\in\Pc^{\lambda}  \} )
		=\sum_{m\ge 1} \Pb( \Vc_m(\Lc_D) ).
	\end{equation}
	By Lemma~\ref{lem:Palm formula} (in a similar way as in the proof of Lemma~\lemloops{} in \cite{GNQ2024a}, see Eq.~(\eqrec{}) there), we have: for all $m\ge 1$,
	\begin{equation} \label{eq:Vm_LcD}
		\Pb( \Vc_m(\Lc_D) ) =\frac{1}{m!} \sum_{\gamma_1\in L_{u,v}} \cdots \sum_{\gamma_m\in L_{u,v}} \Pb( \Vc_m(\Lc_D\uplus\{ \gamma_1,\ldots,\gamma_m \}) ) \prod_{j=1}^m \alpha\nu(\gamma_j),
	\end{equation}
	where $L_{u,v} := L_D^u \setminus L_D^v$ denotes the set of all loops in $D$ that visit $u$ but avoid $v$.
	
	For any $m\ge 2$, and $\gamma_1, \ldots, \gamma_m \in L_{u,v}$, we observe that 
	\begin{equation} \label{eq:Vm_LcD2}
		\Vc_m(\Lc_D\uplus\{ \gamma_1,\ldots,\gamma_m \})
		\subseteq \bigcup_{j=1}^m \big(E_1(\Lc_D\uplus\{ \gamma_j\}) \cap \{ X^j(u)>\lambda/2 \}\big) \cup E_2(\Lc_D\uplus\{ \gamma_j\}),
	\end{equation}
	with $X^j(u):=X_{\{ \gamma_1,\ldots,\gamma_m \}\setminus\{ \gamma_j\}}(u)$. It then follows from the union bound applied to \eqref{eq:Vm_LcD2}, combined with \eqref{eq:Vm_LcD}, that for all $m \geq 1$,
	\begin{align}\notag
		\Pb( \Vc_m(\Lc_D) )&\le  \sum_{\gamma\in L_{u,v}} \Pb( E_1(\Lc_D\uplus\{ \gamma\}) ) \,\alpha\nu(\gamma) \\\notag
		&\hspace{1cm} \times \frac{1}{(m-1)!}  \sum_{\gamma_1\in L_{u,v}} \cdots \sum_{\gamma_{m-1}\in L_{u,v}} \Pb(X_{\{ \gamma_1,\ldots,\gamma_{m-1} \}}(u)>\lambda/2) \prod_{j=1}^{m-1} \alpha\nu(\gamma_j)  \\[3mm] \notag
		&\hspace{2cm} + \frac{(\alpha\nu(L_{u,v}))^{m-1}}{(m-1)!} \sum_{\gamma\in L_{u,v}} \Pb( E_2(\Lc_D\uplus\{ \gamma\}) )\, \alpha\nu(\gamma)  \\ \label{eq:dec2}
		& =: I_m + J_m.
	\end{align}
	We now bound the two terms $I_m$ and $J_m$ separately, respectively in Steps~2 and 3.
	
	\textbf{Step 2.}
	We analyze $I_m$ first, and for this purpose, we denote by $n_u(\gamma)$ the total number of visits to $u$ made by $\gamma$. We claim that there exists a universal constant $0<c<1$ such that for all $n\ge 1$,
	\begin{equation}\label{eq:S}
		(S):=\sum_{m\ge 2} \frac{1}{(m-1)!}
		(\alpha\nu)^{m-1}\bigg( \bigg\{ (\gamma_1,\cdots,\gamma_{m-1})\in L_{u,v}^{m-1}: \sum_{i=1}^{m-1}n_u(\gamma_i)>n \bigg\} \bigg) \le c^n. 
	\end{equation} 
	Let us prove \eqref{eq:S}. To this aim, we rewrite $(S)$ as 
	\begin{align*}
		(S)=\sum_{r>n} \sum_{m\ge 1} \sum_{(n_1,\ldots, n_m)\in S_{m,r}} \frac{1}{m!} \prod_{i=1}^{m} \alpha\nu (\{ \gamma\in L_{u,v}: n_u(\gamma)=n_i \}),
	\end{align*}
	where $S_{m,r}$ denotes the set of all $m$-tuples of positive integers $(n_1,\ldots, n_m)$ such that $n_1+\cdots+n_m=r$. Note that for each $i =1, \ldots, m$,
	\begin{equation}\label{eq:nu-L}
		\nu (\{ \gamma\in L_{u,v}: n_u(\gamma)=n_i \})
		=\mu_0(\{ \text{$l$ is a loop rooted at $u$}: \Uc(l)\in L_{u,v}, n_u(l)=n_i \})/n_i
	\end{equation}
	(where $\mu_0$ is the natural extension of $\mu_0$ to rooted loops, see the paragraph below \eqref{eq:W-nu}), which is at most $\delta^{n_i}/n_i$ for some $0<\delta<1$, from the observation above \eqref{eq:WD} (recall that $L_{u,v}$ is the set of loops in $D$ that visit $u$ but avoid $v$). Therefore,
	\[
	(S)\le \sum_{r>n} \sum_{1 \leq m \leq r} \sum_{(n_1,\ldots, n_m)\in S_{m,r}} \frac{1}{m!} \prod_{i=1}^{m} \frac{\alpha \delta^{n_i}}{n_i} \le \sum_{r>n} \delta^r \lesssim c^n,
	\]
	where in the second inequality, we first dropped $\alpha$, thanks to $\alpha\le 1/2 <1$, and we then used the following combinatorial identity (see for instance Eq.~(6) in \cite{WP2021}):
	\[
	\text{for each $r \geq 1$,} \quad \sum_{1 \leq m \leq r} \sum_{(n_1,\ldots, n_m)\in S_{m,r}} \frac{1}{m!} \prod_{i=1}^{m} \frac{1}{n_i}=1.
	\]
	This shows the claim \eqref{eq:S}. From it, we can deduce that for some $\rho_1 \in (0,1)$,
	\[
	\sum_{m\ge 2} \frac{1}{(m-1)!} \sum_{\gamma_1\in L_{u,v}} \cdots \sum_{\gamma_{m-1}\in L_{u,v}} \Pb(X_{\{ \gamma_1,\ldots,\gamma_{m-1} \}}(u)>\lambda/2) \prod_{j=1}^{m-1} \alpha\nu(\gamma_j) \lesssim \rho_1^\lambda.
	\]
	Indeed, in the sum above, it suffices to distinguish whether the total number of visits made by the loops $\gamma_1, \ldots, \gamma_{m-1}$ together is $> \lambda/4$ or $\leq \lambda/4$. In the the former case, one can use \eqref{eq:S}, while in the latter case, one can notice that $X_{\{ \gamma_1,\ldots,\gamma_{m-1} \}}(u)$ is then stochastically dominated by a Gamma distributed random variable with parameters $(\lambda/4+\alpha,1)$, by construction of the occupation field (see Section~\ref{subsec:RWLS}).
	Therefore, using $E_1(L)\subseteq E(L)$, we get
	\begin{equation} \label{eq:sumIm1}
		\sum_{m\ge 2} I_m \le \rho_1^\lambda\, \sum_{\gamma\in L_{u,v}} \Pb( E(\Lc_D\uplus\{ \gamma\}) ) \,\alpha\nu(\gamma).
	\end{equation}
	
	Finally, we use the same construction as for Lemma~\ref{lem:quasi} in \cite{GNQ2024a}, in order to decouple the three arm events on the right-hand side of \eqref{eq:decom}. More precisely, we can in this way modify any potential crossing loop (which would intersect any two of the three corresponding annuli), which provides the following quasi-multiplicativity result:
	\begin{equation} \label{eq:sumIm2}
		\sum_{\gamma\in L_{u,v}} \Pb( E(\Lc_D\uplus\{ \gamma\}) ) \,\alpha\nu(\gamma) \lesssim_{\alpha} \Pb( \Ac^{\lambda}_{\loc}(d_1,k/8) ) \, \Pb( \Ac^{\lambda}_D(4k,d_2) ) \, \Pb( \Ac^{\lambda}_{\loc}(u;1,k/4) ).
	\end{equation}
	Note that here, there is an additional difficulty compared to the proof of Lemma~\ref{lem:quasi}, since the modification might change the occupation at some $\lambda$-passage edges. However this is not a real issue, thanks to our relaxation, in the chains associated with the $\lambda$-arm events, on the occupation at the first and last $*$-passage edges (see (\ref{rem:relax}) in Remark~\ref{rem:modified arm event} above). Hence, combining \eqref{eq:sumIm1} and \eqref{eq:sumIm2}, we finally obtain
	\begin{equation}\label{eq:sum-I}
		\sum_{m\ge 2} I_m\lesssim_{\alpha} \rho_1^\lambda\,\Pb( \Ac^{\lambda}_{\loc}(d_1,k/8) ) \, \Pb( \Ac^{\lambda}_D(4k,d_2) ) \, \Pb( \Ac^{\lambda}_{\loc}(u;1,k/4) ).
	\end{equation}
	Note that here, we can use an inner radius $1$ instead of $8$ in the above two inequalities, since we do not need to use the special surgery in Lemma~\ref{lem:mod-z}. Later in Step~4, we simply use the inequality $\Pb( \Ac^{\lambda}_{\loc}(u;1,k/4) )\le \Pb( \Ac^{\lambda}_{\loc}(u;8,k/4) )$, which follows from obvious monotonicity.
	
	\textbf{Step 3.} We now analyze $J_m$. First, it follows from \eqref{eq:nu-L} that $\nu(L_{u,v})$ is uniformly bounded, so
	\begin{equation} \label{eq:nu_unif_bd}
		\sum_{m\ge 2}\frac{(\alpha\nu(L_{u,v}))^{m-1}}{(m-1)!}=e^{\alpha\nu(L_{u,v})}-1\le C,
	\end{equation}
	for some universal $C>0$.
	
	Next, we need to deal with the remaining factor $\sum_{\gamma\in L_{u,v}} \Pb( E_2(\Lc_D\uplus\{ \gamma\}) )\, \alpha\nu(\gamma)$. In this case, we know that $\gamma$ contributes an occupation at least $\lambda/2$ at $u$. We would like to use again the modification from the proof of Lemma~\ref{lem:quasi}. However, note that it does not work here when $\gamma$ is also a crossing loop, since we have to preserve the occupation at $u$. We now explain how to solve this issue. For this purpose, we refer to notations in \cite{GNQ2024a}, see in particular Figure~\figquasi{} there. If we decompose $\gamma$ as in Eq.~(\eqquasidec{}) in that paper, but without the superscript $i$, the intermediate parts $\omega_1$ and $\omega_2$ of $\gamma$ might have visits to $u$ (see the blue paths in the figure just mentioned). Hence, we modify the construction by keeping these intermediate parts.
	
	More precisely, recall that $\gamma_1$ and $\gamma_2$ were obtained from $\gamma$ through Eq.~(\eqgamma{}) in \cite{GNQ2024a}. Apart from that, we further sample $\omega_1''$ and $\omega_2''$ according to $\mu^{z_1,z_2}$ and $\mu^{z_3,z_4}$, independently, such that $\omega_1''\subseteq l_1^{d_2/100}$ and $\omega_2''\subseteq l_2^{d_2/100}$ (where $l_1^{d_2/100}$ and $l_2^{d_2/100}$ are the $d_2/100$-sausages of $l_1$ and $l_2$, respectively). Then, we attach them to $\omega_1$ and $\omega_2$ to get a third loop
	$$\gamma_3:=\Uc(\omega_1\oplus \omega_1''\oplus \omega_2\oplus \omega_2''),$$
	that we keep in the final modified configuration.
	
	In this way, we can preserve the occupation produced by $\gamma$ without spoiling the arm event. Therefore, we obtain that 
	\begin{align} \notag
		\sum_{\gamma\in L_{u,v}} & \Pb( E_2(\Lc_D\uplus\{ \gamma\}) ) \,\alpha\nu(\gamma)\\ \label{eq:J1}
		& \quad \lesssim_{\alpha} \Pb( \Ac^{\lambda}_{\loc}(d_1,k/8) ) \, \Pb( \Ac^{\lambda}_{\Lc_D\setminus\Lc_D^u}(4k,d_2) ) \,  \Pb( \Ac^{\lambda}_{D}(u,v;k/4)\cap \{X_{D}(u)>\lambda/2\}).
	\end{align}
	Using Lemma~\ref{lem:mod-z}, we have  
	\begin{equation} \label{eq:J1_arms}
		\Pb( \Ac^{\lambda}_{D}(u,v;k/4)\cap \{X_{D}(u)>\lambda/2\}) \le c_7\, \Pb( \Ac^{\lambda}_{D}(\dot{u};8,k/4) )\, \Pb( \Lc_D(\{u,v\})=\emptyset )\, \Pb( X^v_{D}(u)>\lambda/2 ).
	\end{equation}
	From Lemma~\ref{lem:0w}, we get 
	\begin{equation} \label{eq:uv-u}
		\Pb( \Lc_D(\{u,v\})=\emptyset )\lesssim \Pb( \Lc_D^u=\emptyset ),
	\end{equation}
	and by independence, 
	\begin{equation} \label{eq:LkLu}
		\Pb( \Lc_D^u=\emptyset )=\Pb( \Lc^u_{k/2} = \emptyset ) \, \Pb( \Lc_D^u\setminus \Lc^u_{k/2} = \emptyset ). 
	\end{equation}
	By Lemma~\ref{lem:lambda-locality},
	\[
	\Pb( \Ac^{\lambda}_{D}(\dot{u};8,k/4) )\,\Pb( \Lc^u_{k/2} = \emptyset ) \lesssim \Pb( \overrightarrow\Ac^{\lambda}_D(\dot{u};8,k/4)) \,\Pb( \Lc^u_{k/2} = \emptyset ) = \Pb( \overrightarrow\Ac^{\lambda}_D(\dot{u};8,k/4)\cap \{ \Lc^u_{k/2} = \emptyset \} ).
	\]
	Let $G$ be the event that there is no loop in $\Lc_D$ crossing $A_{k/4,k/2}$ without disconnecting $B_8$ from infinity. By Lemma~\lemcrossingloop{} in \cite{GNQ2024a}, we have $\Pb(G)\ge C$ for some universal constant $C>0$. Note that both events $G$ and $\overrightarrow\Ac^{\lambda}_D(\dot{u};8,k/4)\cap \{ \Lc^u_{k/2} = \emptyset \}$ are decreasing with respect to the loops in $\Lc_D$ which cross $A_{k/4,k/2}$ without disconnecting $B_8$ from infinity. Therefore, by the FKG inequality (Lemma~\ref{lem:FKG-RWLS}), we have 
	\[
	\Pb( \overrightarrow\Ac^{\lambda}_D(\dot{u};8,k/4)\cap \{ \Lc^u_{k/2} = \emptyset \} ) \lesssim \Pb( \overrightarrow\Ac^{\lambda}_D(\dot{u};8,k/4)\cap \{ \Lc^u_{k/2} = \emptyset \}\cap G ).
	\]
	Since $\overrightarrow\Ac^{\lambda}_D(\dot{u};8,k/4)\cap \{ \Lc^u_{k/2} = \emptyset \}\cap G$ implies $\Ac^{\lambda}_{\loc}(u;8,k/4)$, it follows that 
	\begin{equation} \label{eq:J1_arms2}
		\Pb( \Ac^{\lambda}_{D}(\dot{u};8,k/4) )\,\Pb( \Lc^u_{k/2} = \emptyset ) \lesssim \Pb( \Ac^{\lambda}_{\loc}(u;8,k/4)).
	\end{equation}
	
	By adapting again the proof of Lemma~\ref{lem:delete-dot}, we also have 
	\begin{align} \notag
		&\Pb( \Ac^{\lambda}_{\Lc_D\setminus\Lc_D^u}(4k,d_2) ) \,\Pb( \Lc_D^u\setminus \Lc^u_{k/2} = \emptyset  )\\ \label{eq:J1_arms3}
		&\hspace{3cm} \lesssim \Pb( \overleftarrow\Ac^{\lambda}_{\Lc_D\setminus\Lc_D^u}(4k,d_2) ) \,\Pb( \Lc_D^u\setminus \Lc^u_{2k}=\emptyset  )
		\lesssim \Pb(  \Ac^{\lambda}_{D}(4k,d_2) ).
	\end{align}
	If we now plug \eqref{eq:J1_arms}, \eqref{eq:uv-u}, \eqref{eq:LkLu}, \eqref{eq:J1_arms2}, and \eqref{eq:J1_arms3} into \eqref{eq:J1}, we obtain 
	\begin{align*}
		\sum_{\gamma\in L_{u,v}} & \Pb( E_2(\Lc_D\uplus\{ \gamma\}) ) \,\alpha\nu(\gamma)\\
		& \quad \lesssim_{\alpha} \Pb( \Ac^{\lambda}_{\loc}(d_1,k/8) ) \, \Pb(  \Ac^{\lambda}_{D}(4k,d_2) ) \, \Pb( \Ac^{\lambda}_{\loc}(u;8,k/4))\, \Pb( X_{D}^v(u)>\lambda/2 ).
	\end{align*}
	By Lemma~\ref{lem:Gauss} and Chebyshev's inequality, there exists $\rho_2\in(0,1)$ such that
	\begin{equation*}
		\Pb( X^v_{k/2}(u)>{\lambda} )\le \rho_2^{\lambda},
	\end{equation*}
	which yields
	\begin{equation}\label{eq:E2}
		\sum_{\gamma\in L_{u,v}} \Pb( E_2(\Lc_D\uplus\{ \gamma\}) ) \,\alpha\nu(\gamma) 
		\lesssim_{\alpha} \rho_2^{\lambda}\,
		\Pb( \Ac^{\lambda}_{\loc}(d_1,k/8) ) \, \Pb( \Ac^{\lambda}_D(4k,d_2) ) \,  \Pb( \Ac^{\lambda}_{\loc}(u;8,k/4)).
	\end{equation}
	
	Thus, combining \eqref{eq:nu_unif_bd} and \eqref{eq:E2}, we have
	\begin{equation} \label{eq:sum-J}
		\sum_{m\ge 2} J_m\lesssim_{\alpha} \rho_2^\lambda\,\Pb( \Ac^{\lambda}_{\loc}(d_1,k/8) ) \, \Pb( \Ac^{\lambda}_D(4k,d_2) ) \,  \Pb( \Ac^{\lambda}_{\loc}(u;8,k/4) ).
	\end{equation}
	
	\textbf{Step 4.} Finally, we deal with $\Vc_1(\Lc_D)$ and conclude. By definition, $\Vc_1(\Lc_D)\subseteq E_2(\Lc_D)$, therefore, by \eqref{eq:E2},
	\begin{align*}
		& \Pb(\Vc_1(\Lc_D) ) \le \Pb( E_2(\Lc_D) ) \le \sum_{\gamma\in L_{u,v}} \Pb( E_2(\Lc_D\uplus\{ \gamma\}) ) \,\alpha\nu(\gamma) \\
		&\hspace{3cm}
		\lesssim_{\alpha} \rho_2^{\lambda}\,
		\Pb( \Ac^{\lambda}_{\loc}(d_1,k/8) ) \, \Pb( \Ac^{\lambda}_D(4k,d_2) ) \,  \Pb( \Ac^{\lambda}_{\loc}(u;8,k/4)).
	\end{align*}
	
	Combined with the previous bounds, this gives that for some $\rho\in (\rho_1\vee\rho_2,1)$,
	\begin{align*}
		\Pb( \Ac^{\lambda}_D(d_1,d_2) \cap \{ (u,v)\in\Pc^{\lambda}  \} )&	\stackrel{\eqref{eq:dec1}, \eqref{eq:dec2}}{\le}\ \Pb(\Vc_1(\Lc_D) ) + \sum_{m\ge 2} (I_m+J_m) \\
		&	\stackrel{\eqref{eq:sum-I}, \eqref{eq:sum-J}}{\lesssim_{\alpha}}\ 	
		\rho^{\lambda}\,
		\Pb( \Ac^{\lambda}_{\loc}(d_1,k/8) ) \, \Pb( \Ac^{\lambda}_D(4k,d_2) ) \,  \Pb( \Ac^{\lambda}_{\loc}(u;8,k/4)),
	\end{align*}
	which completes the proof of Lemma~\ref{lem:decom}.
\end{proof}

\subsubsection{Main estimate: stability of $\lambda$-arm events}

Now, we are finally in a position to prove the main result of this section, which shows that when $\lambda$ is large enough, the probability of the $\lambda$-arm event $\Ac^{\lambda}_D(z;d_1,d_2)$ is upper bounded by (a constant multiple of) $\pi(d_1,d_2)$, exactly as the arm event $\Ac_D(z;d_1,d_2)$. As we have explained, the proof involves an induction procedure which has the same spirit to that of Proposition~\ref{prop:m4a*} or \cite[Theorem~5.3]{BN2022}.

\begin{proposition}\label{prop:mae}
	For all $\alpha\in (0,1/2)$, there exist $c_9(\alpha)>0$ and $\lambda_0(\alpha)>0$ such that the following holds. For all $\lambda \geq \lambda_0$, $z\in \Zb^2$, $0<d_1< d_2$, and $D \supseteq B_{2d_2}(z)$,
	\begin{equation}\label{eq:modified-arm}
		\Pb(\Ac^{\lambda}_D(z;d_1,d_2))\le c_9\, \pi(d_1,d_2).
	\end{equation}
\end{proposition}

\begin{proof}[Proof of Proposition~\ref{prop:mae}]
	Once again, using monotonicity (see (\ref{rem:mae(b)}) in Remark~\ref{rem:modified arm event}), we can see that it is enough to show that \eqref{eq:modified-arm} holds for $d_1=2^i$ and $d_2=2^j$, for any two integers $0\le i<j$. We prove it by induction on $(i,j)$. 
	
	First, it is clear that for $c_9 \geq 2^{10(2+\beta)}$, \eqref{eq:modified-arm} holds for all pairs $(i,j)$ with $1\le j-i\le 10$. Now, consider a pair $(i,j)$ satisfying $j-i\ge 11$, and assume that \eqref{eq:modified-arm} holds for all pairs $(i',j')$ with $j'<j$, as well as for all pairs $(i',j)$ with $i'>i$. It remains to show that under this assumption, \eqref{eq:modified-arm} holds for the pair $(i,j)$.
	
	Letting
	\begin{equation}\label{eq:Ec12}
		\Ec_1=\Ac^{\lambda}_{D}(z;2^{i+4},2^{j-3})\cap \{ \Nc\le 4 \},
		\quad \Ec_2=\Ac^{\lambda}_D(z;2^i,2^j)\setminus \Ec_1,
	\end{equation}
	where $\Nc$ is the connection number associated with $\Ac^{\lambda}_{D}(z;2^{i+4},2^{j-3})$,
	we have 
	\begin{equation}\label{eq:sp1}
		\Pb(\Ac^{\lambda}_D(z;2^i,2^j))\le \Pb( \Ec_1 ) + \Pb( \Ec_2 ).
	\end{equation}
	From (\ref{rem:mae(b)}) and (\ref{rem:mae(a)}) of Remark~\ref{rem:modified arm event}, respectively, we have the inclusions
	\begin{equation}\label{eq:inclusion-1}
		\Ac^{\lambda}_D(z;2^i,2^j) \subseteq \Ac^{\lambda}_{D}(z;2^{i+4},2^{j-3}) \subseteq \Ac^*_{D}(z;2^{i+4},2^{j-3}). 
	\end{equation}
	Using the second inclusion in \eqref{eq:inclusion-1}, and applying Proposition~\ref{prop:m4a*} with $m=4$, we obtain that 
	\begin{equation}\label{eq:Ec1}
		\Pb(\Ec_1 )\le \Pb( \Ac^*_{D}(z;2^{i+4},2^{j-3})\cap \{ \Nc\le 4 \} ) \le e^{c_5^5}\, \pi(2^{i+4},2^{j-3}) = e^{c_5^5}\, 2^{7(2+\beta)}\, \pi(2^i,2^j).
	\end{equation}
	Next, we deal with $\Ec_2$. Let $\Pc^{\lambda}$ be the $\lambda$-passage set associated with $\Ac^{\lambda}_D(z;2^i,2^j)$.
	By the inclusion \eqref{eq:inclusion-1}, on the event $\Ec_2$, there exists at least one $\lambda$-passage edge $(u,v)\in\Pc^{\lambda}$ such that $u\in A_{2^{i+4},2^{j-3}}(z)$ (see (\ref{rem:mae(last)}) of Remark~\ref{rem:modified arm event}). By the union bound, and then Lemma~\ref{lem:decom}, we have
	\begin{align}\notag
		\Pb(\Ec_2) &\le \sum_{k=i+4}^{j-4} \sum_{u \in A_{2^k,2^{k+1}}(z)} \sum_{v: v \sim^* u}  \Pb( \Ac^{\lambda}_D(z;2^i,2^j) \cap \{ (u,v)\in\Pc^{\lambda}  \} ) \\ \label{eq:lbd-E2}
		&\le \sum_{k=i+4}^{j-4} \sum_{u \in A_{2^k,2^{k+1}}(z)} \sum_{v: v \sim^* u} c_8\, \rho^{\lambda}\,  \Pb( \Ac^{\lambda}_{\loc}(z;2^i,2^{k-3}) ) \,\Pb( \Ac^{\lambda}_D(z;2^{k+2},2^j) )\, \Pb( \Ac^{\lambda}_{\loc}(u;8,2^{k-2}) ),
	\end{align}
	for some universal constant $\rho \in (0,1)$. Applying the induction hypothesis three times, we obtain
	\begin{align*}
		\Pb(\Ec_2) & \le \sum_{k=i+4}^{j-4} \sum_{u \in A_{2^k,2^{k+1}}(z)} \sum_{v: v \sim^* u} c_8\, \rho^{\lambda}\, \cdot c_9\,  \pi(2^i,2^{k-3}) \, \cdot c_9\,  \pi(2^{k+2},2^j)\, \cdot c_9\,  \pi(8,2^{k-2}).
	\end{align*}
	Finally, this yields
	\begin{equation}\label{eq:Ec2}
		\Pb(\Ec_2) \le c_8\, c_9^3\, C_2\,  2^{5(2+\beta)}\,   \rho^{\lambda} \,  \pi(2^i,2^{j}),
	\end{equation}
	using first that 
	$$\pi(2^i,2^{k-3})\, \pi(2^{k+2},2^j) = 2^{5(2+\beta)}\, \pi(2^i,2^{j}),$$
	and then the summability of the remaining sum:
	\begin{align*}
		\sum_{k=i+4}^{j-4} \sum_{u \in A_{2^k,2^{k+1}}(z)} \sum_{v: v \sim^* u}
		\pi(8,2^{k-2}) \le
		\sum_{k=i+4}^{j-4} C_1\, 2^{2k}\,  2^{-(2+\beta)k}\le C_2,
	\end{align*}
	for some universal constants $0<C_1,C_2<\infty$ (it is in this last inequality that the bound $\xi(\alpha) > 2$ on the four-arm exponent is used crucially).
	
	We now explain how to fix $c_9$ so as to make the induction step work. We first take some $c_9 \geq 2^{10(2+\beta)} \vee (e^{c_5^5}\, 2^{7(2+\beta)+1})$, and then for such a $c_9$, we let $\lambda_0$ be large enough so that
	$$c_8\, c_9^3\, C_2\,  2^{5(2+\beta)}\, \rho^{\lambda_0} \leq \frac{c_9}{2}.$$
	For such a choice of $c_9$ and $\lambda_0$, we obtain from \eqref{eq:sp1}, combined with \eqref{eq:Ec1} and \eqref{eq:Ec2}, that for all $\lambda \geq \lambda_0$,
	\[
	\Pb(\Ac^{\lambda}_D(z;2^i,2^j))\le c_9\,  \pi(2^i,2^{j}).
	\]
	Hence, we have verified that \eqref{eq:modified-arm} holds for the pair $(i,j)$. This completes the proof of Proposition~\ref{prop:mae}.
\end{proof}

\subsection{Modified boundary four-arm events}\label{subsec:boundary}

In previous sections, we only deal with $\lambda$-arm events $\Ac^{\lambda}_{D}(z;d_1,d_2)$ under the assumption that $B_{2d_2}(z)\subseteq D$, which can be viewed as an interior case. However, in the proofs of Theorems~\ref{thm:main} and~\ref{thm:carpet1} in Section~\ref{subsec:main}, we will consider $\lambda$-chains starting from the bottom side of a given box, which produces boundary $\lambda$-two-arm events (see Definition~\ref{def:blta}). Hence, we also need to deal with the boundary case, i.e.\ when $B_{2d_2}(z)\nsubseteq D$.
It will be clear in Section~\ref{subsec:blta} that in order to analyze the boundary $\lambda$-two-arm event, we should deal with its four-arm counterpart first. Therefore, we focus on the four-arm case in this section, before turning to two-arm events in the next section.

We first introduce boundary $*$- and $\lambda$-arm events, in a similar way to boundary four-arm events (see Definition~\ref{def:bfa}). For any $0< d_1<d_2$, $z\in\wb\Hb$, and $D \supseteq B_{2d_2}(z)$, the modified boundary four-arm events are defined as
\begin{equation}
	\Ac^{*,+}_D(z;d_1,d_2):=\Ac^{*}_{D^+}(z;d_1,d_2),\quad \text{and} \quad \Ac^{\lambda,+}_D(z;d_1,d_2):=\Ac^{\lambda}_{D^+}(z;d_1,d_2)
\end{equation}
(recall that $D^+:=D\cap\Hb$). In other words, we only use the loops that stay in $\Hb$.

In this situation, we adopt the same terminology as in the interior case, using the connection number $\Nc$, the $*$-passage set $\Pc^{*}$, and the $\lambda$-passage set $\Pc^{\lambda}$ (with obvious adaptations in the definitions). Also, we abbreviate $\Ac^{\lambda,+}_{B_{2d_2}(z)}(z;d_1,d_2)$ as $\Ac^{\lambda,+}_{\loc}(z;d_1,d_2)$. 

These arm events can be analyzed in a similar fashion as in the interior case, so we remain brief, omitting large parts of the proofs. First, we have the following counterpart of Proposition~\ref{prop:m4a*}.

\begin{proposition}\label{prop:m4a+}
	For all $\alpha\in (0,1/2)$, there exists $c_{10}(\alpha)>0$ such that for all $m\ge 0$, $z = (x,0)$, $0<d_1< d_2$, and $D \supseteq B_{2d_2}(z)$,
	\begin{equation}\label{eq:mae+-N}
		\Pb( \Ac^{*,+}_D(z;d_1,d_2) \cap \{\Nc\le m\} ) \le e^{c_{10}^{m+1}}\, \pi(d_1,d_2).
	\end{equation}
\end{proposition}

\begin{proof}
	We use the same induction procedure on $m$ as in the proof of Proposition~\ref{prop:m4a*}, so we adopt the notations introduced there. When $m=0$, the event in the l.h.s. of \eqref{eq:mae+-N} is simply the boundary arm event $\Ac^+_D(z;d_1,d_2)$ (in a similar way as in \eqref{eq:*arm_m=0}), so \eqref{eq:beta3} implies
	$$\Pb( \Ac^{*,+}_D(z;d_1,d_2) \cap \{\Nc=0\} ) =\Pb( \Ac^+_D(z;d_1,d_2) ) \le c'_2\, \pi(d_1,d_2).$$
	For $m\ge 1$, using the same decomposition as in \eqref{eq:I+II}, applying the union bound, and using \eqref{eq:beta3} (analogously to \eqref{eq:(I)}) to bound the first term in the sum, we get 
	\begin{align*}
		\Pb( \Ac^{*,+}_D(z;d_1,d_2) \cap \{\Nc\le m\} ) & \le c'_2 \, \pi( 2^{i+4},2^{j-3} ) \\
		& \hspace{-2cm} + \sum_{k=i+4}^{j-4} \sum_{u \in A_{2^k,2^{k+1}}(z)} \sum_{v: v \sim^* u}  \Pb( \Ac^{*,+}_D(z;d_1,d_2) \cap \{\Nc\le m\} \cap \{ (u,v)\in \Pc^* \} ).
	\end{align*}
	The decoupling result in Lemma~\ref{lem:decom*} also holds for boundary $*$-arm events, from essentially the same proof. Thus, for $u=(x',y')\in A^+_{2^k,2^{k+1}}(z)$ with $2^r\le y'<2^{r+1}$, we have 
	\begin{align}\notag
		\Pb( & \Ac^{*,+}_D(z;d_1,d_2) \cap \{\Nc\le m\} \cap \{ (u,v)\in\Pc^*  \} )\\ \notag
		& \le c_4\, \Pb(  \Ac^{*,+}_{\loc}(z;2^i,2^{k-3})\cap \{\Nc_1\le m-1\} ) \cdot \Pb( \Ac^{*,+}_D(z;2^{k+2},2^j) \cap \{\Nc_2\le m-1\} ) \\ \label{eq:decoup+}
		& \quad \cdot \Pb( \Ac^{*,+}_{\loc}((x',0);2^{r+1} \wedge 2^{k-2},2^{k-2})\cap \{\Nc_3\le m-1\} ) \cdot \Pb( \Ac^*_{\loc}(u;1,2^{r-1} \wedge 2^{k-4})\cap \{\Nc_4\le m-1\} ).
	\end{align}
	Note that in the above decomposition, the last event on the right-hand side is an interior $*$-arm event, so that we can bound it directly by using Proposition~\ref{prop:m4a*}. Then, we apply the induction hypothesis three times, with $m-1$, to the first three events (which are all boundary $*$-arm events), respectively. This yields the following upper bound, for some universal constant $C>0$:
	\[
	\Pb( \Ac^{*,+}_D(z;d_1,d_2) \cap \{\Nc\le m\} ) \le c'_2 \, \pi( 2^{i+4},2^{j-3} ) + C e^{c_5^m} e^{3c_{10}^m} \pi(2^i,2^j).
	\]
	By choosing $c_{10}\ge c_5$ large enough, we can thus conclude the proof by proceeding as we did toward the end of the proof of Proposition~\ref{prop:m4a*}.
\end{proof}

We can then derive the following boundary analog of Proposition~\ref{prop:mae}, by using the same induction argument as for that result, but now with Proposition~\ref{prop:m4a+} as an input (instead of Proposition~\ref{prop:m4a*}).

\begin{proposition}\label{prop:bmae}
	For all $\alpha\in (0,1/2)$, 
	there exist constants $c_{11}(\alpha)>0$ and $\lambda_1(\alpha)\ge\lambda_0(\alpha)$ such that the following holds. For all $\lambda \geq \lambda_1$, $z = (x,0)$, $0<d_1< d_2$, and $D \supseteq B_{2d_2}(z)$,
	\begin{equation}\label{eq:b-modified-arm}
		\Pb(\Ac^{\lambda,+}_D(z;d_1,d_2))\le c_{11}\, \pi(d_1,d_2).
	\end{equation}
\end{proposition}

\begin{proof}
	This result can be proved by using almost the same induction procedure as in the proof of Proposition~\ref{prop:mae}, with only a minor modification in the decoupling result, as we already did in \eqref{eq:decoup+}. More precisely, for $u=(x',y')\in A^+_{2^k,2^{k+1}}(z)$ with $2^r\le y'<2^{r+1}$, we use the following estimate (which can be obtained from similar arguments as for Lemma~\ref{lem:decom}):
	\begin{align}\notag
		\Pb( \Ac^{\lambda,+}_D(z;2^i,2^j) \cap \{ (u,v)\in\Pc^{\lambda}  \} )
		&\le c_8\, \rho^\lambda\, \Pb( \Ac^{\lambda,+}_{\loc}(z;2^i,2^{k-3}) ) \, \Pb( \Ac^{\lambda,+}_D(z;2^{k+2},2^j) ) \\ \label{eq:lbd+}
		& \quad  \cdot \Pb( \Ac^{\lambda,+}_{\loc}((x',0);2^{r+1}\wedge2^{k-2},2^{k-2}) ) \Pb( \Ac^{\lambda}_{\loc}(u;8,2^{r-1}\wedge2^{k-4}) ).
	\end{align}
	As in \eqref{eq:decoup+}, the last event on the right-hand side is an interior $\lambda$-arm event, which can be bounded directly thanks to Proposition~\ref{prop:mae}, and we can bound the first three boundary $\lambda$-arm events by applying the induction hypothesis three times. Hence, we have reached a similar situation as when we showed Proposition~\ref{prop:mae}, and we can thus complete the proof in the same way.
\end{proof}

Finally, we consider the general case of boundary $\lambda$-arm events when the center is at an arbitrary place in $\wb\Hb$, which will be needed later.

\begin{proposition}\label{prop:bbmae}
	For all $\alpha\in (0,1/2)$, there exists $c_{12}(\alpha)>0$ such that the following holds, with $\lambda_1(\alpha)$ as in the statement of Proposition~\ref{prop:bmae}. For all $\lambda \geq \lambda_1$, $z\in\wb\Hb$, $0<d_1< d_2$, and $D \supseteq B_{2d_2}(z)$,
	\begin{equation}\label{eq:bb-modified-arm}
		\Pb(\Ac_{D^+}^{\lambda}(z;d_1,d_2))\le c_{12}\, \pi(d_1,d_2).
	\end{equation}
\end{proposition}

\begin{proof}[Proof of Proposition~\ref{prop:bbmae}]
	We can assume, without loss of generality, that $z=(0,r)$ with $r\ge 0$. We first note that it is sufficient to consider the case $24d_1\le d_2$, since \eqref{eq:bb-modified-arm} holds for any $24d_1> d_2$ by choosing some large constant $c_{12}$. Hence, in the following, we restrict to the case $24d_1\le d_2$.
	\begin{itemize}
		\item If $r\le 4d_1$, then $B(z,d_1)\subseteq B(0,6d_1)$, and $\Ac_{D^+}^{\lambda}(z;d_1,d_2)\subseteq \Ac^{\lambda}_{D^+}(6d_1,d_2)=\Ac^{\lambda,+}_{D}(6d_1,d_2)$ by monotonicity. Hence, \eqref{eq:bb-modified-arm} follows immediately from Proposition~\ref{prop:bmae}. 
		\item If $r > d_2/6$, then we have $D^+\supseteq B_{d_2/6}(z)$. Hence, $\Ac_{D^+}^{\lambda}(z;d_1,d_2)\subseteq \Ac_{D^+}^{\lambda}(z;d_1,d_2/12)$ and the result follows from Proposition~\ref{prop:mae}.
		\item If $4d_1<r\le d_2/6$, then we use the following quasi-multiplicativity estimate for $\lambda$-arm events (which can be obtained in a similar way as Lemma~\ref{lem:quasi-lbd}):
		\begin{equation}\label{eq:par-bdy}
			\Pb(\Ac_{D^+}^{\lambda}(z;d_1,d_2)) \lesssim_{\alpha} \Pb(\Ac^{\lambda}_{\loc}(z;d_1,r/2))\,  \Pb( \Ac^{\lambda,+}_D(3r,d_2)).
		\end{equation}
		By applying Propositions~\ref{prop:mae} and \ref{prop:bmae} to the first and the second factors, respectively, we obtain that
		\begin{equation}\label{eq:par-bdy2}
			\Pb(\Ac^{\lambda}_{\loc}(z;d_1,r/2))\,  \Pb( \Ac^{\lambda,+}_D(3r,d_2)) \lesssim_{\alpha} \pi(d_1,r/2) \, \pi(3r,d_2),
		\end{equation}
		which is bounded by a constant multiple of $\pi(d_1,d_2)$. We can thus conclude by combining \eqref{eq:par-bdy} and \eqref{eq:par-bdy2}.
	\end{itemize}
\end{proof}

\subsection{Modified boundary two-arm events}\label{subsec:blta}

As we mentioned earlier, in order to deal with boundary crossing events for the occupation field, we need to make use of boundary $\lambda$-two-arm events, which are modified versions of the boundary two-arm events $\Bc_D(l,d)$ (see Definition~\ref{def:bta}). Hence, we adapt the definitions in Section~\ref{subsec:moe} as follows.

\begin{definition}[Modified boundary two-arm event]\label{def:blta}
	For $0<d_1<d_2$, $z=(x,0)$, and $D \supseteq B_{2d_2}(z)$, let $\Bc^*_D(z;d_1,d_2)$ be the \emph{boundary $*$-two-arm event} that there exists a $*$-chain in $\Lc_{D^+}$ crossing $A^+_{d_1,d_2}(z)$. Moreover, for $\lambda>0$, let $\Bc^\lambda_D(z;d_1,d_2)$ be the \emph{boundary $\lambda$-two-arm event} that there exists a $*\lambda*$-chain in $\Lc_{D^+}$ crossing $A^+_{d_1,d_2}(z)$.
\end{definition}

We first deal with the boundary $*$-two-arm event, with the estimates for boundary two-arm events (Proposition~\ref{prop:2_arm_proba_bdy}) and boundary $*$-arm events (Proposition~\ref{prop:m4a+}) as necessary ingredients.
\begin{proposition}\label{prop:b2am}
	For all $\alpha\in (0,1/2)$ and $\eps>0$, there exists $c_{13}(\alpha,\eps)>0$ such that for all $m\ge 0$, $z = (x,0)$, $0<d_1< d_2$, and $D \supseteq B_{2d_2}(z)$,
	\begin{equation}\label{eq:b2am}
		\Pb( \Bc^{*}_D(z;d_1,d_2) \cap \{\Nc\le m\} ) \le e^{c_{13}^{m+1}}\, (d_2/d_1)^{-\xi^{2,+}(\alpha)+\eps},
	\end{equation}
	where $\xi^{2,+}(\alpha) > 1$ is the boundary two-arm exponent from \eqref{eq:2_arm_exponent_bdy}. 
\end{proposition}
\begin{proof}
	The proof is similar to that of Proposition~\ref{prop:m4a+}, by induction on $m$. For the first step of the induction, i.e.\ the case $m=0$, we use Proposition~\ref{prop:2_arm_proba_bdy} as an input. For $m\ge 1$, we use the following decoupling result (similar to Lemma~\ref{lem:decom*}): for all $u\in A^+_{2^k,2^{k+1}}(z)$,
	\begin{align}\notag
		\Pb( \Bc^{*}_D & (z;2^i,2^j)  \cap \{\Nc\le m\} \cap \{ (u,v)\in\Pc^{*}  \} )\\ \notag
		& \le c_4\, \Pb( \Bc^{*}_{B_{2d_2}(z)}(z;2^i,2^{k-3}) \cap \{\Nc\le m-1\} ) \, \Pb( \Bc^{*}_D(z;2^{k+2},2^j) \cap \{\Nc\le m-1\} ) \\ \label{eq:btae}
		 & \hspace{1cm} \cdot \Pb \Big( \Ac^{*,+}_{B_{2^{k-1}}(u)}(u;8,2^{k-2}) \cap \{\Nc\le m-1\} \Big),
	\end{align}
	where the last term on the r.h.s of \eqref{eq:btae} is at most $e^{c_{10}^{m}}\, \pi(8,2^{k-2})$, by Proposition~\ref{prop:m4a+} (in fact, we use a stronger form of this estimate, with an arbitrary center in $\wb\Hb$, which can be proved in the same way as Proposition~\ref{prop:bbmae}). By using the induction hypothesis to deal with the first two events on the r.h.s of \eqref{eq:btae}, and applying the union bound as before, we conclude the proof.
\end{proof}

In a similar way as for Proposition~\ref{prop:bmae}, we can obtain the following result, where the inputs are now Propositions~\ref{prop:bbmae} and~\ref{prop:b2am}.

\begin{proposition}\label{prop:b2lam}
	For all $\alpha\in (0,1/2)$ and $\eps>0$, 
	there exist constants $c_{14}(\alpha,\eps)>0$ and $\bar\lambda(\alpha,\eps)>0$ such that the following holds. For all $\lambda \geq \bar\lambda$, $z = (x,0)$, $0<d_1 < d_2$, and $D \supseteq B_{2d_2}(z)$,
	\begin{equation}\label{eq:b2lam}
		\Pb(\Bc^{\lambda}_D(z;d_1,d_2))\le c_{14}\, (d_2/d_1)^{-\xi^{2,+}(\alpha)+\eps}.
	\end{equation}
\end{proposition}

\begin{proof}
First, we can use Proposition~\ref{prop:b2am} to deal with the case when no $\lambda$-passage edge exists in the bulk (analogously to when we used the event $\Ec_1$ in the proof of Proposition~\ref{prop:mae}, see \eqref{eq:Ec12} and below). We can then apply the same induction procedure as in the proof of Proposition~\ref{prop:bmae}, but now with the following decoupling result: for all $u\in A^+_{2^k,2^{k+1}}(z)$,
\begin{align}\notag
	\Pb( \Bc^{\lambda}_D(z;2^i,2^j) \cap \{ (u,v)\in\Pc^{\lambda}  \} )
	&\le c_8\, \rho^\lambda\, \Pb( \Bc^{\lambda}_{\loc}(z;2^i,2^{k-3}) ) \, \Pb( \Bc^{\lambda}_D(z;2^{k+2},2^j) ) \\ \label{eq:blbd+}
	& \hspace{2cm}  \cdot \Pb\bigg( \Ac^{\lambda}_{B^+_{2^{k-1}}(u)}(u;8,2^{k-2}) \bigg).
\end{align}
The first two events on the r.h.s of \eqref{eq:blbd+} can be handled by using the induction hypothesis, and we can use Proposition~\ref{prop:bbmae} to deal with the last event. The remaining steps are similar, so we omit them.
\end{proof}

\subsection{Proof of Theorems~\ref{thm:main} and~\ref{thm:carpet1}} \label{subsec:main}

In this section, we use Propositions~\ref{prop:mae} and \ref{prop:b2lam} to establish Theorems~\ref{thm:main} and~\ref{thm:carpet1}. Let $\alpha \in (0,1/2)$ and $D \subseteq \Zb^2$. Recall that for any $\lambda > 0$, a vertex $z \in D$ is called $\lambda$-open if it has occupation $X_D(z)\le\lambda$ (for the RWLS $\Lc_D$ with intensity $\alpha$ in $D$), and $\lambda$-closed otherwise. We remind the reader that the boundary and the bulk crossing events are denoted by $\Ec_{\alpha,N}(\lambda)$ (see \eqref{eq:E-lbd}) and $\Ec'_{\alpha,N}(\lambda)$ (see \eqref{eq:tE-lbd}), respectively, and the corresponding critical values by $\lambda_*(\alpha)$ \eqref{eq:crit_val1} and $\lambda'_*(\alpha)$ \eqref{eq:crit_val2}.
We also recall the events $\Fc_{\alpha,N}(\lambda)$ and $\Fc'_{\alpha,N}(\lambda)$ in Definition~\ref{def:carpet_cross}.

Let us make one last observation. For any two domains $D \subseteq \tilde{D} \subseteq \Zb^2$, the RWLS $\Lc_D$ and $\Lc_{\tilde{D}}$ are naturally coupled, so we have $X_D(z) \leq X_{\tilde{D}}(z)$ for all $z \in D$ (but there might be loops in $\tilde{D}$ which visit $z$, but are not fully contained in $D$). Hence, we also have some obvious monotonicity for the boundary $\lambda$-two-arm event: $\Bc^{\lambda}_D(z;d_1,d_2)\subseteq \Bc^{\lambda}_{\tilde D}(z;d_1,d_2)$.

\begin{proof}[Proof of Theorems~\ref{thm:main} and~\ref{thm:carpet1}]
	We fix some value $\alpha\in (0,1/2)$ during the whole proof. First, from the monotonicity mentioned above, we have the following. For each $\lambda > 0$, $\Ec'_{\alpha,N}(\lambda) \subseteq \Ec_{\alpha,N/2}(\lambda)$ for all $N \geq 1$, so $\lambda'_*(\alpha) \geq \lambda_*(\alpha)$ by definition.
	
	We first show \eqref{eq:F_alpha1} which will also imply \eqref{eq:main-cr}. By duality, if $\Fc_{\alpha,N}(\lambda)$ does not occur, then there exists a $\lambda$-chain in $\Lc_N$ (the RWLS in $B_N$ with intensity $\alpha$) connecting the top and bottom sides of $B_N$ (See Figure~\ref{fig:blocking_path2}). 
	By the union bound, translation invariance, and the monotonicity of boundary $\lambda$-two-arm events, we get: for all $\lambda\ge\bar\lambda(\alpha,\eps)$,
	\begin{equation}\label{eq:cr-1}
		\Pb \Big( \Ec_{\alpha,N}(\lambda)^c \Big) \le \sum_{|x|\le N} \Pb \Big( \Bc^\lambda_{B_{2N}((x,0))}((x,0);1,N) \Big) \lesssim N^{1-\xi^{2,+}(\alpha)+\eps},
	\end{equation}
	where we used Proposition~\ref{prop:b2lam} in the last inequality. This finishes the proof of \eqref{eq:F_alpha1} and also \eqref{eq:main-cr}.

	\begin{figure}[t]
		\centering
		\includegraphics[width=.6\textwidth]{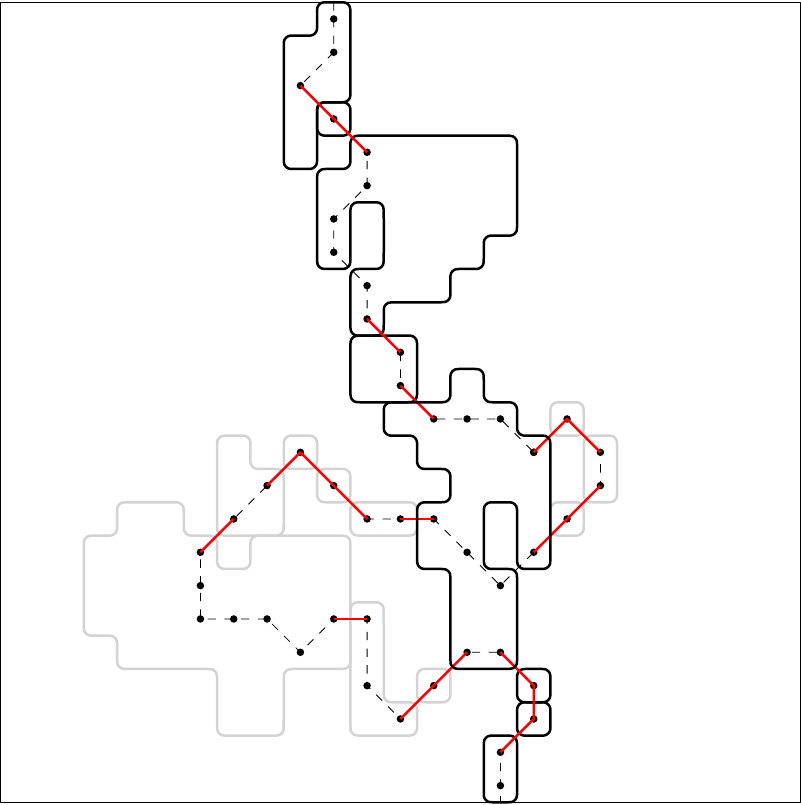}
		\caption{A blocking vertical path with high occupation may visit the same cluster of loops multiple times. From such a chain, we can extract a $\lambda$-chain, where all clusters are distinct, by simply deleting ``loops'' of clusters. For example, the chain shown on this figure can be reduced by keeping only the clusters shown in black (i.e.\ removing the grey ones), and the successive $\lambda$-passage edges connecting them.}
		\label{fig:blocking_path2}
	\end{figure}

	We now turn to the proof of \eqref{eq:F_alpha2}, which will also imply $\lambda'_*(\alpha)<\infty$. Consider any $\lambda \geq \lambda_0(\alpha)$, where $\lambda_0(\alpha)$ appears in Proposition~\ref{prop:mae}. For $N \geq 1$, consider the event $F_N := \{$there exists a cluster in $\Lc_N$ with a diameter $> N/100 \}$. There exists a universal constant $c' > 0$ such that
	\begin{equation}\label{eq:FN'}
		\text{for all $N \geq 1$,} \quad \Pb( F_N^c ) \ge c'
	\end{equation}
	(this follows immediately from the connection between the RWLS and CLE in the scaling limit, see Lemma~\lemclustersize{} in \cite{GNQ2024a}).
	
	We then make the following observation, using again duality. On the event $\Fc'_{\alpha,N}(\lambda)^c\cap F_N^c$, the (interior) $\lambda$-arm event $\Ac^{\lambda}_{B_N}(z;1,N/4)$ occurs at some $z\in B_{N/2}$ (since the above-mentioned $\lambda$-chain contains necessarily a $\lambda$-passage edge in $B_{N/2}$). From the union bound, and then Proposition~\ref{prop:mae} (recall that $\lambda \geq \lambda_0(\alpha)$ by assumption), we obtain that
	$$\Pb \Big( \Fc'_{\alpha,N}(\lambda)^c\cap F_N^c \Big) \le \sum_{z\in B_{N/2}}  \Pb \Big( \Ac^{\lambda}_{B_N}(z; 1, N/4) \Big) \lesssim_{\alpha} N^2 \cdot N^{-2-\beta} \stackrel{N \to \infty}{\longrightarrow} 0,$$
	which, combined with \eqref{eq:FN'}, implies that
	$$\liminf_{N \to \infty} \Pb \big( \Fc'_{\alpha,N}(\lambda) \big) \geq \liminf_{N \to \infty} \Pb \big( \Fc'_{\alpha,N}(\lambda) \cap F_{N}^c \big) \geq c' > 0.$$
Hence, we deduce \eqref{eq:F_alpha2} and $\lambda'_*(\alpha) \leq \lambda < \infty$. Note that by \eqref{eq:crit_val2}, we even obtain that $\lambda'_*(\alpha) \leq \lambda_0(\alpha)$. In fact, in the statements of Theorems~\ref{thm:main} and~\ref{thm:carpet1}, we can choose $\bar\lambda_0:=\bar\lambda\vee\lambda_0$.
	
	There remains to show that $\lambda_*(\alpha)>0$, which turns out to be the easy direction. We want to check that when $\lambda > 0$ is small enough,
	\begin{equation} \label{eq:lim_crossingprobab}
		\lim_{N\rightarrow\infty} \Pb( \Ec_{\alpha,N}(\lambda) ) = 0.
	\end{equation}
	To this aim, we prove that the percolation configuration of the occupation field can be dominated by a configuration of Bernoulli percolation, with some subcritical value of the parameter. 
	For all $z\in B_N$, we introduce $Y_z := \mathbbm{1}_{X_N(z)\le \lambda}$ (i.e., $Y_z$ is the indicator function of the event $\{z$ is $\lambda$-open$\}$).
	It follows from the definition of the occupation field (see Section~\ref{subsec:RWLS}) that for each $z\in B_N$,
	$$\Pb( Y_z=1 \mid \sigma(Y_u, u\neq z) ) \le \Pb( R_z\le \lambda ),$$
	where $R_z$ has $\Gamma(\alpha,1)$ distribution. We deduce that
	$$\Pb( Y_z=1 \mid \sigma(Y_u, u\neq z) ) \leq \frac{1}{\Gamma(\alpha)} \int_{0}^{\lambda} x^{\alpha-1} e^{-x} dx \leq \frac{\lambda^{\alpha}}{\alpha \Gamma(\alpha)}.$$
	Let $p_c = p_c^{\text{site}}(\Zb^2) \in (0,1)$ be the critical threshold for Bernoulli site percolation on $\Zb^2$. We define $\lambda>0$ via the relation
	\begin{equation} \label{eq:choice_lambda}
		\frac{\lambda^{\alpha}}{\alpha \Gamma(\alpha)} = \frac{p_c}{2}.
	\end{equation}
	
	Clearly, for such a choice of $\lambda$, the collection $(Y_z)_{z\in B_N}$ is stochastically dominated by Bernoulli site percolation with parameter $p_c/2$ on $B_N$, which is in subcritical regime. Moreover, $\Ec_{\alpha,N}(\lambda)$ implies that there exists a path in $B_N$ connecting the left and right sides of that box, along which the function $Y$ takes value $1$. For subcritical Bernoulli percolation, the existence of such an open crossing in $B_N$ has a probability tending to $0$ exponentially fast as $N \to \infty$. Hence, we obtain \eqref{eq:lim_crossingprobab} for the value $\lambda$ as in \eqref{eq:choice_lambda}, which finally allows us to conclude the proof of Theorem~\ref{thm:main}.
\end{proof}

\section{Critical case $\alpha=1/2$} \label{sec:critical-metric}

In this section, we deal with the critical case $\alpha=1/2$ for the random walk loop soup, for which our reasonings need to be adjusted. From \eqref{eq:arm_exponent} and \eqref{eq:alpha_kappa}, we know that when $\alpha=1/2$, which corresponds to $\kappa = 4$ for CLE, the exponent $\xi$ associated with the four-arm event is exactly equal to $2$ (instead of being $> 2$ in the subcritical case $\alpha < 1/2$). Because of this, the induction arguments from Section~\ref{sec:pf-thm} fail as they are, since the probabilities that arise are no longer summable. In this case, we need to let $\lambda$ increase slowly with $N$, so as to make the arguments from the subcritical case work.

More concretely, we prove the following theorem in Sections~\ref{subsec:inside-critical} and \ref{subsec:bdy-critical}. Recall that the events $\Fc_{\alpha,N}(\lambda)$ and $\Fc'_{\alpha,N}(\lambda)$ were introduced in Definition~\ref{def:carpet_cross}.
\begin{theorem}\label{thm:carpet2}
	There exist $c_0,a_0>0$, such that for all $a\ge a_0$,
	\begin{align}
		\label{eq:critical_carpet1}
		&\Pb( \Fc_{1/2,N}(a\log\log N) )=1-O( (\log N)^{-c_0 a}) \quad \text{as} \quad N\to \infty.
	\end{align}
	Moreover, for all $a\ge a_0$,
	\begin{align}
		\label{eq:critical_carpet2}
		&\liminf_{N \to \infty}\Pb( \Fc'_{1/2,N}(a\log\log N) )>0.
	\end{align}
\end{theorem}
Because upper bounds on the probabilities of four-arm events happen to be a key input in the various induction arguments, in particular in the proofs of Propositions~\ref{prop:critical-lbd-1} and \ref{prop:critical-lbd-2} below, the effectiveness of our method depends heavily on how precisely such events can be estimated. Hence, the strengthened estimates obtained in \cite{bi2025arm} for $\alpha=1/2$ (recalled as Propositions~\ref{prop:critical four arm} and \ref{prop:critical 2_arm_proba_bdy} above) play a crucial role.

Finally, in Section~\ref{subsec:toGFF}, we translate the results for the critical RWLS to other processes by making use of various isomorphism theorems. More precisely, we establish Theorem~\ref{thm:GFF_carpet} (which implies Theorem~\ref{thm:GFF}) and Theorem~\ref{thm:thick}, using Theorem~\ref{thm:carpet2}.

Since we focus on the critical case throughout this section, we tacitly fix $\alpha=1/2$.

\subsection{Bulk crossing event: proof of \eqref{eq:critical_carpet2}} \label{subsec:inside-critical}

We first analyze the bulk crossing event $\Fc'_{N}(a\log\log N) := \Fc'_{1/2,N}(a\log\log N)$, at a level of order $\log\log N$, and prove \eqref{eq:critical_carpet2}.

For any $z\in \Zb^2$, $0<d_1<d_2$, and $D \supseteq B_{2d_2}(z)$, let $\Ac_{D}(z;d_1,d_2)$ be the four-arm event for the critical RWLS. 
By combining Proposition~\ref{prop:critical four arm} with the proof of Lemma~\ref{lem:m4a}, we can obtain that there exists $c'_2>0$ such that for all $z$, $d_1$, $d_2$, and $D$ as above, we have: for all $m\ge 1$,
\begin{equation}\label{eq:o2}
	\Pb( \Ac^*_D(z;d_1,d_2) \cap \{\Nc=m\} ) \le \Big(c'_2\,\log \Big( \frac{d_2}{d_1} \Big) \Big)^m\, \Big(\frac{d_2}{d_1}\Big)^{-2}.
\end{equation}
Using that upper bound, and the induction arguments in the proof of Proposition~\ref{prop:mae}, we can then derive the following estimate for $\lambda$-arm events at criticality, which is an important ingredient in the proof of \eqref{eq:critical_carpet2}. Throughout this section, we consider the constant $\rho\in (0,1)$ from the statement of Lemma~\ref{lem:decom}, and we let $\delta := - \log \rho > 0$, so that $\rho = e^{-\delta}$.

\begin{proposition}\label{prop:critical-lbd-1}
	There exist $n_0>0$ and $b_0>0$ such that for all $z\in \Zb^2$, $N\ge n_0$, $0<d_1<d_2<N$, and $D \supseteq B_{2d_2}(z)$,
	\begin{equation}\label{eq:critical-lbd-1}
		\Pb\big(\Ac^{\lambda}_D(z;d_1,d_2) \big)\le  b_0\, \Big(\log \Big( \frac{d_2}{d_1} \Big) \Big)^4 \Big(\frac{d_2}{d_1}\Big)^{-2} \quad \text{for $\lambda\ge \frac{14}{\delta} \log\log N$}.
	\end{equation}
\end{proposition}

\begin{proof}[Proof of Proposition~\ref{prop:critical-lbd-1}]
	We use similar arguments as for Proposition~\ref{prop:mae}, following the induction procedure in the proof of that result, and adopting the same notations. In particular, recall that we started by assuming that $d_1$ and $d_2$ are of the form $d_1=2^i$ and $d_2=2^j$, for any two integers $0 \leq i < j$. As before, we first choose $b_0 \ge 2^{20} (\log 2)^{-4}$ so that \eqref{eq:critical-lbd-1} holds for all $(i,j)$ with $1\le j-i\le 10$. 
	
	
	Now, suppose that the pair $(i,j)$ satisfies $j-i\ge 11$, and assume that \eqref{eq:critical-lbd-1} holds for all pairs $(i',j')$ with $j'<j$, and also for all $(i',j)$ with $i'>i$. It remains to show that \eqref{eq:critical-lbd-1} holds for the pair $(i,j)$ under this assumption. Recall that we can write $\Pb(\Ac^{\lambda}_D(z;2^i,2^j))\le \Pb( \Ec_1 ) + \Pb( \Ec_2 )$, where $\Ec_1$ and $\Ec_2$ were defined in \eqref{eq:Ec12}. In the remainder of the proof, we let $\sigma_1,\sigma_2,\sigma_3$ be universal constants. Using \eqref{eq:o2}, we can replace \eqref{eq:Ec1} by
	\begin{equation} \label{eq:prob_E1}
		\Pb( \Ec_1 )\le \Pb( \Ac^*_{D}(z,2^{i+4},2^{j-3})\cap \{ \Nc\le 4 \} ) \le \sigma_1\, (j-i)^4\, 2^{-2(j-i)}.
	\end{equation}
	As for $\Ec_2$, we use Lemma~\ref{lem:decom}, and then apply the induction hypothesis three times, as before, to get (similarly to \eqref{eq:lbd-E2})
	\begin{align} \notag
		\Pb(\Ec_2) & \leq \sum_{k=i+4}^{j-4} \sum_{u \in A_{2^k,2^{k+1}}(z)} \sum_{v: v \sim^* u} c_8\, \rho^{\lambda}\,  \Pb( \Ac^{\lambda}_{\loc}(z;2^i,2^{k-3}) ) \,\Pb( \Ac^{\lambda}_D(z;2^{k+2},2^j) )\, \Pb( \Ac^{\lambda}_{\loc}(u;8,2^{k-2}) ),\\ \label{eq:prob_E2}
		& \le \sum_{k=i+4}^{j-4} \sum_{u \in A_{2^k,2^{k+1}}(z)} \sum_{v: v \sim^* u} c_8\, \rho^{\lambda} \, \sigma_2 \cdot b_0\, (k-i)^4\, 2^{-2(k-i)} \cdot b_0\, (j-k)^4\, 2^{-2(j-k)}\, \cdot b_0\, k^4\, 2^{-2k}.
	\end{align}
	Since $d_2=2^j<N$ and $i\le k\le j$, we have 
	\begin{equation} \label{eq:prob_E2_sum}
		\sum_{k=i+4}^{j-4} \sum_{u \in A_{2^k,2^{k+1}}(z)} \sum_{v: v \sim^* u} (k-i)^4 (j-k)^4 k^4 2^{-2k} \le \sigma_3\, (\log N)^{13}.
	\end{equation}
	For $\lambda\ge (14/\delta)\log\log N$, we have (using that $\rho=e^{-\delta}$): $\rho^{\lambda} = e^{- \delta \lambda} \le e^{-\delta \cdot (14/\delta)\log\log N} = (\log N)^{-14}$. Hence, we deduce from \eqref{eq:prob_E2} and \eqref{eq:prob_E2_sum} that
	\begin{equation} \label{eq:prob_E2'}
		\Pb(\Ec_2) \le c_8 \, \rho^{\lambda} \, \sigma_2 \, b^3_0\,2^{-2(j-i)}\, \sigma_3\, (\log N)^{13} \le \big(c_8\, \sigma_2\,\sigma_3\, b^2_0\, (\log N)^{-1}\big)\, b_0\, 2^{-2(j-i)}.
	\end{equation}
	
	To make the induction work, we first choose $b_0$ large enough so that $\sigma_1\le \frac{b_0}{2} (\log 2)^4$. For such a $b_0$, we then take $n_0$ sufficiently large so that $c_8\,\sigma_2\,\sigma_3\, b^2_0\, (\log n_0)^{-1}\le \frac{1}{2}$. 
    This allows us to complete the proof by combining \eqref{eq:prob_E1} and \eqref{eq:prob_E2'}.
\end{proof}

\begin{remark}
	By using a similar strategy as in the proof of  Proposition~\ref{prop:critical-lbd-1}, we can also show the following, slightly different, result. There exist $a_0>0$ and $b_0>0$ such that for all $z\in \Zb^2$, $0<d_1<d_2$, and $D \supseteq B_{2d_2}(z)$,
	\begin{equation}\label{eq:critical-lbd}
		\Pb\big(\Ac^{\lambda}_D(z; d_1,d_2) \big)\le  b_0\, \Big(\log \frac{d_2}{d_1} \Big)^4 \Big(\frac{d_2}{d_1}\Big)^{-2} \quad \text{for $\lambda \ge a_0 \log\log \Big(\frac{d_2}{d_1}\Big)$}.
	\end{equation}
	Compared with Proposition~\ref{prop:critical-lbd-1}, this bound allows one to consider $\lambda$ of the form $a \log\log(d_2/d_1)$, which is potentially quite lower than the threshold in \eqref{eq:critical-lbd-1}. However, it is only valid for large enough values of the constant $a$.
\end{remark}

We now use Proposition~\ref{prop:critical-lbd-1} to prove \eqref{eq:critical_carpet2}.

\begin{proof}[Proof of \eqref{eq:critical_carpet2}]
	We use the same arguments as for $\lambda'_*(\alpha)<\infty$ in the last part of the proof of Theorems~\ref{thm:main} and~\ref{thm:carpet1}. 
    Let $n_0$ and $b_0$ be as in the statement of Proposition~\ref{prop:critical-lbd-1}. For any $N \geq n_0$, let $\lambda= \frac{14}{\delta} \log\log N$. We have (recall that $F_N$ was defined above \eqref{eq:FN'}, and $\Ac^{\lambda}_{B_N}(z,w;N/4)$ above Lemma~\ref{lem:mod-z})
	\begin{align*}
		\Pb( (\Fc'_N(\lambda) \cup F_N)^c ) &\ \le \sum_{z \in B_{N/2}} \sum_{w: w \sim^* z} \Pb( \Ac^{\lambda}_{B_N}(z,w;N/4) \cap \{ X_{B_N}(z)>{\lambda} \}  ) \\
		&\  \lesssim \sum_{z \in B_{N/2}} \sum_{w: w \sim^* z} \Pb( \Ac^{\lambda}_{B_N}(z;8,N/4) )\, \rho^{\lambda},
	\end{align*}
	where in the second inequality we used Lemma~\ref{lem:mod-z}, combined, successively, with \eqref{eq:uv-u}, Lemma~\ref{lem:Gauss}, and Lemma~\ref{lem:delete-dot} to get the following decoupling result:
	\begin{align}\notag
		\Pb\big( \Ac_{B_N}^{\lambda}(z,w;N/4) \cap \{ X_{B_N}(z)>\lambda \} \big) &\lesssim \Pb\big( \Ac_{B_N}^{\lambda}(\dot{z};8,N/4) \big) \, \Pb( \Lc_{B_N}(\{z,w\})=\emptyset )\, \Pb( X^w_{B_N}(z)>\lambda )\\ \notag
		&\lesssim \Pb\big( \Ac_{B_N}^{\lambda}(\dot{z};8,N/4) \big) \, \Pb( \Lc^z_{B_N}=\emptyset ) \, \rho^{\lambda}\\ \label{eq:decouple}
		&\lesssim \Pb\big( \Ac_{B_N}^{\lambda}(z;8,N/4) \big) \, \rho^{\lambda}.
	\end{align}
	By using Proposition~\ref{prop:critical-lbd-1}, which we are allowed to do from our choice of $\lambda$, we then obtain (since $\rho^{\lambda} = e^{-\delta\lambda} = e^{-\delta \frac{14}{\delta} \log \log N} = (\log N)^{-14}$)
	\begin{equation*}
		\Pb( (\Fc'_N(\lambda) \cup F_N)^c ) \lesssim N^2\cdot b_0\, (\log N)^4 N^{-2} \cdot (\log N)^{-14} \lesssim (\log N)^{-10}.
	\end{equation*}
	Therefore, for our particular choice of $\lambda$, $\Pb( (\Fc'_N(\lambda)\cup F_N)^c )$ converges to $0$ as $N \to \infty$. This, combined with \eqref{eq:FN'} (this property remains true in the critical case), finishes the proof of \eqref{eq:critical_carpet2}.
\end{proof}

\subsection{Boundary crossing event: proof of \eqref{eq:critical_carpet1}}\label{subsec:bdy-critical}

We now deal with the boundary crossing event $\Fc_{N}(a\log\log N) := \Fc_{1/2,N}(a\log\log N)$, at a level of order $\log\log N$, and prove \eqref{eq:critical_carpet1}.

In order to do so, we need to estimate boundary $\lambda$-two-arm events in the critical case, which requires an a priori estimate for its four-arm counterpart (analogously to Section~\ref{subsec:blta}). As in Section~\ref{subsec:inside-critical}, we first gather the ingredients that are needed. Since $\xi^+(1/2)=4$, it follows from \eqref{eq:arm_proba_bdy} that for any $\eps>0$, there exists $c_2^+(\eps)>0$ such that for all $0 < d_1< d_2$, and $D \supseteq B_{2d_2}$,
\begin{equation}\label{eq:b1}
	\Pb( \Ac^+_D(d_1,d_2) ) \le c_2^+\, \Big(\frac{d_2}{d_1}\Big)^{-4+\eps}.
\end{equation}
Regarding boundary $*$-arm events, we have the following result, similarly to Lemma~\ref{lem:m4a}. For any $\eps>0$, there exists $c_2'(\eps) > 0$ such that for all $m\ge 1$, $0 < 2d_1\leq d_2$, and $D \supseteq B_{2d_2}$,
\begin{equation}\label{eq:b2}
	\Pb( \Ac^{*,+}_D(d_1,d_2) \cap \{\Nc=m\} ) \le \Big(c'_2\,\log \Big( \frac{d_2}{d_1} \Big) \Big)^m\, \Big(\frac{d_2}{d_1}\Big)^{-4+\eps}.
\end{equation}

Using these estimates, we can derive the following upper bound.
\begin{proposition}\label{prop:critical-lbd-2}
	There exist $n_1 > 0$ and $b_1 > 0$ such that for all $N\ge n_1$, $0< d_1<d_2<N$, and $D \supseteq B_{2d_2}$,
	\begin{equation}\label{eq:critical-lbd-2}
		\Pb\big(\Ac^{\lambda,+}_D(d_1,d_2) \big)\le  b_1\, \Big(\frac{d_2}{d_1}\Big)^{-3} \quad \text{for $\lambda\ge \frac{14}{\delta}\log\log N$}.
	\end{equation}
\end{proposition}

\begin{proof}
	The proof is similar to that of Proposition~\ref{prop:critical-lbd-1}. Again, we let $\sigma_1, \sigma_2, \sigma_3$ denote universal constants. First, we choose $m_1$ large enough so that
	\begin{equation} \label{eq:m_1}
		\text{for all $m \geq m_1$,} \quad m^4 \le 2^{m/2}.
	\end{equation}
	We then take $b_1$ satisfying $b_1 \geq 2^{3 m_1}$, which yields \eqref{eq:critical-lbd-2} for all $1\le j-i\le m_1$.

    Now, assume that the pair $(i,j)$ satisfies $j-i>m_1$ and \eqref{eq:critical-lbd-2} holds for all pairs $(i',j')$ with $j'<j$ and for $(i',j)$ with $i'>i$. It remains to show that \eqref{eq:critical-lbd-2} holds for the pair $(i,j)$ under this assumption.
	As for $\Ec_1$, taking $\eps=1/2$ in \eqref{eq:b2}, we have
	$$\Pb( \Ec_1 )\le \Pb( \Ac^{*,+}_{D}(2^{i+4},2^{j-3})\cap \{ \Nc\le 4 \} ) \le \sigma_1 \, (j-i)^4\, 2^{-3.5(j-i)}\le \sigma_1\, 2^{-3(j-i)},$$
	where we used \eqref{eq:m_1} in the third inequality.
	
	By \eqref{eq:lbd+}, 
	\begin{align*}
		\Pb(\Ec_2) \le \sum_{k=i+4}^{j-4} \sum_{u \in A_{2^k,2^{k+1}}} \sum_{v: v \sim^* u} c_8\, \rho^{\lambda}\,  \Pb( \Ac^{\lambda,+}_{\loc}(2^i,2^{k-3}) ) \,\Pb( \Ac^{\lambda,+}_D(2^{k+2},2^j) )\, \Pb( \Ac^{\lambda}_{\loc}(u;8,2^{k-2}) ).
	\end{align*} 
	We then use the induction hypothesis twice, for the first two boundary $\lambda$-arm events, and we use Proposition~\ref{prop:critical-lbd-1} for the third $\lambda$-arm event (we let $n_0$ and $b_0$ be as in the statement of that result). This gives, for $N \geq n_0$,
	\begin{align*}
		\Pb(\Ec_2) &\le \sum_{k=i+4}^{j-4} \sum_{u \in A_{2^k,2^{k+1}}} \sum_{v: v \sim^* u} c_8\, \rho^{\lambda}\, \sigma_2 \cdot b_1\, 2^{-3(k-i)} \cdot b_1\, 2^{-3(j-k)} \cdot b_0\, k^4\, 2^{-2k} \\
		&\le \big(c_8 \, \rho^{\lambda} \, \sigma_2 \, b_0\, b_1 \, \sigma_3\, j^5 \big)\, b_1 \, 2^{-3(j-i)}.
	\end{align*} 
	Note that $\rho^{\lambda} \, j^5 \le (\log 2)^{-5} (\log N)^{-9}$, since $2^j < N$. 
    Thus, to make the induction work, it suffices to choose $b_1$ large enough so that $b_1 \geq 2^{3m_1}$ and $\sigma_1 \le \frac{b_1}{2}$, and then, for such a $b_1$, we choose $n_1$ large enough so that $n_1 \ge n_0$ and $c_8\, \sigma_2 \, b_0\, b_1\, \sigma_3 \, (\log 2)^{-5} (\log n_1)^{-9} \le \frac{1}{2}$. This completes the proof.
\end{proof}

Similarly to Proposition~\ref{prop:bbmae}, we can obtain an estimate for boundary $\lambda$-arm events with an arbitrary center $z$ in $\wb\Hb$.
\begin{proposition}\label{prop:critical-lbd-3}
	There exist $n_2>0$ and $b_2>0$ such that for all $N \ge n_2$, $z\in\wb\Hb$, $0<d_1<d_2<N$, and $D \supseteq B_{2d_2}(z)$,
	$$\Pb\big(\Ac^{\lambda}_{D^+}(z;d_1,d_2) \big)\le  b_2\, \Big(\log \Big( \frac{d_2}{d_1} \Big) \Big)^4\, \Big(\frac{d_2}{d_1}\Big)^{-2} \quad \text{for $\lambda\ge \frac{14}{\delta} \log\log N$}.$$
\end{proposition}

\begin{proof}
	First, we take $n_2>n_0\vee n_1$ and $b_2>b_0\vee b_1$, where $n_0$, $b_0$ are from Proposition~\ref{prop:critical-lbd-1}, and $n_1$, $b_1$ are from Proposition~\ref{prop:critical-lbd-2}. Assume that $z=(0,r)$, with $r\ge 0$.
	
We can distinguish cases as in the proof of Proposition~\ref{prop:bbmae}, and in the first two cases, we can proceed in the same way as in that proof. Hence, we only consider the third case $4d_1<r \le d_2/6$. From \eqref{eq:par-bdy}, we have
	\begin{equation} \label{eq:prod_prob_height}
		\Pb(\Ac_{D^+}^{\lambda}(z;d_1,d_2)) \le C\,  \Pb(\Ac^{\lambda}_{\loc}(z;d_1,r/2))\,  \Pb( \Ac^{\lambda,+}_D(3r,d_2)).
	\end{equation}
	By applying Proposition~\ref{prop:critical-lbd-1} and Proposition~\ref{prop:critical-lbd-2}, respectively to the first factor and the second factor in the right-hand side of \eqref{eq:prod_prob_height}, we obtain
	$$\Pb(\Ac_{D^+}^{\lambda}(z;d_1,d_2)) \le C \cdot b_0\, \Big(\log \Big( \frac{r}{d_1} \Big) \Big)^4 \,r^{-2} \cdot b_1\, \Big(\frac{d_2}{r}\Big)^{-3} \le (C\, b_0\,b_1)\, \Big(\log \Big( \frac{d_2}{d_1} \Big) \Big)^4\, d_2^{-2}.$$
	Hence, it suffices to further require $b_2 > C\, b_0\,b_1$ to complete the proof.
\end{proof}

Using Proposition~\ref{prop:critical 2_arm_proba_bdy} and Proposition~\ref{prop:critical-lbd-3}, we are now able to get the following result for the boundary $\lambda$-two-arm event at the critical intensity $\alpha=1/2$. Since its proof is similar to that of Proposition~\ref{prop:critical-lbd-2}, with the aid of the decoupling result in \eqref{eq:blbd+}, we omit the proof.

\begin{proposition}\label{prop:critical-b2}
	There exists $b_3>0$ such that for all $z=(x,0)$, $N \ge n_2$, $0<d_1<d_2<N$, and $D \supseteq B_{2d_2}(z)$,
	$$\Pb\big(\Bc^{\lambda}_{D}(z;d_1,d_2) \big)\le  b_3\, \Big(\log \Big( \frac{d_2}{d_1} \Big) \Big)^4\, \Big(\frac{d_2}{d_1}\Big)^{-1} \quad \text{for $\lambda\ge \frac{14}{\delta} \log\log N$}.$$
\end{proposition}

We are finally in a position to prove \eqref{eq:critical_carpet1}, and thus complete the proof of Theorem~\ref{thm:carpet2}.

\begin{proof}[Proof of  \eqref{eq:critical_carpet1}]
	First, we have the following decoupling result (similar to \eqref{eq:decouple}): for $z=(x,0)$ and $D=B_{2N}(z)$,
	\begin{equation*}
		\Pb( \Bc^\lambda_{D}(z;1,N)\cap \{ X_{D}(z)>\lambda \} ) \lesssim \rho^{\lambda}\,\Pb( \Bc^\lambda_{D}(z;8,N) ).
	\end{equation*}
	Analogously to \eqref{eq:cr-1} (but now incorporating the occupation at the starting point of the $\lambda$-chain), we have 
	\[
	\Pb(\Fc_N(\lambda)^c) \le \sum_{|x|\le N} \Pb( \Bc^\lambda_{D}(z;1,N)\cap \{ X_{D}(z)>\lambda \} )
	\lesssim N\, \rho^{\lambda} \,\Pb( \Bc^\lambda_{D}(z;8,N) ).
	\]
	By Proposition~\ref{prop:critical-b2}, we obtain that for $a \geq \frac{14}{\delta}$ and $\lambda=a\log\log N$, with $N\ge n_2$,
	\[
	\Pb(\Fc_N(\lambda)^c) \lesssim N\, (\log N)^{-\delta a}\, (\log N)^{4}\, N^{-1} = (\log N)^{4-\delta a} \le (\log N)^{-\delta a/2}.
	\]
	This completes the proof of \eqref{eq:critical_carpet1}.
\end{proof}

\subsection{TSLS of GFF and thick points of random walks}\label{subsec:toGFF}

We now explain how the percolation results at the critical intensity $\alpha=1/2$ can be formulated in terms of other processes, especially the GFF. This will be an immediate consequence of the isomorphism theorems, Theorems~\ref{thm:Le-Jan} and~\ref{thm:gsrk}, recalled in Section~\ref{subsec:GFF}.

Let $\varphi_N:=\varphi_{B_N}$ be the GFF on $B_N$. 
Also, let $X_N$ be the occupation field for the RWLS in $B_N$ with intensity $1/2$.

\begin{proof}[Proof of Theorem~\ref{thm:GFF_carpet}]
	We know from Theorem~\ref{thm:Le-Jan} that $\varphi_N$ and $X_N$ can be coupled in such a way that for all $z\in B_N$, $X_N(z)=\frac12 \varphi_N(z)^2$ (almost surely). 
	Theorem~\ref{thm:GFF_carpet} thus follows readily from Theorem~\ref{thm:carpet2}.
\end{proof}



\begin{remark}\label{rmk:yg}
	For $a>0$,  let $L_N(a)$ be the length of the shortest horizontal crossing of $B_{N/2}$ in $\{z\in B_N: |\varphi_N(z)|\le \sqrt{2a \log\log N} \}$ if there exists one (i.e.\ the event appearing in \eqref{eq:gff-1} occurs with $C = \sqrt{2a}$), and let $L_N(a) = \infty$ otherwise. The quantity $L_N(a)$ is also called the \emph{chemical distance}.
	It has been proved in \cite[Proposition~4.4]{MR4419197} that there exists $a_1>0$ such that for each $a\le a_1$, we have: for some $b(a) \in (0,1)$,
	\begin{equation} \label{eq:chem}
		\lim_{N \to \infty}\Pb( L_N(a)\ge N e^{(\log N)^b} ) = 1.
	\end{equation}
	Under the hypothesis that $L_N(a)<\infty$ for some $a\le a_1$ (with positive probability uniformly in $N$), \eqref{eq:chem} would give rise to an interesting regime where low crossings exist and have a length which is superlinear in $N$.
	This would in particular follow if \eqref{eq:critical_carpet2} holds for some $a \le a_1$, but we do not know whether it is true. Indeed, it is not clear how the above-mentioned constant $a_1$ and the constant $a_0$ in Theorem~\ref{thm:carpet2} compare. This is in fact part of the question (Q\ref{q1}) raised earlier.
\end{remark}

Furthermore, using the generalized second Ray-Knight theorem (Theorem~\ref{thm:gsrk}), we can also obtain results about the connectivity of the level sets of $\Lfr^t_N := \Lfr^t_{B_N}$, defined in \eqref{eq:local-time} (the local times induced by a continuous-time simple random walk). By combining Theorem~\ref{thm:GFF} and Corollary~\ref{cor:gsrk}, we can obtain immediately the following.
\begin{theorem}\label{thm:lt}
	Let $c_0,C_0$ be as in Theorem~\ref{thm:GFF}. Then for all $C\ge C_0$ and $t > 0$, we have
	\begin{align}\label{eq:lt-2}
		& \Pb\Big( \Cc_{N}\Big( \Big\{ z\in B_N: \Lfr^t_N(z)\le \frac12\big(C\sqrt{\log\log N}+\sqrt{2t}\big)^2 \Big\} \Big) \Big)=1-O((\log N)^{-c_0C^2}) \quad \text{as} \quad N\to \infty.
	\end{align}
	Moreover, for all $C\ge C_0$ and $t > 0$, 
	\begin{align}
		\label{eq:lt-1}
		&\liminf_{N \to \infty} \Pb\Big( \Cc_{N/2}\Big( \Big\{ z\in B_N: \Lfr^t_N(z)\le \frac12\big(C\sqrt{\log\log N}+\sqrt{2t}\big)^2 \Big\} \Big) \Big)>0.
	\end{align}
\end{theorem}
Finally, we explain why Theorem~\ref{thm:lt} implies Theorem~\ref{thm:thick}.
\begin{proof}[Proof of Theorem~\ref{thm:thick}]
	Recall that $t_N$ was defined in \eqref{eq:tN}. Note that
	$$\frac{1}{2} \big(C\sqrt{\log\log N}+\sqrt{2t_N} \big)^2 = \frac{(\sqrt{\theta}+a_{C, N})^2}{\pi}(\log N)^2,$$
	with
	$$a_{C, N}:=C\, \sqrt{\frac{\pi}{2}}\, \frac{\sqrt{\log\log N}}{\log N}.$$
	Hence, the set $B_N\setminus \Tc_N^+(\theta,a_{C, N})$ consists of the vertices where $\Lfr^{t_N}_N < \frac12\big(C\sqrt{\log\log N}+\sqrt{2t_N}\big)^2$, so (from \eqref{eq:lt-2} and \eqref{eq:lt-1}, respectively)
	\begin{align*}
		\Pb\big( \Cc_{N}\big( B_N\setminus \Tc_N^+(\theta,a_{C, N}) \big) \big)=1-O((\log N)^{-c_0C^2}) \quad \text{and} \quad \liminf_{N \to \infty} \Pb\big( \Cc_{N/2}\big( B_N\setminus \Tc_N^+(\theta,a_{C, N}) \big) \big)>0.
	\end{align*}
	We can then complete the proof by using duality.
\end{proof}

\begin{remark} \label{rem:*thick}
From the way in which we conclude the proof, i.e.\ invoking duality, we can see that in fact, a stronger result holds. Indeed, the properties \eqref{eq:thick-2} and \eqref{eq:thick-1} remain true if in the set of thick points $\Tc_N^+(\theta,a_{C,N})$, we consider $*$-paths (instead of simply paths on $\Zb^2$).
\end{remark}

\section{Extensions} \label{sec:extensions}

In this last section, we consider further estimates on crossing probabilities, and some generalizations. First, we derive a Russo-Seymour-Welsh type lower bound on crossing probabilities in Section~\ref{sec:RSW}, based on a general result obtained recently by K\"ohler-Schindler and Tassion \cite{KST2023}. In Section~\ref{subsec:annulus}, we then study crossings of annuli, which allows us to obtain an exponential decay property throughout the whole subcritical regime. Finally, in another direction, we explain in Section~\ref{sec:metric_graph} how our reasonings can also be applied on the metric graph (cable system) associated with $\Zb^2$.

\subsection{A Russo-Seymour-Welsh result} \label{sec:RSW}

In this section, we prove that the (strongly correlated) percolation model in this paper has a Russo-Seymour-Welsh (RSW) type property. We consider an arbitrary intensity $\alpha > 0$, and use the same notations as in previous sections. Hence, for every subset $D \subseteq \Zb^2$, we denote by $\Lc_D$ the RWLS with intensity $\alpha$ in $D$, and by $X_D$ the associated occupation field (see Section~\ref{subsec:RWLS}). For any threshold $\lambda>0$, we consider the configuration $\omega = (\omega(z))_{z \in D} \in \{0,1\}^{D}$, with $\omega(z) :=\mathbbm{1}_{X_D(z)\le \lambda}$: each vertex $z \in D$ is called open if $\omega(z)=1$, and closed otherwise. In this way, we have defined a site percolation process in $D$, whose distribution is denoted by $\Pf_{\alpha,D}^{\lambda}$. In the remainder of this section, we omit the dependence on $\alpha$ and $\lambda$, simply writing $\Pf_D$, and we abbreviate $\Pf_N:=\Pf_{B_N}$.

We consider crossing events of the form 
$$\Cc(m,n) := \left\{  \begin{array}{c} \text{there exists a path crossing from left to right}\\ \text{in $([-m,m]\times [-n,n]) \cap \Zb^2$ which consists of open sites} \end{array} \right\}, \quad m,n\ge 1.$$
We show later that $\Pf_N$ satisfies the conditions of Theorem~2 in \cite{KST2023} (see Lemmas~\ref{lem:SSD} and~\ref{lem:FKG} below), which will imply, thanks to that result, the following RSW estimate for the finite-volume measure $\Pf_N$.

\begin{theorem}\label{thm:RSW}
	There exists a homeomorphism $\psi: [0,1]\rightarrow [0,1]$ such that for all $\alpha>0$, $\lambda>0$, and $N\ge 1$,
	$$\Pf_{2N}( \Cc(2N,N) )\ge \psi \left( \Pf_{6N}( \Cc(N,2N) ) \right).$$
\end{theorem}

More specifically, we need to check that $\Pf_N$ possesses two properties: (1) it is invariant under the symmetries of $\Zb^2$; (2) it is positively associated (i.e., it satisfies the FKG inequality). In fact, the authors of \cite{KST2023} consider both the infinite-volume and the finite-volume setups, but here we restrict to the latter case, discussed in Section~6 of that paper, since $\Pf_{\Zb^2}$ is degenerate anyway. In addition, note that even though the results in \cite{KST2023} are formulated in terms of bond percolation, their proofs are quite robust (see the explanation just below Theorem~1 in that paper). In particular, they can be extended to site percolation, which is our focus here.

In order to state and check the above-mentioned two conditions, we first need to introduce some additional notations. We denote by $\Sigma$ the group of symmetries of $\Zb^2$, which is generated by the translation by the vector $(1,0) \in\Zb^2$, the rotation by an angle $\pi/2$ around the origin, and the reflection across the $y$-axis. We write, for each $m \geq 1$,
$$\Sigma_m := \big\{ \sigma\in\Sigma: \sigma\cdot B_{m} \subseteq B_{2m} \big\}.$$
An event $E$ (for the percolation configuration) is called \emph{increasing} if its indicator function $\mathbbm{1}_E$ is increasing (similarly to the definition above Lemma~\ref{lem:FKG-RWLS} for loop configurations), i.e.\ if $\omega \in E$ and $\omega \le \omega'$ imply $\omega'\in E$. We are now in a position to state the two conditions, in the finite-volume setup, as the following two lemmas, respectively.

\begin{lemma}[Symmetric stochastic domination] \label{lem:SSD}
	For all $N \geq 1$, $\sigma\in \Sigma_N$, and any increasing event $E$ depending only on the sites in $B_N$, we have
	$$\Pf_N(E) \ge \Pf_{2N}(\sigma \cdot E).$$
\end{lemma}

\begin{lemma}[Positive association] \label{lem:FKG}
	For all $N \geq 1$, and any two increasing events $E$ and $F$ depending only on the vertices in $B_N$, we have
	$$\Pf_N(E \cap F) \ge \Pf_N(E) \cdot \Pf_N(F).$$
\end{lemma}

As we mentioned above, Theorem~\ref{thm:RSW} follows immediately from Theorem~2 of \cite{KST2023}, combined with Lemmas~\ref{lem:SSD} and \ref{lem:FKG}. Hence, there only remains to prove these two lemmas.

\begin{proof}[Proof of Lemma~\ref{lem:SSD}]
	We first observe that the family of probability measures $\Pf_D$, $D\subseteq \Zb^2$, is ordered in the following sense: for any two subsets $D \subseteq D' \subseteq \Zb^2$, for any increasing event depending only on the vertices in $D$,
	$$\Pf_{D}(E) \ge \Pf_{D'}(E).$$
	Indeed, this follows immediately from the natural coupling of $\Lc_D$ and $\Lc_{D'}$. Hence, for any increasing $E$ depending only on the vertices in $B_N$, and every $\sigma\in \Sigma_N$, we have (since $\sigma\cdot B_N \subseteq B_{2N}$ by definition)
	$$\Pf_N(E) = \Pf_{\sigma\cdot B_N}(\sigma \cdot E) \ge  \Pf_{2N}(\sigma \cdot E),$$
	which completes the proof.
\end{proof}

Next, we turn to the proof of Lemma~\ref{lem:FKG}. As we explain, it follows rather directly from Lemma~\ref{lem:FKG-RWLS} and the FKG inequality for product measures.

\begin{proof}[Proof of Lemma~\ref{lem:FKG}]
	We recall the construction of $X_D$ in Section~\ref{subsec:RWLS}, and the families $\bar n:= (n(z))_{z\in D}$ and $\bar t := (t(z))_{z\in D}$ introduced there. It is easy to see that an event $E$ which is increasing for the percolation configuration in $B_N$ is necessarily decreasing in both $\bar n$ and $\bar t$. Conditionally on $\bar n$, we know that the distribution of $\bar t$ is simply the product measure under which for each $z \in D$, $t(z) \sim \Gamma(n(z) + \alpha,1)$. In other words, given $\bar n$, we can write 
	$$\Pf_N \mid _{\bar n}\ =\prod_{z\in D} \mu_z,$$
	where $\mu_z$ is the Bernoulli distribution on $\{0,1\}$, where $\mu_z(1)$ is given by the cumulative distribution function at $\lambda$ of the $\Gamma(n(z) + \alpha,1)$ distribution, and $\mu_z(0)=1-\mu_z(1)$. Thus, by the standard lattice FKG inequality for product measures (see e.g. Section~2.2 of the classical reference \cite{MR1707339}), we have 
	$$\Pf_N(E \cap F \mid \bar n) \ge \Pf_N(E \mid \bar n)\cdot \Pf_N(F \mid \bar n).$$
	By applying Lemma~\ref{lem:FKG-RWLS} with $f=\Pf_N(E \mid \bar n)$ and $g=\Pf_N(F \mid \bar n)$, which are both decreasing, we then obtain that 
	$$\Eb(\Pf_N(E \mid \bar n)\cdot \Pf_N(F \mid \bar n)) \ge \Eb(\Pf_N(E \mid \bar n)) \cdot \Eb(\Pf_N(F \mid \bar n))= \Pf_N(E) \cdot \Pf_N(F).$$
	This completes the proof. 
\end{proof}

We have focused on primal crossings so far, but similar results hold for dual crossing events, of the form 
$$\Cc^*(m,n) := \left\{  \begin{array}{c} \text{there exists a $*$-path crossing from left to right}\\ \text{in $([-m,m]\times [-n,n]) \cap \Zb^2$ which consists of closed sites} \end{array} \right\}, \quad m,n\ge 1.$$
In particular, the conclusion of Theorem~\ref{thm:RSW} remains valid with $\Cc^*(m,n)$ in place of $\Cc(m,n)$, as we now explain. First, we have by duality: for all $N \geq n \vee m$,
$$\Pf_N(\Cc(m,n))=1-\Pf_N(\Cc^*(n,m)).$$
Therefore, there exists a homeomorphism $\psi^*: [0,1]\rightarrow [0,1]$ such that for all $\alpha>0$, $\lambda>0$, and $N\ge 1$,
\begin{equation} \label{eq:dual-RSW}
	\Pf_{6N}( \Cc^*(2N,N) )\ge \psi^* \left( \Pf_{2N}( \Cc^*(N,2N) ) \right). 
\end{equation}
Moreover, since $\Cc^*(2N,N)$ is a decreasing event, we have that for any finite $D \supseteq B_{6N}$,
$$\Pf_{D}( \Cc^*(2N,N) )\ge \Pf_{6N}( \Cc^*(2N,N) ).$$
By Lemma~1 of \cite{KST2023} (based on standard path-gluing techniques and the FKG inequality), we obtain the following. For every integer $\rho \ge 2$, we have
$$\Pf_{3\rho N}( \Cc^*(\rho N,N) )\ge \big( \Pf_{6N}( \Cc^*(2N,N) ) \big)^{2\rho-1} \ge \big( \psi^* \big( \Pf_{2N}( \Cc^*(N,2N) ) \big) \big)^{2\rho-1},$$
where \eqref{eq:dual-RSW} was used in the second inequality. In other words, we have obtained a general form of RSW for dual crossing events.

\subsection{Annulus setting and exponential decay}\label{subsec:annulus}

In this section, we consider annulus-crossing events. As we will see immediately, the critical values defined in this setup are the same as $\lambda'_*(\alpha)$ (see \eqref{eq:crit_val2}), which is associated with the bulk box-crossing event. Moreover, we can derive in this way an exponential decay property for the one-arm probability in the subcritical regime.

First, the boundary and bulk annulus-crossing events are defined, respectively, by
\begin{equation}
	\overline\Ec_{\alpha,N}(\lambda):=\{ \text{there exists a $\lambda$-open path in $B_N$ that crosses $A_{N/2,N}$} \},
\end{equation}
and 
\begin{equation}
	\overline\Ec'_{\alpha,N}(\lambda):=\{ \text{there exists a $\lambda$-open path in $B_N$ that crosses $A_{N/4,N/2}$} \}.
\end{equation}
The associated critical values are then, respectively,
\begin{equation}\label{eq:bar-lambda}
	\overline\lambda_*(\alpha) := \inf \left\{ \lambda\ge 0: \liminf_{N \to \infty} \Pb( \overline\Ec_{\alpha,N}(\lambda) )>0 \right\},
\end{equation}
and 
\begin{equation}\label{eq:inside-bar-lambda}
	\overline\lambda'_*(\alpha) := \inf \left\{ \lambda\ge 0: \liminf_{N \to \infty} \Pb( \overline\Ec'_{\alpha,N}(\lambda) )>0 \right\}.
\end{equation}

Using some geometric observations, together with the RSW estimate from Theorem~\ref{thm:RSW}, we can show that the above two critical values are equal, and that they in fact coincide with the value $\lambda'_*(\alpha)$ introduced earlier.

\begin{theorem} \label{thm:equal}
	For all $\alpha\in (0,1/2]$, we have $\overline\lambda_*(\alpha)=\overline\lambda'_*(\alpha)=\lambda'_*(\alpha)$.
\end{theorem}

\begin{proof}
	Since the occupation field is increasing in the loop soup, the following monotonicity holds true. For each $\lambda > 0$, $\overline\Ec'_{\alpha,N}(\lambda) \subseteq \overline\Ec_{\alpha,N/2}(\lambda)$ for all $N \geq 1$, so 
	\begin{equation} \label{eq:ineq_Ec_Ec'1}
		\Pb( \overline\Ec'_{\alpha,N}(\lambda) ) \le \Pb( \overline\Ec_{\alpha,N/2}(\lambda) ).
	\end{equation}
	Conversely, we consider a ``circuit'' of $36$ non-overlapping boxes of side length $N/8$ around $B_{N/2}$ that, together, disconnect that box from infinity. Since any open path crossing $A_{N/2,N}$ has to intersect one of these smaller boxes, and thus cross an annulus with radii $N/16$ and $N/8$ with the same center as this box, we deduce that
	\begin{equation} \label{eq:ineq_Ec_Ec'2}
		\Pb( \overline\Ec_{\alpha,N}(\lambda) ) \le 36\, \Pb( \overline\Ec'_{\alpha,N/4}(\lambda) ).
	\end{equation}
	From \eqref{eq:ineq_Ec_Ec'1} and \eqref{eq:ineq_Ec_Ec'2}, we can conclude that $\overline\lambda_*(\alpha)=\overline\lambda'_*(\alpha)$.
	
	Next, we compare $\overline\lambda'_*(\alpha)$ with $\lambda'_*(\alpha)$. For this, we consider $8$ boxes with side length $N/8$ which are contained in $B_{N/2}$ and lie along the left side of that box. This implies readily 
	\begin{equation} \label{eq:ineq_Ec_Ec'3}
		\Pb( \Ec'_{\alpha,N}(\lambda) ) \le 8\, \Pb( \overline\Ec'_{\alpha,N/4}(\lambda) ).
	\end{equation}
	In the other direction, let $R_1 := [-4N,-2N]\times [-4N,4N]$, which has the same left side as $B_{4N}$, and let $R_2$, $R_3$ and $R_4$ be the rectangles obtained by rotating $R_1$ around the origin by an angle, respectively, $\pi/2$, $\pi$ and $3 \pi/2$. Then, any path in $B_{8N}$ across the annulus $A_{2N,4N}$ has to cross at least one of these $4$ rectangles in the short direction, which implies that
	\begin{equation} \label{eq:ineq_Ec_Ec'4}
		\Pb( \overline\Ec'_{\alpha,8N}(\lambda) )\le 4\, \Pf_{5N}( \Cc(N,4N) ),
	\end{equation}
	where we use the same notation as in Section~\ref{sec:RSW}. By Theorem~\ref{thm:RSW} (the aspect ratio used there is not essential), we have 
	\begin{equation} \label{eq:ineq_Ec_Ec'5}
		\psi'( \Pf_{5N}( \Cc(N,4N) ) ) \le \Pf_{4N}( \Cc(4N,N) ), 
	\end{equation}
	where $\psi'$ is a homeomorphism from $[0,1]$ to $[0,1]$. Finally, we note that
	\begin{equation} \label{eq:ineq_Ec_Ec'6}	
		\Pf_{4N}( \Cc(4N,N) )\le \Pb( \Ec'_{\alpha,2N}(\lambda) ).
	\end{equation}
	We can thus complete the proof of $\overline\lambda'_*(\alpha)=\lambda'_*(\alpha)$ by combining \eqref{eq:ineq_Ec_Ec'3}, \eqref{eq:ineq_Ec_Ec'4}, \eqref{eq:ineq_Ec_Ec'5}, and \eqref{eq:ineq_Ec_Ec'6}.
\end{proof}

\begin{remark}
	In fact, by using the RSW estimate in Theorem~\ref{thm:RSW}, one can easily show that furthermore, for each $\alpha\in (0,1/2]$, the three critical values considered in Theorem~\ref{thm:equal} do not depend on the aspect ratios in their definition.
\end{remark}

We are finally in a position to prove Theorem~\ref{thm:exp_decay}.

\begin{proof}[Proof of Theorem~\ref{thm:exp_decay}]
Consider any $\lambda<\lambda'_*(\alpha)$ and $\eps>0$. From Theorem~\ref{thm:equal}, we also have $\lambda<\overline\lambda'_*(\alpha)$, so it follows immediately from the definition \eqref{eq:inside-bar-lambda} that there exists $m$ large enough such that
	\begin{equation}\label{eq:de-per}
		\Pb( \overline\Ec'_{\alpha,m}(\lambda) ) < \eps.
	\end{equation}
	Now, we consider some $10m < n \le N$, and we tile $B_N$ with boxes which are translations of $B_{m/4}$. If $T$ is such a tile in $B_N$, we declare $T$ to be open if there exists a $\lambda$-open path in $T''$ that crosses the annulus $T'\setminus T$, where $T$, $T'$, $T''$ are concentric boxes such that $T''\setminus T'$ and $T'\setminus T$ are translations of $A_{m/2,m}$ and $A_{m/4,m/2}$, respectively, and the property of being $\lambda$-open is with respect to the field $\Lc_{T''}$. Then, $T_i$ is open with probability smaller than $\eps$ by \eqref{eq:de-per}, and the random variables $\ind_{\{\text{$T_i$ is open}\}}$ and $\ind_{\{\text{$T_j$ is open}\}}$ are independent as soon as $T''_i\cap T''_j=\emptyset$. Hence, we have produced a $k$-dependent translation-invariant percolation process on the tiling, for some suitable finite $k$.
	
	Now, we can note that every $\lambda$-open path as in \eqref{eq:arm_def}, connecting $0$ to $\partial B_n$, produces a ``path'' of open tiles. Moreover, if a $\lambda$-open path in $B_N$ crosses the annulus $T'\setminus T$, then such a crossing of $T'\setminus T$ also exists in $T''$. We can thus get \eqref{eq:arm} immediately by choosing $\eps > 0$ sufficiently small in \eqref{eq:de-per}, using the classical stochastic domination of $k$-dependent percolation by Bernoulli percolation established in \cite{LSS1997}.
\end{proof}

\subsection{RWLS on the metric graph} \label{sec:metric_graph}

In this final section, we discuss another setting that has been studied in the literature, namely the loop soup on the \emph{metric graph} of $\Zb^2$, obtained by considering the edges as continuous segments (a precise definition is given below). Our main goal is to explain why results in the discrete case can be translated to this situation, with relatively minor adaptations to the proofs.

First, we briefly review the definitions of various objects on the metric graph, and we refer the reader to Section~2 in \cite{MR3502602} for more details. In order to distinguish this setting from the discrete one considered so far in this paper, we add a $\sim$ above each notation for a discrete quantity, to indicate the corresponding metric-graph version. Consider any finite subset $D \subseteq \Zb^2$. In this section, we denote by $\Gc_D:= (D,E_D)$ the corresponding subgraph of $\Zb^2$, where $E_D$ is the set of all edges contained in $D$, consisting of the edges of $\Zb^2$ that connect two vertices in $D$. We then obtain the metric graph (also known as \emph{cable system}) of $D$, denoted by $\wt\Gc_D$, by associating an open interval $I_e$ of length $1/2$ with each edge $e \in E_D$, and identifying the endpoints of $I_e$ with the two vertices adjacent to $e$. That is, we let $\wt\Gc_D := D\cup\bigcup_{e\in E_D} I_e$.

For any $D \subseteq \Zb^2$, recall that $\partial^{\mathrm{out}} D = \partial^{\mathrm{in}}(D^c)$ denotes the outer boundary of $D$, and let $\bar D := D \cup (\partial^{\mathrm{out}} D)$. In order to have a direct comparison to the discrete setting, we need to define the various random processes on the metric graph $\wt\Gc_{\bar D}$, instead of $\wt\Gc_D$. A standard Brownian motion $\wt W$ on $\wt\Gc_{\bar D}$ can be constructed in the natural way, by gluing stopped Brownian motions or Brownian excursions, as in Section 2 of \cite{MR3502602}. Let $\wt\xi:=\inf\{ t\ge 0: \wt W_t \in \partial^{\mathrm{out}} D \}$. The $0$-potential of the process $(\wt W_t)_{0\le t<\wt\xi}$ has a density with respect to the Lebesgue measure on $\wt\Gc_{\bar D}$, which is the Green's function $(\wt G_D(u,v))_{u,v\in \wt\Gc_{\bar D}}$, and $\wt G_D(u,v)$ coincides with $G_D(u,v)$ when $u,v\in D$.

Following the framework of \cite{FR2014} (see also \cite{MR3502602}), $(\wt W_t)_{0\le t<\wt\xi}$ can be used to construct a measure $\wt\nu_{\bar D}$ supported on time-parametrized continuous loops in $\wt\Gc_{\bar D} \setminus \partial^{\mathrm{out}} D$. The (metric-graph) loop soup $\wt\Lc_{\bar D}$ of intensity $\alpha>0$ is the Poisson point process of loops with intensity $\alpha\wt\nu_{\bar D}$. We then let $\wt X_D: \wt\Gc_{D} \longrightarrow [0,\infty)$ be the occupation field associated with $\wt\Lc_{\bar D}$, restricted to $\wt\Gc_{D}$, which is defined formally as
$$\wt X_D(z) := \sum_{\tilde\gamma\in \wt\Lc_{\bar D}} L_{T(\tilde\gamma)}^z(\tilde\gamma), \quad z \in \wt\Gc_{D},$$
where $T(\tilde\gamma)$ is the time duration of $\tilde\gamma$, and $(L_{t}^u(\tilde\gamma))_{0\le t\le T(\tilde\gamma),u\in\wt\Gc_{\bar D}}$ are the space-time continuous local times of $\tilde\gamma$ (with respect to the Lebesgue measure on $\wt\Gc_{\bar D}$). As explained in Section~2 of \cite{MR3502602} (see also Section~7.3 of \cite{FR2014}), the RWLS $\Lc_D$ can be obtained from $\wt\Lc_{\bar D}$ by taking the trace of the latter on $D$, i.e., restricting the continuous loops in $\wt\Gc_{\bar D} \setminus \partial^{\mathrm{out}} D$ to the discrete set $D$. From now on, we will always assume that $\Lc_D$ and $\wt\Lc_{\bar D}$ are naturally coupled on the same probability space through such a restriction. In particular, $\wt X_D(z)$ coincides with $X_D(z)$ for all $z\in D$. 

Analogously to the discrete case, we say that $z\in \wt\Gc_D$ is $\lambda$-open if $\wt X_D(z)\le\lambda$. We write $\wt X_N := \wt X_{B_N}$, and $\wt\Gc_N := \wt\Gc_{B_N}$. We define a $\lambda$-open path for $\wt X_N$ as a continuous curve in $\wt\Gc_{N}$ along which all points are $\lambda$-open (with respect to $\wt X_N$). We can study the same percolation problem for the occupation field $\wt X_N$ as for $X_N$. For this, 
we consider two types of crossing events, completely analogous to \eqref {eq:E-lbd} and \eqref{eq:tE-lbd}, respectively:
\begin{equation}\label{eq:E-lbd-m}
	\wt\Ec_{\alpha,N}(\lambda) :=  \left\{ \text{there exists a $\lambda$-open path for $\wt X_N$ crossing from left to right in $\wt\Gc_N$} \right\},
\end{equation}
and 
\begin{equation}\label{eq:tE-lbd-m}
	\wt\Ec'_{\alpha,N}(\lambda) := \left\{ \text{there exists a $\lambda$-open path for $\wt X_N$ crossing from left to right in $\wt\Gc_{N/2}$} \right\}.
\end{equation}
We then introduce the associated critical parameters, i.e.
\begin{equation}
	\wt\lambda_*(\alpha) := \inf \left\{ \lambda\ge 0 : \liminf_{N \to \infty} \Pb( \wt\Ec_{\alpha,N}(\lambda) )>0 \right\},
\end{equation}
and 
\begin{equation}
	\wt\lambda'_*(\alpha) := \inf \left\{ \lambda\ge 0 : \liminf_{N \to \infty} \Pb( \wt\Ec'_{\alpha,N}(\lambda) )>0 \right\}.
\end{equation}

Now, the corresponding result on the metric graph can be stated as follows.
\begin{theorem}\label{thm:main-m}
	For all $\alpha\in (0,1/2)$, we have $0<\lambda_*(\alpha)\le \wt\lambda_*(\alpha)\le\wt\lambda'_*(\alpha)<\infty$. Moreover, for all $\eps>0$, there exists $\wt\lambda_0(\alpha,\eps)\in (0,\infty)$ such that the following hods. For all $\lambda\ge\wt\lambda_0(\alpha,\eps)$,
	$$\Pb( \wt\Ec_{\alpha,N}(\lambda) )=1-O(N^{1-\xi^{2,+}(\alpha)+\eps}) \quad \text{as } N \to \infty.$$
\end{theorem}

We prove Theorem~\ref{thm:main-m} by comparing crossing events on the metric graph with the corresponding discrete events. Note that duality (i.e., ``blocking'' paths) takes a slightly different form here, as illustrated on Figure~\ref{fig:metric_graph}. We will explain how to circumvent this difficulty.

\begin{figure}[t]
	\centering
	\includegraphics[width=.58\textwidth]{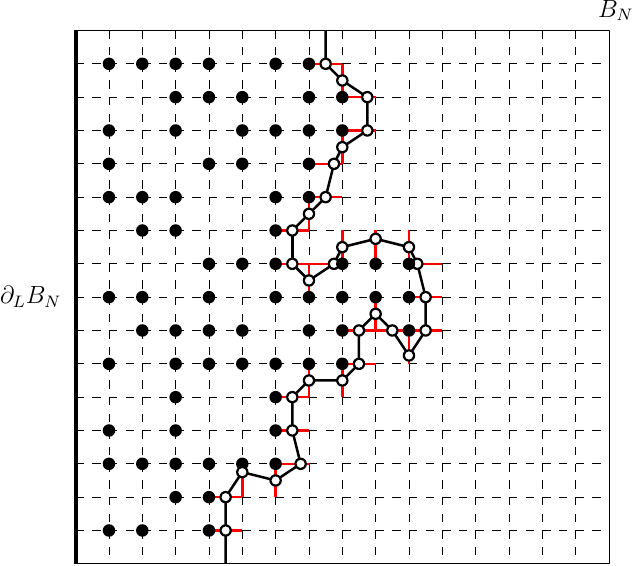}
	\caption{In a given box $B_N$, we can consider the set $\mathcal{C}_L$ of all vertices which are connected to the left side $\partial_L B_N$ by a $\lambda$-open path. By following the edge boundary $\partial_e \mathcal{C}_L$ of this set (or, rather, the edge boundary of its ``filling''), depicted in red, we can see that if there is no horizontal $\lambda$-open crossing in $B_N$, then there must exist a high ``path'' along $\partial_e \mathcal{C}_L$ as shown here, made of $\lambda$-closed vertices.}
	\label{fig:metric_graph}
\end{figure}

In order to perform this comparison, we first need to strengthen our results in the discrete case. For all $z\in B_N$, let $Y_N(z)$ be the ``discrete'' occupation field, that is, $Y_N(z) := n(z)$, with $D=B_N$ (i.e.\ the total number of visits to $z$ by loops, see \eqref{eq:disc-occup}). We further define
$$Z_N(z) := \max_{ w : |w-z| \le 1} X_N(w) \vee Y_N(w).$$
One can easily check that all the estimates derived in the previous sections still hold when the field $X_N$ is replaced by the field $Y_N$, and even by $Z_N$. In particular, the various summation arguments in Section~\ref{sec:pf-thm} remain valid, and also Theorem~\ref{thm:main}.

\begin{theorem} \label{thm:Z}
	The same conclusions as in Theorem~\ref{thm:main} hold true when the field $X_N$ in replaced by $Y_N$, or by $Z_N$.
\end{theorem} 

\begin{remark}
	The above theorem holds for roughly the following reasons. First, for any $k\ge 1$, $w \neq 0$ with $|w|\le k$, $d\ge 2k$, and $D \supseteq B_{2k}$, the analogs of both Lemma~\ref{lem:0w} and Lemma~\ref{lem:Gauss} hold (with constants $c$, $C$ depending only on $k$). And second, the modification of loops in Section~\ref{sec:pf-thm} is very robust, and not sensitive to the occupation at a few points. This means that one could even alternatively define
	$$Z_N(z) :=\max_{w: |w-z| \le k} X_N(w)\vee Y_N(w),$$
	and still get the same result. However, one should note that the associated critical values might depend on $k$.
\end{remark}

Now, we conclude this section by completing the proof of Theorem~\ref{thm:main-m}.
\begin{proof}[Proof of Theorem~\ref{thm:main-m}]
	Obviously, percolation for the metric graph is harder than for the discrete graph, in the sense that $\wt\Ec_{\alpha,N}(\lambda) \subseteq \Ec_{\alpha,N}(\lambda)$ (under the natural coupling explained above). Thus, we get $\wt\lambda_*(\alpha)\ge \lambda_*(\alpha)>0$ directly. The inequality $\wt\lambda'_*(\alpha) \geq \wt\lambda_*(\alpha)$ is also clear by monotonicity (in the same way as $\lambda'_*(\alpha) \geq \lambda_*(\alpha)$ in the proof of Theorem~\ref{thm:main}).

	In the following, we show that for all $\lambda$ large enough,
	$$\liminf_{N \to \infty} \Pb( \wt\Ec'_{\alpha,N}(\lambda) ) >0, \quad \text{and} \quad \Pb( \wt\Ec_{\alpha,N}(\lambda)^c )=O(N^{1-\bar\xi(\alpha)+\eps}) \quad \text{as } N \to \infty.$$
	This is done by employing the same summation argument as in Section~\ref{sec:pf-thm}. For this, we need to modify the definitions a little bit. Let $K$ be a large universal constant, that we will explain how to choose later. We say that a vertex $z\in B_N$ is \emph{$\lambda$-good} if it satisfies two conditions (and \emph{$\lambda$-bad} otherwise): (1) $Z_N(z)\le\lambda$, i.e.\ $X_N(w), Y_N(w) \leq \lambda$ for all $w \sim z$, and for $w=z$; (2) every point that lies in an interval adjacent to $z$ has $\wt X_N$-value smaller than $K \lambda$. Note that the existence of a $\lambda$-good discrete path crossing from left to right in $B_N$ implies that one can find a corresponding $(K \lambda)$-open path (crossing from left to right) for $\wt X_N$ in $\wt\Gc_N$. Therefore, it is enough to bound the probability of existence of a $*$-crossing made by $\lambda$-bad vertices from top to bottom, in a similar way as before.
	
	In this new setting, we say that two clusters $\Cc,\Cc'$ of $\Lc_N$ are $\lambda$-connected if there exists a $*$-passage edge $(u,v)$ of $(\Cc,\Cc')$ such that both $u$ and $v$ are $\lambda$-bad. With this modification, all the other notions introduced in Section~\ref{sec:pf-thm} apply here. The main issue is to obtain a decoupling result as in Lemma~\ref{lem:decom}. For this, let us use the setup in that lemma, and assume that $(u,v)\in\Pc^\lambda$ (using the modified notion of $\lambda$-connectedness). In particular, $u$ is necessarily $\lambda$-bad, which can come from two reasons:
	\begin{enumerate}[(i)]
		\item \label{eq:casei} either $Z_N(u)\ge\lambda$,
		
		\item \label{eq:caseii} or there exists an edge $e = (u,w) \in E$ such that
		$$X_N(u)\vee Y_N(u)\vee X_N(w)\vee Y_N(w)\le\lambda \quad \text{and} \quad \sup_{z\in I_e}\wt X_N(z)>K\lambda.$$
	\end{enumerate}
	From Theorem~\ref{thm:Z}, and the comment prior to it, we know that the estimate \eqref{eq:decom} holds in case~(\ref{eq:casei}). Hence, we only need to further analyze case~(\ref{eq:caseii}), which we do from now on. Let $n(e)$ be the number of crossings of $e$ by loops in $\Lc_N$, which is smaller than $Y_N(u)\wedge Y_N(w)\le\lambda$. From Section~2 in \cite{MR3502602}, we can reconstruct the trace of the loops, as well as the occupation field inside $I_e$, from the knowledge of $n(e)$, $X_N(u)$ and $X_N(w)$, in the following way.
	\begin{itemize}
		\item Sample $n(e)$ independent Brownian excursions from $u$ to $w$ inside $I_e$. The occupation field of each excursion is a squared Bessel bridge of dimension $2$ from $0$ at $u$ to $0$ at $w$, over the time interval $I_e$ with length $1/2$ (see Theorem 4 in \cite{williams1970decomposing}).
		
		\item Sample the Poisson point process of excursions from $u$ to $u$ inside $I_e$, with intensity 
		$$X_N(u) \cdot n_+\, \ind_{\text{height excursion}<1/2},$$
		where $n_+$ is the measure on positive Brownian excursions. By the second Ray-Knight theorem, its occupation field is a squared Bessel process of dimension $0$ started from $X_N(u)$ at $u$, conditioned to hit $0$ before time $1/2$.
		
		\item Similarly, sample excursions from $w$ to $w$ inside $I_e$, in the same way as from $u$ to $u$.
		
		\item Sample the loop soup inside $I_e$. According to Proposition 4.6 of \cite{Lu2018}, its occupation field is a squared Bessel bridge of dimension $2\alpha$ over $I_e$, from $0$ at $u$ to $0$ at $w$.
	\end{itemize}
	
	In what follows, we use the following notation. For $\delta\ge 0$ and $x\ge 0$, we let $(\beta_t^{\delta,x})_{0\le t\le 1/2}$ be a squared Bessel bridge of dimension $\delta$ over $[0,1/2]$, from $x$ to $0$. 
	For $y\in\Rb$, we denote by $(b_t^{y})_{0\le t\le 1/2}$ a standard (linear) Brownian bridge of length $1/2$ from $y$ to $0$.
	We start by recalling two basic properties of squared Bessel bridges (see Section~5 in \cite{pitman1982decomposition}, or Chapter~XI in \cite{RY1999}).
	\begin{itemize}
		\item Reversibility: for all $\delta>0$, $(\beta_t^{\delta,0})_{0\le t\le 1/2}$ and $(\beta_{1/2-t}^{\delta,0})_{0\le t\le 1/2}$ have the same law.
		\item Additivity: if $(\beta_t^{\delta,x})_{0\le t\le 1/2}$ and $(\beta_t^{\delta',x'})_{0\le t\le 1/2}$ are independent, then their sum has the same law as $(\beta_t^{\delta+\delta',x+x'})_{0\le t\le 1/2}$.
	\end{itemize}
	For simplicity, we now identify $I_e$ with $(0,1/2)$, $u$ with $0$, and $w$ with $1/2$. Then, the occupation inside $I_e$ is distributed as the following sum of four independent processes:
	$$\Big(\beta_t^{2n(e),0}+\beta_t^{0,X_N(u)}+\beta_{1/2-t}^{0,X_N(w)}+\beta_t^{2\alpha,0}\Big)_{0\le t\le 1/2}$$
	(where we have used the additivity property for the first term). According to Eq.~(5.8) in \cite{pitman1982decomposition}, $(\beta_t^{1,x})_{0\le t\le 1/2}$ has the same law as $(|b_t^{\sqrt{x}}|^2)_{0\le t\le 1/2}$. From the additivity property, it also has the same law as the independent sum $(\beta_t^{0,x}+\beta_t^{1,0})_{0\le t\le 1/2}$. Hence, for $y>x>0$,
	\begin{equation} \label{eq:bd_Bessel1}
		\Pb\Big( \sup_{t\in (0,1/2)}(\beta_t^{0,x}+\beta_t^{1,0}) > y \Big) = \Pb\Big( \sup_{t\in (0,1/2)}|b_t^{\sqrt{x}}| >\sqrt{y} \Big)\le 2e^{-4\sqrt{y}(\sqrt{y}-\sqrt{x})},
	\end{equation}
	where in the inequality, we used a standard estimate about the tail of the maximum of a Brownian bridge (see e.g. \cite{BS2002}, Chapter IV.26). By the reversibility property, $(\beta_t^{1,0})_{0\le t\le 1/2}$ has the same law as $(\beta_{1/2-t}^{1,0})_{0\le t\le 1/2}$, so we also have
	\begin{equation} \label{eq:bd_Bessel2}
		\Pb\Big( \sup_{t\in (0,1/2)}(\beta_{1/2-t}^{0,x}+\beta_t^{1,0}) > y \Big)\le 2e^{-4\sqrt{y}(\sqrt{y}-\sqrt{x})}.
	\end{equation}
	
	Next, we bound the maximum of $\beta_t^{k,0}$ over $t\in [0,1/2]$, for each integer $k \geq 1$. We denote $M := \sup_{t\in (0,1/2)}|b_t^0|^2$, and we let $M_1,\ldots,M_k$ be $k$ i.i.d. copies of $M$. Since $(\beta_{t}^{k,0})_{0\le t\le 1/2}$ can be interpreted as the sum of $k$ i.i.d. copies of $(|b_t^0|^2)_{0\le t\le 1/2}$ (by additivity), we have immediately 
	$$\Eb\Big(e^{\sup_{t\in (0,1/2)}\beta_t^{k,0}}\Big) \le  \Eb\Big(e^{\sum_{1\le i\le k}M_i}\Big)= C^k,$$
	where by \eqref{eq:bd_Bessel1},
	$$C :=\Eb \big( e^M \big) \le \int_0^\infty 2\,e^{-4x+x}dx<\infty.$$
	Therefore, by Chebyshev's inequality,
	\begin{equation} \label{eq:bd_Bessel3}
		\Pb \bigg( \sup_{t\in (0,1/2)}\beta_t^{k,0} >y \bigg) \le e^{-y} C^k.
	\end{equation}
	We now choose the value
	\begin{equation} \label{eq:valueK}
		K := 12(\log C)\vee 16.
	\end{equation}
	Putting things together, on the event $\{X_N(u)\vee Y_N(u)\vee X_N(w)\vee Y_N(w) \vee n(e) \le \lambda\}$, and conditioned on the values of $X_N(u)$, $X_N(w)$, and $n(e)$, we have (for ease of notation, below we simply write $\tilde{\Pb}$ for the conditional probability measure): for some universal $c>0$,
	\begin{align*}
		\tilde\Pb\Big( \sup_{t\in (0,1/2)} & (\beta_t^{2n(e),0} + \beta_t^{0,X_N(u)}+\beta_{1/2-t}^{0,X_N(w)} + \beta_t^{2\alpha,0}) > K\lambda \Big) \\
		& \le \tilde\Pb\Big( \sup_{t\in (0,1/2)}\beta_t^{2n(e),0} > \frac{K\lambda}{3} \Big) +
		\tilde\Pb\Big( \sup_{t\in (0,1/2)}(\beta_t^{0,X_N(u)} + \beta_t^{1,0}) > \frac{K\lambda}{3} \Big) \\
		& \hspace{2cm} +  \tilde\Pb\Big( \sup_{t\in (0,1/2)}(\beta_{1/2-t}^{0,X_N(w)} + \beta_t^{1,0}) > \frac{K\lambda}{3} \Big),
	\end{align*}
	where we have bounded $\beta_t^{2\alpha,0}$ by the sum of two independent copies of $\beta_t^{1,0}$, by using $2\alpha\leq 2$, additivity and monotonicity. By applying \eqref{eq:bd_Bessel3}, \eqref{eq:bd_Bessel1} and \eqref{eq:bd_Bessel2} to the three terms in the sum above (respectively), and using the value of $K$ (see \eqref{eq:valueK}), we then obtain
	\begin{align*}
		\tilde\Pb\Big( \sup_{t\in (0,1/2)} & (\beta_t^{2n(e),0} + \beta_t^{0,X_N(u)}+\beta_{1/2-t}^{0,X_N(w)} + \beta_t^{2\alpha,0}) > K\lambda \Big) \\
		& \le e^{-K\lambda/3} C^{2n(e)} + 2e^{-\sqrt{K\lambda}(\sqrt{K\lambda}-2\sqrt{X_N(u)})} + 2e^{-\sqrt{K\lambda}(\sqrt{K\lambda}-2\sqrt{X_N(w)})} \\
		& \le e^{-c\lambda},
	\end{align*}
	for some universal constant $c > 0$. Hence, \eqref{eq:decom} also holds in case~(\ref{eq:caseii}).
	
	In other words, whenever a $\lambda$-passage edge $(u,v)\in\Pc^\lambda$ arises, we obtain an extra cost decaying exponentially fast in $\lambda$. Hence, we can obtain all the desired results by following an induction argument which is analogous to those in Section~\ref{sec:pf-thm}. We omit further details, thus concluding the proof of Theorem~\ref{thm:main-m}.
\end{proof}

Finally, the result can be generalized to the critical case on the metric graph in a straightforward way, since the same description of occupation inside edges, in terms of Bessel bridges, holds when $\alpha=1/2$, and also the same exponential upper bound for the maximum occupation. Hence, we can obtain the result below.

\begin{theorem}
	There exist $c'_0>0$ and $C'_0$, such that for all $C\ge C'_0$,
	\begin{align}\label{eq:mcritical2}
		&\Pb( \wt\Ec_{1/2,N}(C\sqrt{\log\log N}) )= 1-O( (\log N)^{-c'_0 C^2}) \quad \text{as} \quad N\to \infty,\\[2mm]
		\label{eq:mcritical1}
		&\liminf_{N \to \infty}\Pb( \wt\Ec'_{1/2,N}(C\sqrt{\log\log N}) )>0.
	\end{align}
\end{theorem}

\subsection*{Acknowledgments}
Part of YG's work was done while working at City University of Hong Kong.
YG and PN are partially supported by a GRF grant from the Research Grants Council of the Hong Kong SAR (project CityU11307320). WQ is partially supported by a GRF grant from the Research Grants Council of the Hong Kong SAR (project CityU11308624).

\bibliographystyle{abbrv}
\bibliography{percolation_GFF2}

\end{document}